\documentclass[11pt]{amsart}

\usepackage{amsmath,amssymb,amsthm,mathtools}
\usepackage[margin=28mm]{geometry}
\usepackage{enumitem}
\usepackage{array,tabularx,booktabs}
\usepackage{tikz}
\usepackage{tikz-cd}
\usetikzlibrary{calc,arrows.meta,positioning,decorations.pathmorphing,
  decorations.markings,patterns,shapes.geometric,backgrounds,fit,intersections}
\usepackage[most]{tcolorbox}
\usepackage{microtype}
\usepackage{hyperref}

\hypersetup{
  colorlinks=true,
  linkcolor=blue!60!black,
  citecolor=blue!60!black,
  urlcolor=blue!60!black,
  pdftitle={Essential-Surface Complexes of Knot Exteriors: Image, Kernel, and Reconstruction},
  pdfauthor={Makoto Ozawa},
  pdfsubject={Rigidity, mapping class groups, and essential-surface complexes of knot exteriors},
  pdfkeywords={knot exterior, two-bridge knot, twist knot, essential-surface complex, mapping class group, rigidity, boundary slope, JSJ decomposition, characteristic submanifold, Kakimizu complex, figure-eight knot}
}

\numberwithin{equation}{section}

\newtheorem{theorem}{Theorem}[section]
\newtheorem{proposition}[theorem]{Proposition}
\newtheorem{lemma}[theorem]{Lemma}
\newtheorem{corollary}[theorem]{Corollary}

\newtheorem{problem}[theorem]{Problem}
\newtheorem{question}[theorem]{Question}

\theoremstyle{definition}
\newtheorem{definition}[theorem]{Definition}
\newtheorem{warning}[theorem]{Warning}

\newtheorem{notation}[theorem]{Notation}
\newtheorem{recipe}[theorem]{Recipe}

\theoremstyle{remark}
\newtheorem{remark}[theorem]{Remark}

\newtheorem{theoremA}{Theorem}

\newcommand{\ES}{\mathcal{ES}}
\newcommand{\CES}{\mathcal{CES}}
\newcommand{\HA}{\mathfrak H}
\newcommand{\Mod}{\operatorname{Mod}}
\newcommand{\Aut}{\operatorname{Aut}}
\newcommand{\Homeo}{\operatorname{Homeo}}
\newcommand{\Int}{\operatorname{int}}

\newcommand{\Char}{\operatorname{Char}}
\newcommand{\Fr}{\operatorname{Fr}}
\newcommand{\supp}{\operatorname{supp}}
\newcommand{\lk}{\operatorname{lk}}
\newcommand{\Sl}{\operatorname{Sl}}
\newcommand{\Out}{\operatorname{Out}}

\newcommand{\Image}{\operatorname{Im}}
\providecommand{\sslash}{\mathbin{/\mkern-6mu/}}
\newcolumntype{Y}{>{\raggedright\arraybackslash}X}

\definecolor{surfA}{RGB}{31,110,168}     % first surface / vertex
\definecolor{surfB}{RGB}{200,80,30}      % second surface / vertex
\definecolor{surfC}{RGB}{40,140,80}      % third surface / vertex
\definecolor{greyish}{RGB}{110,110,110}
\definecolor{lightfill}{RGB}{235,240,246}

\newtcolorbox{keybox}[1][]{
  colback=lightfill, colframe=surfA!70!black, boxrule=0.5pt,
  left=6pt,right=6pt,top=5pt,bottom=5pt, arc=1.5pt, #1}

\newtcolorbox{scopebox}[1][]{
  colback=white, colframe=greyish, boxrule=0.5pt,
  left=6pt,right=6pt,top=5pt,bottom=5pt, arc=1.5pt, #1}

\title[Essential-Surface Complexes of Knot Exteriors]
{Essential-Surface Complexes of Knot Exteriors:\\
Image, Kernel, and Reconstruction}

\author{Makoto Ozawa}
\address{Department of Natural Sciences, Faculty of Arts and Sciences, Komazawa University, Tokyo, Japan}
\email{w3c@komazawa-u.ac.jp}
\date{November 12, 2026}

\begin{document}

\begin{abstract}
The curve complex of a surface is deliberately forgetful, yet in most
cases its automorphism group recovers the mapping class group.  We
study the same question for the essential-surface complex $\ES(E)$ of
a knot exterior $E=E(K)$ --- equivalently Schultens's initial surface
complex $\mathcal S_0(E)$ --- organized around three questions: which
combinatorial symmetries are geometric (the \emph{image}), which
homeomorphisms are invisible (the \emph{kernel}), and how much of $E$
can be read off from the complex (\emph{reconstruction}).

Call a knot \emph{slope-separated} if its exterior contains no closed
and no meridional essential surface and at most one essential surface
for each boundary slope.  For such a knot every orientation-preserving
mapping class acts trivially on $\ES(E)$, and the image of the action
is $\mathbb Z/2$ exactly when the knot is amphichiral and has a
nonzero boundary slope.  For two-bridge knots the problem is solved in a
companion paper, whose results are surveyed here without proofs: a
two-bridge knot is slope-separated if and only if it is a
$(2,q)$-torus knot or a twist knot; the complex of a two-bridge knot is
the disjoint union of the complexes carried by the individual
Hatcher--Thurston branched surfaces; it is finite exactly for the knots
with continued fraction $[a]$, $[a,b]$, $[a,1,b]$, or $[a,1,1,b]$ with
$a,b$ odd; the knot $6_2$ is the first whose complex is not discrete;
and the kernel of the action is determined for every hyperbolic
two-bridge knot, while the automorphism group of the complex decorated
with surface types and boundary slopes is uncountable when the complex
is infinite.  Further results are that torus knots are determined among
all knots by the complex decorated with surface types and boundary
slopes, and that for a connected sum of nonfibered prime knots the
decomposing-annulus twist translates Banks's winding coordinate, so an
infinite cyclic subgroup is visible on $\ES(E)$.

The article is written so that it can also serve as an entry point:
the basic objects are first drawn in classical examples, the elementary
structure theory (flag property, layering by boundary slope, finite but
unbounded dimension, Kakimizu's complexes as full subcomplexes) is
collected with proofs, and the article closes with a three-task
formulation of the rigidity problem, a list of problems with their
current status, and a worked two-bridge recipe.
Section~\ref{subsec:status} lists precisely what is classical and what
is new.
\end{abstract}

\maketitle
\tableofcontents

\section{Introduction}
\label{sec:intro}

\subsection{From curves on a surface to surfaces in a knot exterior}

The curve complex $\mathcal C(S)$ of a surface is a deliberately
forgetful object.  Its vertices are isotopy classes of essential
simple closed curves, and a collection of vertices spans a simplex
when the curves can be made pairwise disjoint.  Nothing else is
recorded: not lengths, not positions, not intersection numbers.
Nevertheless, apart from a short list of low-complexity surfaces, the
Ivanov--Korkmaz--Luo theorem \cite{Ivanov,Korkmaz,Luo} recovers the
extended mapping class group from this bare combinatorics,
\[
  \Aut(\mathcal C(S))\;\cong\;\Mod^{\pm}(S).
\]

Replace ``surface'' by ``knot exterior'' and ``curve'' by ``essential
surface'' and the same question can be asked in dimension three
(Figure~\ref{fig:analogy}).  Let
\[
  E=E(K)=S^3\setminus\Int N(K)
\]
be the exterior of a nontrivial knot $K\subset S^3$.  Its essential
surfaces span a complex $\ES(E)$ in exactly the same way, and every
self-homeomorphism of $E$ permutes its vertices.  How much of $E$, and
how much of its symmetry, survives in this combinatorics?

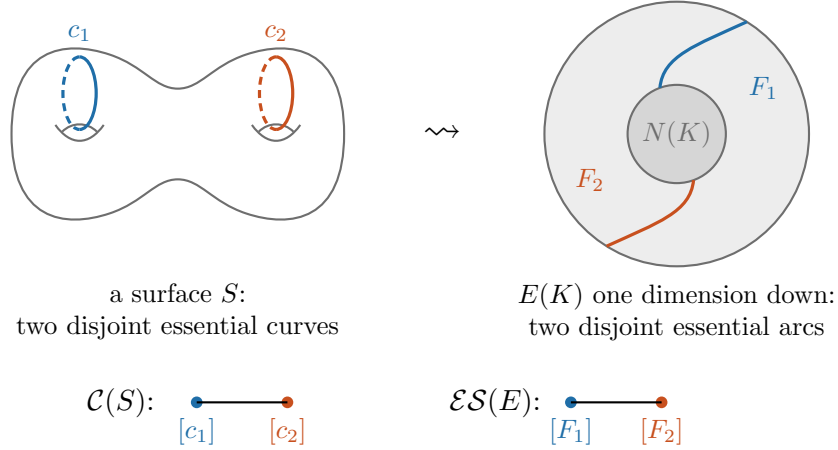
\begin{figure}[htb]
\centering
\begin{tikzpicture}[scale=1.0]
% ================= left: genus-two surface with two disjoint curves =======
\begin{scope}[xshift=0cm]
  % outline of a genus-two surface
  \draw[thick,greyish]
    (-2.2,0)
    .. controls (-2.2,1.15)  and (-1.5,1.25)  .. (-0.9,1.05)
    .. controls (-0.45,0.90) and (-0.3,0.60)  .. (0,0.60)
    .. controls (0.3,0.60)   and (0.45,0.90)  .. (0.9,1.05)
    .. controls (1.5,1.25)   and (2.2,1.15)   .. (2.2,0)
    .. controls (2.2,-1.15)  and (1.5,-1.25)  .. (0.9,-1.05)
    .. controls (0.45,-0.90) and (0.3,-0.60)  .. (0,-0.60)
    .. controls (-0.3,-0.60) and (-0.45,-0.90).. (-0.9,-1.05)
    .. controls (-1.5,-1.25) and (-2.2,-1.15) .. (-2.2,0) -- cycle;
  % the two handles
  \draw[thick,greyish] (-1.62,0.12) .. controls (-1.42,-0.16) and (-1.18,-0.16) .. (-0.98,0.12);
  \draw[thick,greyish] (-1.53,0.03) .. controls (-1.40,0.16)  and (-1.20,0.16)  .. (-1.07,0.03);
  \draw[thick,greyish] ( 0.98,0.12) .. controls ( 1.18,-0.16) and ( 1.42,-0.16) .. ( 1.62,0.12);
  \draw[thick,greyish] ( 1.07,0.03) .. controls ( 1.20,0.16)  and ( 1.40,0.16)  .. ( 1.53,0.03);
  % two disjoint meridians: each encircles the band between the
  % silhouette and the corresponding hole, so each is essential on S
  \draw[surfA,very thick]                (-1.3,1.02) arc (90:-90:0.22 and 0.48);
  \draw[surfA,very thick,densely dashed] (-1.3,1.02) arc (90:270:0.22 and 0.48);
  \draw[surfB,very thick]                ( 1.3,1.02) arc (90:-90:0.22 and 0.48);
  \draw[surfB,very thick,densely dashed] ( 1.3,1.02) arc (90:270:0.22 and 0.48);
  \node[surfA,font=\small] at (-1.3,1.33) {$c_1$};
  \node[surfB,font=\small] at ( 1.3,1.33) {$c_2$};
\end{scope}
% ================= middle arrow ==========================================
\node at (3.5,0) {\Large $\rightsquigarrow$};
% ================= right: the exterior, one dimension down ===============
\begin{scope}[xshift=6.6cm]
  \fill[greyish,opacity=0.12] (0,0) circle (1.75);
  \fill[white] (0,0) circle (0.65);
  \fill[greyish,opacity=0.28] (0,0) circle (0.65);
  \draw[thick,greyish] (0,0) circle (1.75);
  \draw[thick,greyish] (0,0) circle (0.65);
  \node[greyish,font=\small] at (0,0) {$N(K)$};
  % two disjoint essential arcs
  \draw[surfA,very thick]
    (110:0.65) .. controls (100:1.05) and (78:1.20) .. (58:1.75);
  \draw[surfB,very thick]
    (290:0.65) .. controls (280:1.05) and (258:1.20) .. (238:1.75);
  \node[surfA,font=\small] at (28:1.30) {$F_1$};
  \node[surfB,font=\small] at (208:1.30) {$F_2$};
\end{scope}
% ================= aligned sub-captions ==================================
\node[font=\small,align=center] at (0,-2.35)
   {a surface $S$:\\ two disjoint essential curves};
\node[font=\small,align=center] at (6.6,-2.35)
   {$E(K)$ one dimension down:\\ two disjoint essential arcs};
% ================= bottom: the two complexes =============================
\begin{scope}[yshift=-3.55cm]
  \node at (-0.75,0) {$\mathcal C(S)$:};
  \fill[surfA] (0.25,0) circle (2.2pt);
  \fill[surfB] (1.45,0) circle (2.2pt);
  \draw[thick] (0.25,0) -- (1.45,0);
  \node[surfA,below=2pt,font=\small] at (0.25,0) {$[c_1]$};
  \node[surfB,below=2pt,font=\small] at (1.45,0) {$[c_2]$};
  \node at (6.6-2.4,0) {$\ES(E)$:};
  \fill[surfA] (6.6-1.4,0) circle (2.2pt);
  \fill[surfB] (6.6-0.2,0) circle (2.2pt);
  \draw[thick] (6.6-1.4,0) -- (6.6-0.2,0);
  \node[surfA,below=2pt,font=\small] at (6.6-1.4,0) {$[F_1]$};
  \node[surfB,below=2pt,font=\small] at (6.6-0.2,0) {$[F_2]$};
\end{scope}
\end{tikzpicture}
\caption{The analogy.  Left: two disjoint essential simple closed
curves on a surface $S$ span an edge of the curve complex $\mathcal
C(S)$.  Right: the knot exterior $E(K)$, drawn schematically one
dimension down as an annulus around $N(K)$, in which properly embedded
essential surfaces appear as properly embedded essential arcs; two
disjoint such arcs span an edge of the essential-surface complex
$\ES(E)$.  The question of this article is how much of $E$ and of its
symmetry group survives in the right-hand picture.}
\label{fig:analogy}
\end{figure}

Every self-homeomorphism of a nontrivial knot exterior preserves the
meridian slope, by the Gordon--Luecke theorem \cite{GL}.  We therefore
do not carry the meridian through the main notation: write
$\Mod^{\pm}(E)=\pi_0\Homeo(E)$, and let $X$ denote one of the three
objects used below, $\ES(E)$, $\ES^{\mathrm{per}}(E)$, or $\CES(E)$.
Each has a natural action
\begin{equation}
\label{eq:intro-natural-action}
  \Phi_X:\Mod^{\pm}(E)\longrightarrow\Aut(X).
\end{equation}
The meridian reappears only when an actual slope coordinate or an
explicit reconstruction marking is needed.

\begin{keybox}
\textbf{Image.}\ \ Which combinatorial symmetries of $X$ come from
homeomorphisms of $E$?

\medskip
\noindent\textbf{Kernel.}\ \ Which homeomorphisms of $E$ are invisible to
$X$?

\medskip
\noindent\textbf{Reconstruction.}\ \ How much of $E$ itself can be read off
from $X$, and with how little extra information?
\end{keybox}

\noindent
The three questions are logically independent, and in dimension three
they behave quite differently from their two-dimensional models.  The
purpose of this article is to make all three concrete enough that a
reader meeting them for the first time can compute examples, to settle
them completely in the first nontrivial family where this is possible
--- the slope-separated two-bridge knots --- and to record what is
elementary, what is known, and what remains open.

\subsection{What is in this article}

We have tried to keep the entry cost low.  The prerequisites are the
basic theory of incompressible surfaces in $3$-manifolds at the level
of Hempel's or Jaco's books, or Hatcher's notes
\cite{Hempel,HatcherNotes}; Seifert fibered spaces and the JSJ
decomposition are used as black boxes, and are recalled where they are
needed.  Section~\ref{sec:firstlook}
computes $\ES(E)$ by hand in three classical families, with pictures
and without any machinery beyond the classification of surfaces in
Seifert fibered spaces and the Hatcher--Thurston census.  Only after
those examples do we introduce the general definitions.  Readers who
want the general framework immediately may go directly to
Section~\ref{sec:complexes}.

The mathematical content is of four kinds.

\smallskip
\emph{Elementary structure theory} (Sections~\ref{sec:complexes}
and~\ref{sec:visibility}).  Two facts organize every example in the
paper.  The first is that $\ES(E)$ is layered by boundary slope: a
simplex can use only one slope, because essential curves of distinct
slopes on the boundary torus must intersect
(Lemma~\ref{lem:slope-layer}).  The second is that an
orientation-preserving homeomorphism of $E$ fixes every unoriented
slope, whereas an orientation-reversing one reverses the sign of every
slope (Remark~\ref{rem:GL}).  Combining them gives the visibility
theorems below.  In the opposite direction, iterated cables show that
$\ES(E)$ nevertheless has unbounded dimension
(Proposition~\ref{prop:dimension}), so the discreteness seen in the
first examples is an accident of small complexity, not a general
phenomenon.

\smallskip
\emph{The framework} (Sections~\ref{sec:decorations}--\ref{sec:reduction}).
The general rigidity problem is organized around only three named
objects: the bare complex $\ES(E)$, a peripheral refinement recording
surface type and boundary slope, and a characteristic refinement
recording the JSJ/characteristic data.  The general problem is then
separated into intrinsic recognition, geometric realization, and
kernel determination.  Technical realization data and an intentionally
over-marked reconstruction benchmark are kept in the appendices so
that they do not obstruct a first reading.

\smallskip
\emph{Model computations} (Sections~\ref{sec:torus-knots},
\ref{sec:figure-eight}, \ref{sec:composite}, \ref{sec:satellite}).
Torus knots, the figure-eight knot, connected sums, and cables.  The
surface censuses and symmetry groups used here are classical and are
cited where used; what is computed is the resulting image and kernel.

\smallskip
\emph{Two-bridge knots} (Section~\ref{sec:two-bridge}).  Here the
theory goes beyond model computations.  We survey the results of the
companion paper \cite{OzawaTB}: starting from the Hatcher--Thurston
classification it determines the complex of every two-bridge knot, in
particular all slope-separated two-bridge knots, the two-bridge knots
with finitely many essential surfaces, and the first knot whose
complex has an edge.

\subsection{Main results}

The first three statements below are short consequences of classical
results; we call them theorems because they organize the whole
article.  Theorems~\ref{thmD}, \ref{thmE} and~\ref{thmF} are the main
results on two-bridge knots; they are proved in the companion paper
\cite{OzawaTB}.
The status of every statement is described precisely in
Section~\ref{subsec:status}.  The first theorem needs a hypothesis that
recurs throughout, and it is worth isolating it.

\begin{definition}[Slope-separated knot]
\label{def:slope-separated-intro}
A nontrivial knot $K$ is \emph{slope-separated} if
\begin{enumerate}[label=(\roman*),leftmargin=*]
\item $E(K)$ contains no closed essential surface;
\item $E(K)$ contains no essential \emph{meridional} surface, that is,
no essential surface of boundary slope $\mu$; and
\item distinct vertices of $\ES(E(K))$ have distinct boundary slopes.
\end{enumerate}
We then write $S(K)\subset\mathbb Q$ for the set of boundary slopes,
which is finite by Hatcher's theorem \cite{HatcherBoundaryCurves}, so that $\ES(E)$ is a set of $|S(K)|$ isolated points, one per slope.

Condition (ii) is not a technicality.  Writing slopes as $p/q$ with
the meridian as $\infty=1/0$, the involution $s\mapsto-s$ of
Remark~\ref{rem:GL} fixes \emph{two} slopes, namely $0$ and $\infty$;
excluding $\infty$ makes $0$ its unique fixed point, which is what the
parity statement in Theorem~\ref{thmA} uses.  Condition (ii) holds
for torus knots and, by Hatcher--Thurston, for all two-bridge knots,
whose boundary slopes are even integers \cite{HatcherThurston}.
\end{definition}

Torus knots ($S=\{0,pq\}$) and the figure-eight knot
($S=\{0,+4,-4\}$) are slope-separated; a general knot is not, and we
say so explicitly in Remark~\ref{rem:slope-sep-rare}.  For this class
the whole image--kernel problem has a complete and very short answer,
and --- this is the point --- the answer does not require knowing the
mapping class group.

\begin{theoremA}[Visibility for slope-separated knots]
\label{thmA}
Let $K$ be a nontrivial slope-separated knot, $E=E(K)$, and
$b=|S(K)|$.  For $X=\ES(E)$ or $\ES^{\mathrm{per}}(E)$ (and also
$X=\CES(E)$ when $\Fr\Char(E)=\varnothing$), every
orientation-preserving mapping class acts trivially:
\[
  \Mod^{+}(E)\subseteq \mathcal K_X(E)=\ker\Phi_X .
\]
Moreover
\[
  \Image\Phi_X\cong
  \begin{cases}
    \mathbb Z/2, & \text{if $K$ is amphichiral and $S(K)\neq\{0\}$},\\[2pt]
    1, & \text{otherwise.}
  \end{cases}
\]
The nontrivial element, when it occurs, is the sign involution
$v_s\leftrightarrow v_{-s}$.  For the bare complex one also has
\[
  \Aut(\ES(E))\cong S_b,
  \qquad |\Image\Phi_{\ES}|\le2,
\]
so the combinatorial automorphism group is usually much larger than
the geometric image.  Finally, amphichirality forces
$S(K)=-S(K)$; hence $b$ is odd, and an even number of boundary slopes
forces $K$ to be chiral and the bare action to be trivial.
\end{theoremA}

Torus knots and the figure-eight knot are immediate instances.  The
torus-knot image is trivial; for the figure-eight knot the image is
$\mathbb Z/2$, generated by amphichirality.  The point is that none of
this requires computing the mapping class group first.

The second result is a warning against generalizing from the small
examples.

\begin{theoremA}[Unbounded dimension]
\label{thmB}
For every $n\ge1$ there is a knot $K$ with $\dim\ES(E(K))\ge n$.  One
may take $K$ to be an $n$-fold iterated nontrivial cable of a
nontrivial knot; $n$ of its JSJ tori together with the outermost
cabling annulus are simultaneously disjoint and pairwise
non-isotopic, and therefore span an $n$-simplex.
\end{theoremA}

Thus $\ES(E)$ is a discrete set only in low complexity.  A first
nondiscrete example --- a single edge, spanned by the companion torus
and the cabling annulus of a cable knot --- is computed in
Section~\ref{subsec:cable-first-edge} and drawn in
Figure~\ref{fig:cable}.

The third result is the connected-sum calculation, where an
infinite-order twist is visible rather than hidden.

\begin{theoremA}[A visible twist for a connected sum]
\label{thmC}
Let $K=K_1\mathbin\#K_2$ with $K_1,K_2$ nontrivial prime knots.  Then
the standard decomposing annulus $A$ represents the unique annulus
vertex of $\ES^{\mathrm{per}}(E)$, and is therefore fixed by every
peripheral-data-preserving automorphism.  If moreover both factors are
nonfibered, then in Banks's winding coordinates for the Kakimizu
complex $\operatorname{MS}(K)$ the Dehn twist $\tau_A$ acts by
\[
  \tau_A\cdot\Psi(R_1,R_2,n)=\Psi(R_1,R_2,n+1),
\]
so $\langle\tau_A\rangle\cong\mathbb Z$ acts freely on that
Kakimizu layer.  Since $\operatorname{MS}(K)\subseteq\ES(E)$ by
Proposition~\ref{prop:kakimizu-subcomplex}, the same cyclic subgroup is
already visible in the action on $\ES(E)$.
\end{theoremA}

The last three results concern two-bridge knots $K_{p/q}$, and rest on
the Hatcher--Thurston classification \cite{HatcherThurston}; their
proofs are in \cite{OzawaTB}.  Recall
that the twist knots are the two-bridge knots $K_{2/q}$, $q\ge5$ odd:
$4_1,5_2,6_1,7_2,\dots$

\begin{theoremA}[Slope-separated two-bridge knots]
\label{thmD}
A two-bridge knot is slope-separated if and only if it is, up to
mirror image, a torus knot $T(2,q)$ or a twist knot.  For the twist
knot $K=K_{2/q}$ the complex $\ES(E)$ consists of three isolated
vertices: a once-punctured torus of slope $0$, a twice-punctured torus
of slope $4\epsilon$, and a twice-punctured surface of genus $(q-3)/2$
and slope $-(q+1)+2\epsilon$, where $\epsilon=\pm1$ and
$q\equiv\epsilon\pmod 4$.  Consequently:
\begin{enumerate}[label=\textup{(\arabic*)},leftmargin=*]
\item the figure-eight knot is the only slope-separated two-bridge
knot whose slope set is symmetric about $0$;
\item for every other slope-separated two-bridge knot the whole group
$\Mod^{\pm}(E)$ acts trivially on $\ES(E)$, $\ES^{\mathrm{per}}(E)$ and
$\CES(E)$, while $\Aut(\ES^{\mathrm{per}}(E))=1$ and
$\Aut(\ES(E))\cong S_2$ or $S_3$;
\item a slope-separated two-bridge knot is determined, among two-bridge
knots and up to mirror image, by its peripheral complex.
\end{enumerate}
\end{theoremA}

\begin{theoremA}[Finiteness and the first edge]
\label{thmE}
Let $K=K_{p/q}$ be a two-bridge knot.  Then $\ES(E(K))$ is finite if
and only if every continued-fraction expansion
$p/q=r+[b_1,\dots,b_k]$ with all $|b_i|\ge2$ has at most one odd term
$b_i$.  In that case the vertices are the single-sheeted
Hatcher--Thurston surfaces of the all-even expansion and \emph{all}
two-sheeted surfaces of the remaining expansions.  For $K=6_2$ the
complex consists of three isolated vertices and a path $w_0-w_1-w_2$
of three vertices of the same boundary slope; thus $6_2$ is the first
knot in the table, and a hyperbolic one, whose essential-surface
complex is not discrete.  The image of $\Mod^{\pm}(E)$ in
$\Aut(\ES(E))\cong S_3\times\mathbb Z/2$ is the group of order two
exchanging $w_0$ and $w_2$; it equals $\Aut(\ES^{\mathrm{per}}(E))$,
its generator is induced by the symmetry of period two and by one
strong inversion, and the kernel is generated by the other strong
inversion.
\end{theoremA}

\begin{theoremA}[The complex of a two-bridge knot]
\label{thmF}
Let $K=K_{p/q}$ be a two-bridge knot and $E=E(K)$.
\begin{enumerate}[label=\textup{(\arabic*)},leftmargin=*]
\item $\ES(E)$ is the disjoint union of the subcomplexes carried by the
individual Hatcher--Thurston branched surfaces, and the simplices of
each of them are given by the components of the carried surfaces,
which are computed by an explicit rule.  In particular essential
surfaces carried by different branched surfaces are never disjoint.
\item $\ES(E)$ is finite if and only if, up to mirror image, the
ordinary continued fraction of $q/p$ is $[a]$, $[a,b]$, $[a,1,b]$, or
$[a,1,1,b]$ with $a,b$ odd.  In that case $\ES(E)$ is a graph whose
components are paths with $1$, $2$, $3$ or $5$ vertices, listed in
Table~\ref{tab:class-two}, and the image and the kernel of the action
are determined: for $[a,b]$ and $[a,1,b]$ the action on
$\ES^{\mathrm{per}}(E)$ is surjective with kernel of order two, and for
$[a,1,1,b]$ it is injective and not surjective.
\item In general, each of the subcomplexes in \textup{(1)} is a
disjoint union of finite connected complexes, whose vertices are the
lattice points of explicit polytopes and whose edges are given by an
explicit rule (Theorem~\ref{thm:general-structure}).  For $K=6_3$ the
complex is an infinite discrete set, with two families of vertices
indexed by $\mathbb Q\cap[0,2]$.
\item If $\ES(E)$ is infinite, then $\Mod^{\pm}(E)$ acts faithfully on
it, whereas $\Aut(\ES^{\mathrm{per}}(E))$ is uncountable.  For every
hyperbolic two-bridge knot the kernel of the action is determined: it
is $\Mod^{+}(E)$ for the twist knots, of order two for the other
double twist knots, and trivial otherwise.
\end{enumerate}
\end{theoremA}

The precise statements are
Theorem~\ref{thm:two-bridge-slope-separated},
Corollaries~\ref{cor:two-bridge-image-kernel}
and~\ref{cor:two-bridge-question},
Propositions~\ref{prop:two-bridge-reconstruction},
\ref{prop:two-bridge-finite} and~\ref{prop:six-two}, and
Theorems~\ref{thm:six-two}, \ref{thm:two-bridge-decomposition},
\ref{thm:three-classes}, \ref{thm:class-two},
\ref{thm:six-three}, \ref{thm:general-structure},
\ref{thm:two-bridge-kernel} and~\ref{thm:class-three-aut} in
Section~\ref{sec:two-bridge}, and the proofs are in \cite{OzawaTB}.  In contrast with Theorem~\ref{thmA}, the
example $6_2$ shows that outside the slope-separated class
orientation-preserving symmetries of finite order can be visible, even
for a hyperbolic knot.  In addition,
torus knots are determined among \emph{all} knots by their peripheral
complex (Proposition~\ref{prop:torus-reconstruction}).

\subsection{Scope and context}
\label{subsec:relation}

The surface classifications, symmetry groups, and $3$-manifold
structure theorems used below are classical and are cited where they
enter: Seifert-fibered surface theory \cite{HatcherNotes,Waldhausen1967},
Hatcher--Thurston's two-bridge census \cite{HatcherThurston},
Jaco--Shalen--Johannson theory \cite{JS,Johannson}, hyperbolic rigidity
\cite{Thurston}, Schubert prime decomposition \cite{Schubert}, and
Banks's connected-sum model for Kakimizu complexes
\cite{BanksConnectedSum}.  What this article offers is a common
image--kernel--reconstruction viewpoint, the elementary structural
results that make the model calculations uniform, a survey of the
solution for two-bridge knots obtained in \cite{OzawaTB}
(Section~\ref{sec:two-bridge}), and a
set of computable problems and recipes for further examples.

The closest general precedent for the underlying complex is
Schultens's surface complex $\mathcal S(M)$
\cite{SchultensSurfaceComplex}.  Proposition~\ref{prop:ES-flag}
identifies $\ES(E)$ with her initial complex $\mathcal S_0(E)$.  Zhang
and Guo \cite{ZhangGuo} study a related boundary-bearing complex for a
pair $(M,S)$; their contractibility theorem assumes that $S$ is
compressible, so it does not apply to the boundary of a nontrivial
knot exterior.

In dimension three the closest precedents for the image--kernel
question itself concern handlebodies and connected sums of
$S^2\times S^1$, and there the answer is complete.  Korkmaz and
Schleimer show that the automorphism group of the disk complex of a
handlebody is its mapping class group \cite{KorkmazSchleimer}.
Charitos, Papadoperakis and Tsapogas build a complex of incompressible
surfaces in a handlebody $H_g$, $g\ge2$, whose vertices are compressing
disks, boundary-parallel annuli and pairs of pants, and prove that the
natural map $\Mod^{\pm}(H_g)\to\Aut$ is surjective, with trivial kernel
for $g\ge3$ and kernel $\mathbb Z/2$, generated by the hyperelliptic
involution, for $g=2$ \cite{CharitosPapadoperakisTsapogas}; see also
\cite{CharitosIncompressible}.  For $M_n=\#_n S^2\times S^1$ the
sphere complex is the graph of free splittings, whose automorphism
group is $\Out(F_n)$ by Aramayona--Souto \cite{AramayonaSoutoSplittings}.
Since $\Mod(M_n)\to\Out(F_n)$ is onto with a finite kernel generated
by sphere twists \cite{Laudenbach}, which act trivially on the sphere
complex, this is also the earliest instance we know of a nontrivial
kernel generated by twists that are invisible to a complex of
surfaces.  Knot exteriors behave differently on both sides: already
for torus knots and the figure-eight knot the image is a proper
subgroup of $\Aut(\ES(E))$ and the kernel is nontrivial
(Section~\ref{sec:firstlook}), so the question becomes which additional
information restores rigidity.  Heegaard-theoretic disjointness
complexes \cite{Schleimer} are a further, more distant, precedent.

Kakimizu's complexes $\operatorname{IS}(K)$ and
$\operatorname{MS}(K)$ \cite{Kakimizu,KakimizuClassification,
PrzytyckiSchultens} occur naturally as full subcomplexes of $\ES(E)$;
see Proposition~\ref{prop:kakimizu-subcomplex}.  Thus a symmetry
detected on a Kakimizu layer is automatically detected by the full
essential-surface action, although recognizing that layer intrinsically
inside the bare complex is a separate problem.

On the Kakimizu layers, finite symmetries have been studied through
invariant Seifert surfaces.  Edmonds showed that a periodic knot has
an invariant minimal-genus Seifert surface \cite{Edmonds};
Przytycki and Schultens showed that every finite group of symmetries
fixes a simplex of $\operatorname{MS}(K)$ \cite{PrzytyckiSchultens};
and Hirasawa, Hiura and Sakuma showed that a strong inversion either
fixes a vertex of $\operatorname{MS}(K)$ or exchanges two adjacent
vertices \cite[Corollary~1.6]{HHSInvariant}, with the two-bridge case
treated in \cite{HirasawaHiuraSakuma}.  These are fixed-point
statements rather than kernel computations, but they are the
counterpart, on the Kakimizu layer, of the finite invisible symmetries
of Section~\ref{sec:visibility}.

The analogy with Ivanov theory should not be read as a claim of a
general rigidity theorem.  More broadly, rigidity questions for curve-
and region-type complexes belong to the circle of ideas often called
Ivanov's metaconjectural programme; see, for example,
\cite{AramayonaSoutoSurvey,McCarthyPapadopoulos,BrendleMargalit}.
These works also underscore a useful caution here: enriching a
combinatorial object does not by itself guarantee that all of its
automorphisms are geometric.  The model examples below show both
nongeometric automorphisms and nontrivial kernels.  Our aim is instead
to isolate which small amount of information repairs which failure,
and to state the remaining recognition and realization problems
cleanly.

\subsection{What is classical and what is recorded here}
\label{subsec:status}

Since a good part of this article is expository, we state explicitly
where each ingredient comes from.  ``Classical'' means that the statement, or a
statement from which it follows in a line, is in the cited
literature; ``elementary'' means that it follows from classical
results by a short argument which is standard among specialists,
whether or not it is written down; ``recorded here'' means that we
have not found the statement in this form in the literature, although
its proof is short and uses only the classical input indicated;
the results on two-bridge knots surveyed in
Section~\ref{sec:two-bridge} are proved in the companion paper
\cite{OzawaTB}, as indicated.  Table~\ref{tab:status} lists
the ingredients accordingly.

\begin{table}[htbp]
\centering
\caption{Status of the ingredients of this article.}
\label{tab:status}
\small
\renewcommand{\arraystretch}{1.15}
\begin{tabularx}{\linewidth}{@{}>{\raggedright\arraybackslash}p{0.36\linewidth}>{\raggedright\arraybackslash}p{0.16\linewidth}Y@{}}
\toprule
statement & where & status and source\\
\midrule
$\ES(E)=\mathcal S_0(E)$; flag property
 & Prop.~\ref{prop:ES-flag}
 & classical: Schultens's complex \cite{SchultensSurfaceComplex}, with Neumann--Swarup's disjointness lemma \cite{NeumannSwarup}\\
finite dimension
 & Lemma~\ref{lem:finite-dim}
 & classical: Kneser--Haken finiteness \cite{Hempel,JS}\\
sign rule for slopes
 & Remark~\ref{rem:GL}
 & classical: Gordon--Luecke \cite{GL}\\
layering by boundary slope
 & Lemma~\ref{lem:slope-layer}
 & elementary\\
Kakimizu complexes as full subcomplexes
 & Prop.~\ref{prop:kakimizu-subcomplex}
 & elementary, from the definitions \cite{Kakimizu,PrzytyckiSchultens}\\
surface censuses and symmetry groups of $T(p,q)$ and $4_1$
 & Sections~\ref{sec:torus-knots}, \ref{sec:figure-eight}
 & classical: \cite{TsauTorusSurfaces,HatcherThurston,HeusenerMunozPorti,SakumaSurvey}\\
unbounded dimension (Theorem~\ref{thmB})
 & Prop.~\ref{prop:dimension}
 & elementary, from the JSJ decomposition of iterated cables \cite{Budney,JS}\\
visibility for slope-separated knots (Theorem~\ref{thmA})
 & Thm.~\ref{thm:visibility}
 & recorded here; uses only Lemma~\ref{lem:slope-layer} and Remark~\ref{rem:GL}\\
twist on the winding coordinate (Theorem~\ref{thmC})
 & Prop.~\ref{prop:kakimizu-spinning}
 & recorded here; a naturality consequence of Banks's construction \cite{BanksConnectedSum}\\
slope-separated two-bridge knots are the knots $T(2,q)$ and the twist knots; their complexes, images and kernels (Theorem~\ref{thmD})
 & Thm.~\ref{thm:two-bridge-slope-separated}, Cor.~\ref{cor:two-bridge-image-kernel}, \ref{cor:two-bridge-question}
 & \cite{OzawaTB}; from Hatcher--Thurston \cite{HatcherThurston} by a lattice count and a continued-fraction argument\\
finiteness criterion and vertex census for two-bridge knots; orientability of two-sheeted surfaces (Theorem~\ref{thmE})
 & Prop.~\ref{prop:two-bridge-finite}, Lemma~\ref{lem:carried-basic}(4)
 & \cite{OzawaTB}; the remaining parts of Lemmas~\ref{lem:carried-basic}, \ref{lem:orbits} and the uniqueness of the Seifert vertex are in \cite{HatcherThurston}\\
$\ES(E(6_2))$, its automorphisms, image and kernel (Theorem~\ref{thmE})
 & Prop.~\ref{prop:six-two}, Thm.~\ref{thm:six-two}
 & \cite{OzawaTB}\\
decomposition of $\ES(E)$ by branched surfaces; sheet rule; components and orientability of carried surfaces (Theorem~\ref{thmF})
 & Thm.~\ref{thm:two-bridge-decomposition}, Lemma~\ref{lem:sheet-rule}, Prop.~\ref{prop:components}
 & \cite{OzawaTB}; the decomposition is \cite[Theorem~1]{HatcherThurston} applied to disconnected surfaces, and Prop.~\ref{prop:components} sharpens \cite[Proposition~1]{HatcherThurston}\\
the three classes; complete results for Class~II; the knot $6_3$ (Theorem~\ref{thmF})
 & Thm.~\ref{thm:three-classes}, \ref{thm:class-two}, \ref{thm:six-three}
 & \cite{OzawaTB}\\
general structure of $\ES_\beta(E)$; kernel for all two-bridge knots; non-rigidity in Class~III (Theorem~\ref{thmF})
 & Thm.~\ref{thm:general-structure}, \ref{thm:two-bridge-kernel}, \ref{thm:class-three-aut}
 & \cite{OzawaTB}\\
action of the standard symmetries on Hatcher--Thurston coordinates; non-adjacency of doubles
 & Prop.~\ref{prop:two-bridge-symmetry-action}, Lemma~\ref{lem:doubles-not-adjacent}
 & \cite{OzawaTB}; the first from the proof of \cite[Theorem~1(b)]{HatcherThurston}\\
reconstruction of torus knots, and of slope-separated two-bridge knots among two-bridge knots
 & Prop.~\ref{prop:torus-reconstruction}, \ref{prop:two-bridge-reconstruction}
 & \cite{OzawaTB}; short proofs are included\\
image and kernel for handlebodies and $\#_n S^2\times S^1$
 & Section~\ref{subsec:relation}
 & classical: \cite{KorkmazSchleimer,CharitosPapadoperakisTsapogas,AramayonaSoutoSplittings,Laudenbach}\\
three levels, three tasks, frontier problem, problem list, two-bridge recipe
 & Sections~\ref{sec:decorations}, \ref{sec:reduction}, \ref{sec:problems}, App.~\ref{app:recipe}
 & expository organization; the open problems are stated as problems\\
\bottomrule
\end{tabularx}
\end{table}

\subsection{How to read this article}

Figure~\ref{fig:roadmap} summarizes the reading order.  The hands-on
examples and the visibility theorem form the self-contained core; the
later sections add connected sums, JSJ structure, and the roadmap
toward rigidity.

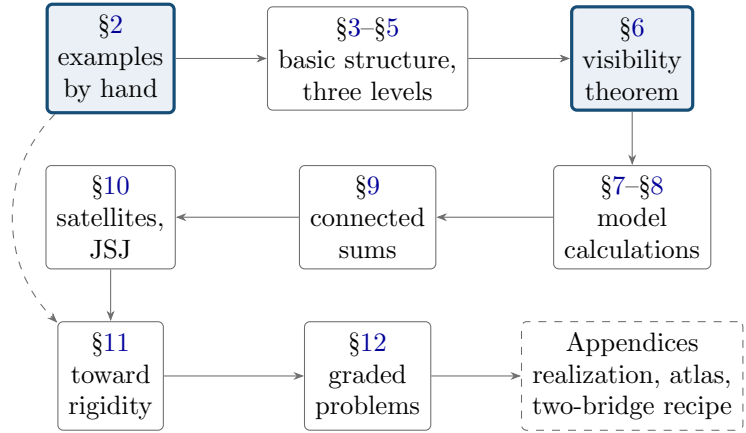
\begin{figure}[htb]
\centering
\begin{tikzpicture}[
  node distance=6mm,
  box/.style={draw=greyish,rounded corners=2pt,align=center,
              inner sep=4pt,font=\small,minimum height=8mm},
  hi/.style={box,draw=surfA!70!black,fill=lightfill,very thick},
  ap/.style={box,draw=greyish,dashed,fill=white},
  ar/.style={-{Stealth[length=4pt]},greyish}
]
\node[hi] (s2) at (0,0) {\S\ref{sec:firstlook}\\ examples\\ by hand};
\node[box] (s34) at (3.4,0) {\S\ref{sec:definitions}--\S\ref{sec:decorations}\\ basic structure,\\ three levels};
\node[hi] (s6) at (6.9,0) {\S\ref{sec:visibility}\\ visibility\\ theorem};
\node[hi] (s78) at (6.9,-2.1) {\S\ref{sec:torus-knots}--\S\ref{sec:two-bridge}\\ model calculations,\\ two-bridge knots};
\node[box] (s9) at (3.4,-2.1) {\S\ref{sec:composite}\\ connected\\ sums};
\node[box] (s10) at (0,-2.1) {\S\ref{sec:satellite}\\ satellites,\\ JSJ};
\node[box] (s11) at (0,-4.2) {\S\ref{sec:reduction}\\ toward\\ rigidity};
\node[box] (s12) at (3.4,-4.2) {\S\ref{sec:problems}\\ problems\\ with status};
\node[ap] (app) at (6.9,-4.2) {Appendices\\ realization, atlas,\\ two-bridge recipe};
\draw[ar] (s2) -- (s34);
\draw[ar] (s34) -- (s6);
\draw[ar] (s6) -- (s78);
\draw[ar] (s78) -- (s9);
\draw[ar] (s9) -- (s10);
\draw[ar] (s10) -- (s11);
\draw[ar] (s11) -- (s12);
\draw[ar] (s12) -- (app);
\draw[ar,dashed] (s2) to[out=-135,in=135] (s11);
\end{tikzpicture}
\caption{A map of the article.  The shaded boxes are the easiest entry
points: examples first, then the visibility theorem, then its
complete application to two-bridge knots.  The dashed arrow
indicates that a reader interested mainly in the general rigidity
programme may pass from the examples to Section~\ref{sec:reduction}.}
\label{fig:roadmap}
\end{figure}

\section{A first look: three complexes you can draw by hand}
\label{sec:firstlook}

This section is an unhurried introduction.  Everything in it is
recomputed carefully later; here we only want the reader to see the
objects.  Proofs of the surface censuses are postponed to
Sections~\ref{sec:torus-knots}--\ref{sec:satellite}, and the general
definitions to Section~\ref{sec:complexes}.

\subsection{The exterior, its boundary torus, and slopes}
\label{subsec:slopes}

Let $K\subset S^3$ be a nontrivial knot and let
$E=E(K)=S^3\setminus\Int N(K)$ be its exterior, a compact orientable
$3$-manifold whose boundary is a torus.  On $\partial E$ we single out
two curves: the \emph{meridian} $\mu$, which bounds a disk in the
solid torus $N(K)$, and the \emph{preferred longitude} $\lambda$,
characterized by $\lk(K,\lambda)=0$.  Every isotopy class of essential
simple closed curve on $\partial E$ --- a \emph{slope} --- is then
$p\mu+q\lambda$ for coprime integers $p,q$, and we write it as the
fraction $p/q$.  Slopes are unoriented, so $p\mu+q\lambda$ and
$-p\mu-q\lambda$ are the same slope (Figure~\ref{fig:slopes}).

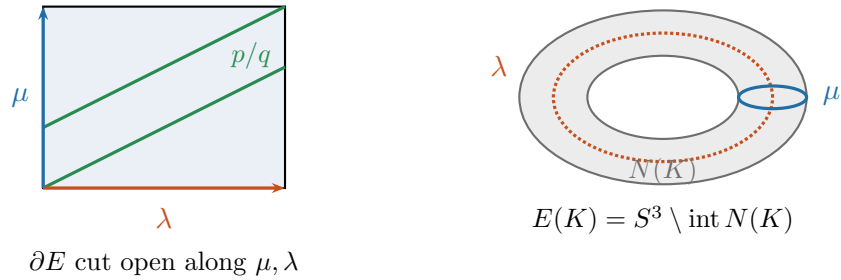
\begin{figure}[htb]
\centering
\begin{tikzpicture}[scale=1.0]
% ---- left: boundary torus cut open ----
\begin{scope}[xshift=0cm]
  \fill[lightfill] (0,0) rectangle (3.2,2.4);
  \draw[thick] (0,0) rectangle (3.2,2.4);
  \draw[surfC,very thick] (0,0) -- (3.2,1.6);
  \draw[surfC,very thick] (0,0.8) -- (3.2,2.4);
  \node[surfC,above left,font=\small] at (3.15,1.45) {$p/q$};
  \draw[surfA,very thick,-{Stealth[length=5pt]}] (0,0) -- (0,2.4);
  \draw[surfB,very thick,-{Stealth[length=5pt]}] (0,0) -- (3.2,0);
  \node[surfA,left] at (-0.05,1.2) {$\mu$};
  \node[surfB,below] at (1.6,-0.12) {$\lambda$};
  \node[font=\small] at (1.6,-1.0) {$\partial E$ cut open along $\mu,\lambda$};
\end{scope}
% ---- right: the exterior ----
\begin{scope}[xshift=8.2cm,yshift=1.2cm]
  \fill[greyish,opacity=0.13] (0,0) ellipse (1.9 and 1.15);
  \fill[white] (0,0) ellipse (1.0 and 0.55);
  \draw[thick,greyish] (0,0) ellipse (1.9 and 1.15);
  \draw[thick,greyish] (0,0) ellipse (1.0 and 0.55);
  \draw[surfB,very thick,densely dotted] (0,0) ellipse (1.45 and 0.85);
  \draw[surfA,very thick] (1.45,0) ellipse (0.45 and 0.15);
  \node[surfA,right] at (1.98,0.02) {$\mu$};
  \node[surfB,left] at (-1.92,0.42) {$\lambda$};
  \node[font=\small] at (0,-1.62) {the torus $\partial E=\partial N(K)$};
\end{scope}
\end{tikzpicture}
\caption{Slopes on the boundary torus.  Left: $\partial E$ cut open
along $\mu$ and $\lambda$; a slope $p/q$ is a line of rational
direction.  Right: the meridian $\mu$ and the preferred longitude
$\lambda$ drawn on the torus $\partial E=\partial N(K)$.  Two essential
curves of distinct slopes on a torus have nonzero geometric
intersection number --- this single fact is responsible for most of
the combinatorics below.}
\label{fig:slopes}
\end{figure}

A properly embedded compact connected surface $F\subset E$ is
\emph{essential} if it is two-sided, orientable, incompressible,
boundary-incompressible, not boundary-parallel, and is neither a
sphere nor a disk (Definition~\ref{def:essential}).  Its boundary
$\partial F$, if nonempty, is a union of disjoint essential curves on
the torus $\partial E$; disjoint essential curves on a torus are
parallel, so all of them have one and the same slope.  That slope is
the \emph{boundary slope} $\Sl(F)$.

\begin{keybox}
\textbf{The one fact to remember.}  If two essential surfaces have
disjoint representatives, then their boundaries are disjoint on the
torus $\partial E$, hence parallel.  So:
\[
  \Sl(F)\neq\Sl(F')\ \Longrightarrow\
  \text{$[F]$ and $[F']$ are \emph{not} joined by an edge.}
\]
Different slopes can never coexist in one simplex.
\end{keybox}

\subsection{Torus knots: two isolated points}
\label{subsec:torus-first}

Let $K=T(p,q)$ with $1<p<q$ and $\gcd(p,q)=1$.  The exterior $E$ is
Seifert fibered over the disk with two cone points of orders $p$ and
$q$, and every essential surface in such a manifold is isotopic either
to a \emph{vertical} surface (a union of fibres) or to a
\emph{horizontal} one (transverse to every fibre)
\cite{HatcherNotes,Waldhausen1967}.  Working in the base orbifold
$D^2(p,q)$ makes the census immediate (Figure~\ref{fig:torusknot}):

\begin{itemize}[leftmargin=*]
\item vertical surfaces sit over arcs and loops in $D^2(p,q)$.  A loop
around one cone point gives a compressible torus, a loop around both
gives a boundary-parallel one; an arc cutting off a disk with no cone
point gives a boundary-parallel annulus.  The only surviving
possibility is the arc separating the two cone points, and its
preimage is the \emph{cabling annulus} $A$, of boundary slope $pq$.
\item horizontal surfaces are fibres of a fibration of $E$ over $S^1$;
since $H^1(E;\mathbb Z)\cong\mathbb Z$ there is only one up to isotopy,
namely the \emph{fibre surface} $F_0$, of boundary slope $0$.
\end{itemize}

\noindent
So $\ES(E)$ has exactly two vertices.  They cannot be joined by an
edge, because their slopes $0$ and $pq$ differ.  Hence:
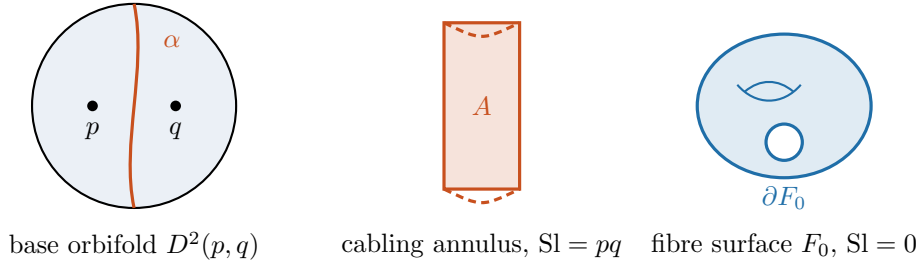
\[
  \ES\bigl(E(T(p,q))\bigr)=\{\,[F_0]\,\}\ \sqcup\ \{\,[A]\,\}.
\]

\begin{figure}[htb]
\centering
\begin{tikzpicture}[scale=1.0]
% base orbifold
\begin{scope}[xshift=0cm]
  \fill[lightfill] (0,0) circle (1.35);
  \draw[thick] (0,0) circle (1.35);
  \fill (-0.55,0) circle (2pt); \node[below] at (-0.55,-0.08) {\small $p$};
  \fill (0.55,0) circle (2pt);  \node[below] at (0.55,-0.08) {\small $q$};
  \draw[surfB,very thick] (0,1.35) .. controls (0.18,0.5) and (-0.18,-0.5) .. (0,-1.35);
  \node[surfB] at (0.5,0.85) {\small $\alpha$};
  \node at (0,-1.85) {\small base orbifold $D^2(p,q)$};
\end{scope}
% cabling annulus preimage
\begin{scope}[xshift=4.6cm]
  \fill[surfB,opacity=0.16] (-0.5,-1.1) rectangle (0.5,1.1);
  \draw[surfB,very thick] (-0.5,-1.1) rectangle (0.5,1.1);
  \draw[surfB,very thick,densely dashed] (-0.5,1.1) .. controls (0,0.86) .. (0.5,1.1);
  \draw[surfB,very thick,densely dashed] (-0.5,-1.1) .. controls (0,-1.34) .. (0.5,-1.1);
  \node[surfB] at (0,0) {\small $A$};
  \node at (0,-1.85) {\small cabling annulus, $\Sl=pq$};
\end{scope}
% fibre surface
\begin{scope}[xshift=8.6cm]
  \fill[surfA,opacity=0.14] (0,0) ellipse (1.15 and 0.95);
  \draw[surfA,very thick] (0,0) ellipse (1.15 and 0.95);
  \draw[surfA,thick] (-0.62,0.28) .. controls (-0.38,0.0) and (-0.02,0.0) .. (0.22,0.28);
  \draw[surfA,thick] (-0.52,0.19) .. controls (-0.32,0.36) and (-0.08,0.36) .. (0.12,0.19);
  \fill[white] (0.0,-0.48) circle (0.24);
  \draw[surfA,very thick] (0.0,-0.48) circle (0.24);
  \node[surfA,font=\small] at (0,-1.22) {$\partial F_0$};
  \node at (0,-1.85) {\small fibre surface $F_0$, $\Sl=0$};
\end{scope}
% the complex
\begin{scope}[yshift=-3.0cm,xshift=4.3cm]
  \node at (-2.2,0) {$\ES(E)\;=$};
  \fill[surfA] (-0.6,0) circle (2.4pt);  \node[below=2pt,surfA] at (-0.6,0) {\small $[F_0]$, $0$};
  \fill[surfB] (0.9,0) circle (2.4pt);   \node[below=2pt,surfB] at (0.9,0) {\small $[A]$, $pq$};
  \node at (2.9,0) {\small (no edge: $0\neq pq$)};
\end{scope}
\end{tikzpicture}
\caption{The torus-knot complex.  The only essential arc in the base
orbifold separates the two cone points; its preimage is the cabling
annulus.  The unique horizontal surface is the fibre surface.  The
complex is two isolated points.}
\label{fig:torusknot}
\end{figure}

Now look at the symmetry.  The mapping class group of $E(T(p,q))$ is
$\mathbb Z/2$, generated by a strong inversion, and \emph{every} class
is orientation-preserving (Lemma~\ref{lem:torus-mcg}).  An
orientation-preserving homeomorphism fixes every unoriented slope, so
it fixes both vertices; hence the entire mapping class group is
invisible.  Meanwhile $\Aut(\ES(E))\cong\mathbb Z/2$: there is a
transposition of the two points, induced by no homeomorphism, since a
homeomorphism cannot send an annulus to a positive-genus surface.

\begin{keybox}
Already in the simplest example the two failures appear together.
\emph{The kernel is nontrivial}, and its generator is a finite
symmetry rather than a twist, so no refinement by twist coordinates
will remove it.  \emph{The image is not everything}: the bare
complex admits a nongeometric automorphism.  Adding the surface type
--- or, alternatively, the boundary slope --- kills the false
automorphism at once.
\end{keybox}

\subsection{The figure-eight knot: three points and a mirror}
\label{subsec:fig8-first}

Now let $K=4_1$ and $E_8=E(4_1)$.  Hatcher--Thurston's census for
two-bridge knots \cite{HatcherThurston} gives three essential surfaces
up to isotopy (Lemma~\ref{lem:figure-eight-census}):
a once-punctured torus $F$ of slope $0$, and two twice-punctured tori
$S_+$ and $S_-$ of slopes $+4$ and $-4$.  Their slopes are pairwise
distinct, so again there are no edges:
\[
  \ES(E_8)=\{\,[F]\,\}\ \sqcup\ \{\,[S_+]\,\}\ \sqcup\ \{\,[S_-]\,\}.
\]

The figure-eight knot is amphichiral, so $E_8$ admits an
orientation-reversing self-homeomorphism, and this acts on slopes by
$s\mapsto-s$.  It therefore fixes $[F]$ and interchanges $[S_+]$ and
$[S_-]$ (Figure~\ref{fig:fig8}).  Every orientation-preserving class,
by contrast, fixes all three vertices.  Since $\Mod^{\pm}(E_8)$ is
dihedral of order eight with orientation-preserving subgroup
$(\mathbb Z/2)^2$ \cite{HeusenerMunozPorti,SakumaSurvey}, we get
\[
  \Image\Phi\cong\mathbb Z/2,
  \qquad
  \mathcal K\cong(\mathbb Z/2)^2 ,
  \qquad
  \Mod^{\pm}(E_8)/\mathcal K\;\cong\;\mathbb Z/2 .
\]
On the other hand $\Aut(\ES(E_8))\cong S_3$, because three isolated
points have six symmetries.  Four of the six are nongeometric.

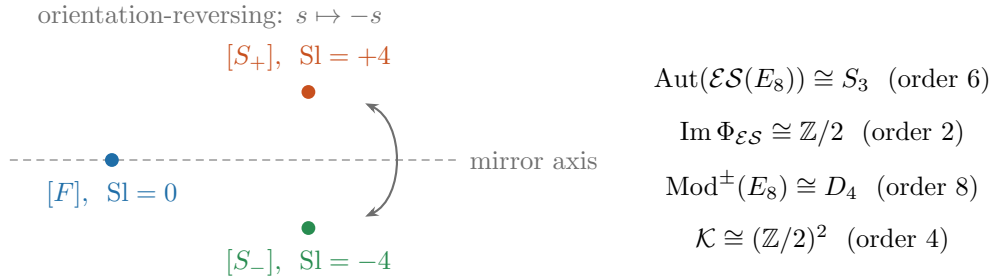
\begin{figure}[htb]
\centering
\begin{tikzpicture}[scale=1.0]
\begin{scope}
  \draw[greyish,densely dashed] (-1.35,0) -- (4.55,0);
  \node[greyish,right,font=\small] at (4.6,0) {mirror axis};
  \fill[surfA] (0,0) circle (2.6pt);
  \node[below=4pt,surfA,font=\small] at (0,0) {$[F]$, \ $\Sl=0$};
  \fill[surfB] (2.6,0.9) circle (2.6pt);
  \node[above=4pt,surfB,font=\small] at (2.6,0.9) {$[S_+]$, \ $\Sl=+4$};
  \fill[surfC] (2.6,-0.9) circle (2.6pt);
  \node[below=4pt,surfC,font=\small] at (2.6,-0.9) {$[S_-]$, \ $\Sl=-4$};
  \draw[{Stealth[length=5pt]}-{Stealth[length=5pt]},greyish,thick]
        (3.35,0.75) to[out=-25,in=25] (3.35,-0.75);
  \node[greyish,font=\small] at (1.3,1.9) {orientation-reversing: $s\mapsto -s$};
\end{scope}
\begin{scope}[xshift=9.4cm]
  \node[font=\small,align=left] at (0,1.05)
    {$\Aut(\ES(E_8))\cong S_3$ \ (order $6$)};
  \node[font=\small,align=left] at (0,0.35)
    {$\Image\Phi_{\ES}\cong\mathbb Z/2$ \ (order $2$)};
  \node[font=\small,align=left] at (0,-0.35)
    {$\Mod^{\pm}(E_8)\cong D_4$ \ (order $8$)};
  \node[font=\small,align=left] at (0,-1.05)
    {$\mathcal K\cong(\mathbb Z/2)^2$ \ (order $4$)};
\end{scope}
\end{tikzpicture}
\caption{The figure-eight complex: three isolated vertices, one per
boundary slope.  The only geometrically realized nontrivial symmetry
is the mirror, which reverses the sign of every slope.  Four of the
six combinatorial symmetries are not geometric, and half of the
mapping class group is invisible.}
\label{fig:fig8}
\end{figure}

Comparing the two examples suggests the mechanism which
Section~\ref{sec:visibility} isolates.  In both cases each slope
carries exactly one surface, so a homeomorphism can only permute
vertices the way it permutes slopes --- and it permutes slopes in only
one of two ways, according to whether it preserves or reverses
orientation.  That is the whole story, and it is the content of
Theorem~\ref{thmA}.

\subsection{A cable knot: the first edge}
\label{subsec:cable-first-edge}

The two examples above are discrete, and it would be easy to conclude
that $\ES(E)$ is always a set of points.  It is not.

Let $J$ be a nontrivial knot and let $K=C_{p,q}(J)$ be a nontrivial
cable: place $K$ on the boundary of a concentric solid torus
$V'\subset\Int N(J)$, running $p\ge2$ times longitudinally and $q$
times meridionally.  Then $E(K)$ contains two obviously disjoint
essential surfaces (Figure~\ref{fig:cable}):

\begin{itemize}[leftmargin=*]
\item the \emph{companion torus} $T=\partial N(J)$, which is the JSJ
torus of $E(K)$ and is essential because $J$ is nontrivial;
\item the \emph{cabling annulus}
$A=\partial V'\setminus\Int N(K)$, properly embedded in $E(K)$ with
$\partial A\subset\partial E(K)$ of slope $pq$.
\end{itemize}

Since $A\subset\partial V'$ and $V'\subset\Int N(J)$, the two surfaces
are disjoint; and they are certainly non-isotopic, one being closed
and the other not.  So $\{[A],[T]\}$ is an edge of $\ES(E(K))$.
Iterating the construction produces higher-dimensional simplices, and
this is how Theorem~\ref{thmB} is proved in
Section~\ref{subsec:dimension}.

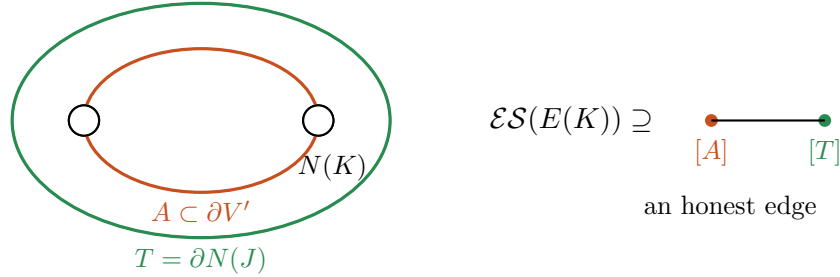
\begin{figure}[htb]
\centering
\begin{tikzpicture}[scale=1.0]
% nested tori, drawn in cross-section (two concentric annuli)
\draw[thick,greyish] (0,0) ellipse (2.5 and 1.55);

\draw[surfC,very thick] (0,0) ellipse (2.5 and 1.55);
% the knot K on partial V', drawn as small circles
\fill[white] (1.55,0) circle (0.2);
\draw[thick] (1.55,0) circle (0.2);
\fill[white] (-1.55,0) circle (0.2);
\draw[thick] (-1.55,0) circle (0.2);
\node[font=\small] at (1.98,0.42) {$N(K)$};
% cabling annulus: the part of partial V' outside N(K)
\draw[surfB,very thick] (0,0) ellipse (1.55 and 0.95);
\fill[white] (1.55,0) circle (0.2);
\draw[thick] (1.55,0) circle (0.2);
\fill[white] (-1.55,0) circle (0.2);
\draw[thick] (-1.55,0) circle (0.2);
\node[surfB] at (0,-1.2) {\small $A\subset\partial V'$};
\node[surfC] at (0,-1.85) {\small $T=\partial N(J)$};
% right: the edge
\begin{scope}[xshift=6.4cm]
  \node at (-1.5,0) {$\ES(E(K))\supseteq$};
  \fill[surfB] (0.35,0) circle (2.4pt);
  \fill[surfC] (1.85,0) circle (2.4pt);
  \draw[thick] (0.35,0) -- (1.85,0);
  \node[below=3pt,surfB] at (0.35,0) {\small $[A]$};
  \node[below=3pt,surfC] at (1.85,0) {\small $[T]$};
  \node[font=\small,align=center] at (0.6,-1.15) {an honest edge};
\end{scope}
\end{tikzpicture}
\caption{A cable knot $K=C_{p,q}(J)$ over a nontrivial companion $J$,
in schematic cross-section.  The cabling annulus $A$ lies on the inner
torus $\partial V'$ and the companion torus $T=\partial N(J)$ lies
outside it, so the two are disjoint and span an edge.  Iterating the
cabling gives simplices of every dimension
(Proposition~\ref{prop:dimension}).}
\label{fig:cable}
\end{figure}

\begin{proposition}[Kakimizu layers are natural subcomplexes]
\label{prop:kakimizu-subcomplex}
Let $K$ be a nontrivial knot with exterior $E$.  By a
\emph{spanning-surface vertex} we mean a vertex represented by an
orientable spanning surface (oriented in Kakimizu's convention
\cite{Kakimizu}), equivalently a connected properly embedded surface whose
boundary is a single preferred longitude on $\partial E$.  Kakimizu's
complex $\operatorname{IS}(K)$ of incompressible spanning surfaces is
the full subcomplex of $\ES(E)$ spanned by these vertices, and
$\operatorname{MS}(K)$ is the full subcomplex of
$\operatorname{IS}(K)$ spanned by the minimal-genus vertices.  Thus
there are natural simplicial inclusions
\[
  \operatorname{MS}(K)\subseteq\operatorname{IS}(K)\subseteq\ES(E).
\]
These inclusions are equivariant for the natural mapping-class-group
actions.
\end{proposition}

\begin{proof}
For a nontrivial knot, an incompressible orientable spanning surface
is also boundary-incompressible.  Indeed, in an irreducible manifold
with incompressible torus boundary, an incompressible but
boundary-compressible orientable surface is a boundary-parallel disk or
annulus; neither can be a spanning surface of a nontrivial knot.  Hence every vertex of $\operatorname{IS}(K)$ is a vertex
of $\ES(E)$.  The simplices in both complexes are defined by the same
condition --- the existence of simultaneously disjoint representatives
in the knot exterior \cite{Kakimizu,PrzytyckiSchultens} --- so the
inclusion is full.  Restricting to minimal genus gives
$\operatorname{MS}(K)$.  Equivariance is immediate from the natural
action on surfaces.
\end{proof}

\begin{remark}[What the cyclic cover is used for]
The infinite cyclic cover remains fundamental in the study of
Kakimizu complexes: it supplies useful distance and ordering
constructions \cite{SakumaShackleton,PrzytyckiSchultens} and is
central in many proofs \cite{SchultensSimplyConnected}.  It is not, however, a
different edge relation for the knot Kakimizu complex itself.  The
point here is therefore not to contrast two notions of disjointness,
but to regard the Kakimizu complexes as distinguished spanning-surface
layers inside the much larger all-essential-surface complex.
\end{remark}

\subsection{Summary of the first look}

\begin{center}
\small
\renewcommand{\arraystretch}{1.2}
\begin{tabular}{@{}lccc@{}}
\toprule
 & $T(p,q)$ & $4_1$ & $C_{p,q}(J)$, $J\ne$ unknot \\
\midrule
vertices of $\ES(E)$ & $2$ & $3$ & $\ge2$ \\
boundary slopes & $0,\,pq$ & $0,\,\pm4$ & includes $0,\,pq$ \\
simplices of dimension $\ge1$ & none & none & at least one edge\\
$\Aut(\ES(E))$ & $\mathbb Z/2$ & $S_3$ & --- \\
$\Mod^{\pm}(E)$ & $\mathbb Z/2$ & $D_4$ (order $8$) & --- \\
geometric image & $1$ & $\mathbb Z/2$ & --- \\
kernel & $\mathbb Z/2$ & $(\mathbb Z/2)^2$ & --- \\
\bottomrule
\end{tabular}
\end{center}

\noindent
Two morals, both proved in general in Section~\ref{sec:visibility}.
First, when each slope carries one surface, the action can only see
the sign of a slope, so at most a $\mathbb Z/2$ is visible and
everything orientation-preserving is invisible.  Second, the
bare complex has far too many automorphisms; some peripheral information may be needed, and the interesting question is how little suffices.

\subsection{A five-step workflow for a new knot family}
\label{subsec:workflow}

The examples suggest a practical way to enter the subject without
solving the full reconstruction problem at once.  For a new knot
family, separate the calculation into the five steps of
Figure~\ref{fig:family-workflow}.  A complete census of all essential
surfaces is ideal, but it is not always the first thing one needs: a
small invariant collection can already reveal a nongeometric
permutation, an invisible finite symmetry, or a visible twist.

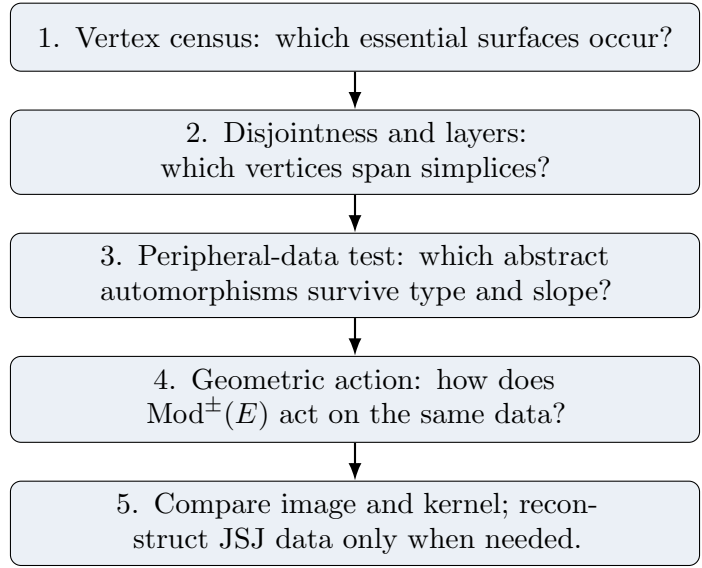
\begin{figure}[htb]
\centering
\begin{tikzpicture}[>=Latex,node distance=5mm,
  box/.style={draw,rounded corners,align=center,inner sep=4.5pt,
  text width=88mm,minimum height=9mm,fill=lightfill}]
\node[box] (a) {1. Vertex census: which essential surfaces occur?};
\node[box,below=of a] (b) {2. Disjointness and layers: which vertices span simplices?};
\node[box,below=of b] (c) {3. Peripheral-data test: which abstract automorphisms survive type and slope?};
\node[box,below=of c] (d) {4. Geometric action: how does $\Mod^{\pm}(E)$ act on the same data?};
\node[box,below=of d] (e) {5. Compare image and kernel; reconstruct JSJ data only when needed.};
\draw[->,thick] (a)--(b); \draw[->,thick] (b)--(c);
\draw[->,thick] (c)--(d); \draw[->,thick] (d)--(e);
\end{tikzpicture}
\caption{A practical workflow for testing a new knot family.  The
five steps separate surface classification, combinatorics, and
mapping-class-group information, so partial progress at any one stage
remains useful.}
\label{fig:family-workflow}
\end{figure}

The model cases illustrate different entry points.  For torus knots
the census has only two vertices and the decisive step is the peripheral-data test.
For connected sums, Banks's winding coordinate exposes an infinite
cyclic twist direction.  In the hyperbolic case, Mostow--Prasad
rigidity makes the kernel finite before the detailed surface census is
known.  For the figure-eight knot, and more generally for all twist
knots (Section~\ref{sec:two-bridge}), all five steps can be completed
by hand.  Appendix~\ref{app:recipe} turns the same workflow into a
concrete Hatcher--Thurston recipe for two-bridge knots.

\section{Conventions}
\label{sec:definitions}

We work in the PL category.  All $3$-manifolds are compact and
orientable unless otherwise stated; homeomorphisms and isotopies are
PL, and all mapping class groups are $\pi_0$ of the relevant groups of
PL homeomorphisms.  This convention is used only to avoid irrelevant
analytic issues in the patterned-ball and hierarchy statements of
Appendix~\ref{app:atlas}.

\begin{definition}[Knot exterior and peripheral marking]
Let $K\subset S^3$ be a knot; see \cite{Rolfsen} for generalities.  Its
exterior is $E(K)=S^3\setminus\Int N(K)$, with torus boundary.  We
write $\mu$ for the meridian slope and $\lambda$ for the preferred
longitude, characterized by $\lk(K,\lambda)=0$; equivalently $\lambda$
is the unique slope which is null-homologous in $E(K)$.  A boundary
slope is written $p\mu+q\lambda$, or $p/q$ after a meridian--longitude
basis is fixed.  Slopes are unoriented.  For nontrivial $K$ the
exterior is irreducible with incompressible boundary
\cite{Rolfsen,Hempel}.
\end{definition}

\begin{definition}[Mapping class group]
For $E=E(K)$ put $\Mod^{\pm}(E)=\pi_0\Homeo(E)$, allowing
orientation-reversing homeomorphisms, and write $\Mod^{+}(E)$ for
the orientation-preserving subgroup.
\end{definition}

\begin{remark}[The sign rule]
\label{rem:GL}
For a nontrivial knot $K$, every self-homeomorphism of $E(K)$
preserves the meridian slope, by the Gordon--Luecke theorem
\cite{GL}, and preserves the longitude slope, since $\lambda$ is
homologically characterized.  In particular the meridian needs no
separate bookkeeping: it is preserved automatically.  Moreover:
\begin{enumerate}[label=(\arabic*),leftmargin=*]
\item an orientation-preserving homeomorphism preserves the boundary
orientation, so acts on $H_1(\partial E;\mathbb Z)$ with determinant
$+1$; since it preserves $\{\pm\mu\}$ and $\{\pm\lambda\}$ it sends
$(\mu,\lambda)$ to $\pm(\mu,\lambda)$ and therefore \emph{fixes every
unoriented slope};
\item an orientation-reversing homeomorphism acts with determinant
$-1$, hence on unoriented slopes by the involution
\[
  p\mu+q\lambda\longmapsto p\mu-q\lambda ,
  \qquad\text{that is,}\qquad s\longmapsto -s ,
\]
which fixes $\mu$ and $\lambda$.
\end{enumerate}
This dichotomy is the engine of Section~\ref{sec:visibility}, and it
is worth stating in words: \emph{a homeomorphism of a knot exterior
either fixes all slopes or reverses all their signs, and which of the
two happens is determined by its orientation behaviour alone.}
\end{remark}

\begin{corollary}[Boundary-slope obstruction to amphichirality]
\label{cor:chirality}
Let $K$ be a nontrivial knot.  If $K$ is amphichiral, then its
boundary-slope set satisfies $S(K)=-S(K)$.  Since $0\in S(K)$ for
every nontrivial knot (a minimal-genus Seifert surface is essential,
Lemma~\ref{lem:finite-dim}), a nontrivial knot whose boundary-slope
set is not symmetric about the origin is chiral.  In particular a
torus knot $T(p,q)$, whose boundary-slope set is $\{0,pq\}$ with
$pq\ne0$ by Lemma~\ref{lem:torus-classification}, is chiral.
\end{corollary}

\begin{proof}
An orientation-reversing self-homeomorphism of $E$ carries essential
surfaces to essential surfaces and, by Remark~\ref{rem:GL}(2),
negates every boundary slope.
\end{proof}

\begin{definition}[Essential surface]
\label{def:essential}
A properly embedded compact connected surface $F\subset E$ is
\emph{essential} if it is two-sided, orientable, incompressible,
boundary-incompressible, not boundary-parallel, and is neither a
$2$-sphere nor a disk.
\end{definition}

\begin{remark}
\label{rem:essential-basic}
Since $E$ is irreducible, every embedded $2$-sphere bounds a ball;
since $\partial E$ is incompressible for nontrivial $K$, there is no
essential disk.  So excluding spheres and disks discards nothing, and
every essential surface satisfies $\chi(F)\le0$.  As the ambient
manifold is orientable we restrict to two-sided orientable surfaces
throughout; one-sided surfaces are not vertices of the complexes
considered here.  Finally, the boundary components of an essential
surface are disjoint essential curves on the torus $\partial E$, hence
pairwise parallel, so the boundary slope $\Sl(F)$ is well defined.
\end{remark}

\begin{definition}[Boundary pattern]
A \emph{boundary pattern} on a compact $3$-manifold $M$ is a finite
collection $\mathcal P$ of compact connected subsurfaces of $\partial
M$, called faces, such that the intersection of any two distinct faces
is a (possibly empty) union of arcs and circles, any three distinct
faces meet in a finite set, and the induced stratification of
$\partial M$ is a finite cell structure.  A \emph{patterned
homeomorphism} preserves this face structure.  For a compact surface
$Q$ with a boundary pattern we write $\Mod(Q,\mathcal P)$ for the
group of pattern-preserving homeomorphisms modulo pattern-preserving
isotopy.
\end{definition}

\section{\texorpdfstring{The complex $\ES(E)$ and its basic structure}{The complex ES(E) and its basic structure}}
\label{sec:complexes}

\begin{definition}[Essential-surface complex]
\label{def:ES}
The \emph{essential-surface complex} $\ES(E)$ is the simplicial
complex whose vertices are isotopy classes $[F]$ of connected
essential surfaces in $E$, a finite set of distinct vertices spanning
a simplex when the surfaces admit \emph{simultaneously} pairwise
disjoint representatives.
\end{definition}

\begin{proposition}[Flag property]
\label{prop:ES-flag}
For every nontrivial knot exterior $E$, the complex $\ES(E)$ is a flag
complex: if every pair in a finite collection of vertices is joined by
an edge, the whole collection spans a simplex.  Consequently $\ES(E)$
coincides with Schultens's initial surface complex $\mathcal S_0(E)$
\cite{SchultensSurfaceComplex}.
\end{proposition}

\begin{proof}
The exterior $E$ is irreducible and every representative of a vertex
is incompressible.  If $[F_0],\dots,[F_k]$ form a clique in the
$1$-skeleton, each pair can be isotoped to be disjoint, and
Neumann--Swarup's simultaneous-disjointness lemma for incompressible
surfaces in an irreducible $3$-manifold gives representatives of the
whole family which are simultaneously pairwise disjoint
\cite[Lemma~2.3]{NeumannSwarup}.  The same conclusion is compatible
with the least-area theory of Freedman--Hass--Scott \cite{FHS}.  For
the minimal-genus Kakimizu complex the analogous flag property is
proved by Schultens \cite{SchultensSimplyConnected}; see also
Banks \cite[Theorem~1.12]{BanksSplit}.
Schultens's $\mathcal S_0(M)$ is by definition the flag complex on the
disjointness graph, so the two agree.
\end{proof}

\begin{lemma}[Nonemptiness and finite dimension]
\label{lem:finite-dim}
Let $K$ be a nontrivial knot and $E=E(K)$.  Then $\ES(E)\ne\varnothing$
--- indeed a minimal-genus Seifert surface is essential, so
$0\in S(K)$ --- and $\ES(E)$ is finite-dimensional: there is a
constant $h(E)$ with $\dim\ES(E)\le h(E)-1$.
\end{lemma}

\begin{proof}
A minimal-genus Seifert surface $F$ of $K$, properly isotoped into
$E$, is incompressible; since $K$ is nontrivial, $F$ has genus at
least one, hence is not a boundary-parallel annulus, and an
incompressible surface with boundary on the torus $\partial E$ which
is boundary-compressible is a boundary-parallel annulus.  So $F$ is
essential, and its boundary slope is $0$.

For finite dimension, let $\{[F_0],\dots,[F_n]\}$ span a simplex with
simultaneously disjoint representatives.  These are pairwise
non-parallel, since parallel surfaces are isotopic.  By the
Kneser--Haken finiteness theorem \cite{Hempel,JS}, the number of
pairwise disjoint, pairwise non-parallel essential surfaces in the
Haken manifold $E$ is bounded by a constant $h(E)$, so $n+1\le h(E)$.
\end{proof}

\subsection{The slope-layer decomposition}
\label{subsec:layers}

The following lemma is immediate but it governs every computation in
this article, so we isolate it.

\begin{lemma}[Slope layering]
\label{lem:slope-layer}
Let $F,F'$ be essential surfaces in $E$ with $\partial F\ne\varnothing
\ne\partial F'$ and $\Sl(F)\ne\Sl(F')$.  Then $[F]$ and $[F']$ span no
edge of $\ES(E)$.  Consequently every simplex of $\ES(E)$ contains
boundary-bearing vertices of at most one boundary slope.
\end{lemma}

\begin{proof}
Disjoint representatives have disjoint boundaries on the torus
$\partial E$.  Essential simple closed curves on a torus of distinct
slopes have nonzero geometric intersection number, so they cannot be
disjoint.
\end{proof}

\begin{notation}[Layers]
\label{not:layers}
For $s\in\mathbb Q\cup\{\infty\}$ let $\ES_s(E)$ be the full
subcomplex of $\ES(E)$ spanned by vertices of boundary slope $s$, and
let $\ES_{\mathrm{cl}}(E)$ be the full subcomplex spanned by closed
vertices.  We call $\ES_s(E)$ the \emph{slope-$s$ layer}.  By
Lemma~\ref{lem:slope-layer},
\[
  \ES(E)\;=\;\bigcup_{s\in S(K)}
  \bigl\langle \ES_s(E)\cup\ES_{\mathrm{cl}}(E)\bigr\rangle ,
\]
the union being over the boundary-slope set $S(K)$ and
$\langle\,\cdot\,\rangle$ denoting the full subcomplex; distinct
layers meet only inside $\ES_{\mathrm{cl}}(E)$.  See
Figure~\ref{fig:layers}.
\end{notation}

\begin{figure}[htb]
\centering
\begin{tikzpicture}[scale=1.0]
\foreach \y/\s/\c in {2.0/{s_1}/surfA, 0.9/{s_2}/surfB, -0.2/{s_3}/surfC}{
  \fill[\c,opacity=0.10] (-0.2,\y-0.35) rectangle (5.0,\y+0.35);
  \draw[\c,opacity=0.55] (-0.2,\y-0.35) rectangle (5.0,\y+0.35);
  \node[left,\c] at (-0.35,\y) {$\ES_{\s}(E)$};
}
% vertices in layers with some edges inside layers
\fill[surfA] (0.6,2.0) circle (2pt); \fill[surfA] (1.6,2.15) circle (2pt);
\fill[surfA] (2.6,1.9) circle (2pt);
\draw[surfA,thick] (0.6,2.0)--(1.6,2.15)--(2.6,1.9)--(0.6,2.0);
\fill[surfB] (0.9,0.9) circle (2pt); \fill[surfB] (2.0,0.95) circle (2pt);
\draw[surfB,thick] (0.9,0.9)--(2.0,0.95);
\fill[surfC] (1.2,-0.2) circle (2pt);
% closed part
\fill[greyish,opacity=0.10] (5.4,-0.6) rectangle (7.0,2.4);
\draw[greyish,opacity=0.6] (5.4,-0.6) rectangle (7.0,2.4);
\node[greyish,align=center,font=\small] at (6.2,2.98) {closed\\ vertices};
\fill[greyish] (6.2,1.5) circle (2pt);
\fill[greyish] (6.2,0.4) circle (2pt);
\draw[greyish,thick] (6.2,1.5)--(6.2,0.4);
% edges from closed to layers (allowed)
\draw[greyish,thick,densely dashed] (6.2,1.5) -- (2.6,1.9);
\draw[greyish,thick,densely dashed] (6.2,0.4) -- (2.0,0.95);
% forbidden edge
\draw[red!70!black,very thick,densely dotted] (2.6,1.9) -- (2.0,0.95);
\node[red!70!black,font=\small,right] at (2.85,1.45) {forbidden};
\end{tikzpicture}
\caption{The slope-layer decomposition (Lemma~\ref{lem:slope-layer}).
Edges may occur inside a single slope layer, and between a closed
vertex and anything, but never between two boundary-bearing vertices
of different slopes.  All the examples of
Section~\ref{sec:firstlook} have singleton layers and no closed
vertices, which is why they are discrete.}
\label{fig:layers}
\end{figure}
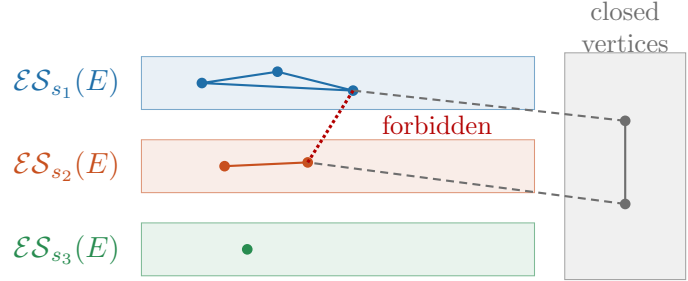

\begin{remark}[Layer preservation]
\label{rem:layer-preservation}
By Remark~\ref{rem:GL}, the natural action preserves this structure:
an orientation-preserving mapping class carries $\ES_s(E)$ to itself
for every $s$, while an orientation-reversing one carries $\ES_s(E)$
to $\ES_{-s}(E)$; both preserve $\ES_{\mathrm{cl}}(E)$.  In particular
$\Image\Phi_{\ES}$ acts on the set $S(K)$ of slopes through a group of
order at most $2$.
\end{remark}

\subsection{Unbounded dimension}
\label{subsec:dimension}

The examples of Section~\ref{sec:firstlook} are discrete, and the
cable example already has an edge.  The dimension is in fact
unbounded.  Recall that for a nontrivial knot $J$ and coprime
$(p,q)$ with $p\ge2$, the cable $C_{p,q}(J)$ is the curve on
$\partial V'$ running $p$ times longitudinally and $q$ times
meridionally, where $V'\subset\Int N(J)$ is concentric; the exterior
decomposes along the JSJ torus $\partial N(J)$ into $E(J)$ and a cable
space \cite{Budney,JS}.

\begin{proposition}[Simplices of every dimension]
\label{prop:dimension}
Fix $n\ge1$.  Let $J$ be a nontrivial knot and define an iterated
cable
\[
  K_0=J,\qquad K_i=C_{p_i,q_i}(K_{i-1})\quad(1\le i\le n),
\]
with each cabling nontrivial, that is $p_i\ge2$.  Then in
$E=E(K_n)$ the $n$ tori
\[
  T^{(i)}=\partial N(K_{i-1})\qquad(1\le i\le n),
\]
each of which is a JSJ torus of $E$, together with the outermost
cabling annulus $A\subset E$ are
simultaneously disjoint, pairwise non-isotopic essential surfaces.
Hence they span an $n$-simplex and
\[
  \dim\ES\bigl(E(K_n)\bigr)\ \ge\ n .
\]
In particular $\sup_K\dim\ES(E(K))=\infty$, while each individual
$\dim\ES(E)$ is finite by Lemma~\ref{lem:finite-dim}.
\end{proposition}

\begin{proof}
The exteriors $E(K_i)$ are satellite exteriors and their JSJ
decomposition is the standard one: $E(K_n)$ is obtained from $E(J)$ by
successively gluing cable spaces $CS_1,\dots,CS_n$ along the tori
$T^{(1)},\dots,T^{(n)}$, where $CS_i$ lies between $T^{(i)}$ and
$T^{(i+1)}$ (and $CS_n$ lies between $T^{(n)}$ and $\partial E$)
\cite{Budney,JS,Johannson}.  These tori are essential and pairwise
non-parallel, since no $CS_i$ is a product $T^2\times I$ --- each is
Seifert fibered over an annulus with one exceptional fibre of order
$p_i\ge2$.  They are nested, hence pairwise disjoint.  The outermost
cabling annulus $A=\partial V'\setminus\Int N(K_n)$ is properly
embedded in $E$ with $\partial A\subset\partial E$ of slope $p_nq_n$;
it is the frontier of the characteristic Seifert piece of the
outermost cable space and is essential because $K_{n-1}$ is nontrivial
\cite{Budney}.  The tori are nested, with $T^{(1)}=\partial N(J)$
outermost and $T^{(n)}=\partial N(K_{n-1})$ innermost, and $A$ lies on
the torus $\partial V'$ which, apart from $\partial A\subset\partial E$,
is contained in the interior of the outermost cable space $CS_n$;
hence $A$ may be taken disjoint from every $T^{(i)}$.  Finally $A$
is not isotopic to any $T^{(i)}$, being a surface with nonempty
boundary.  Thus the $n+1$ classes are distinct vertices with
simultaneously disjoint representatives, and by
Definition~\ref{def:ES} they span an $n$-simplex.
\end{proof}

\begin{figure}[htb]
\centering
\begin{tikzpicture}[scale=0.95]
% nested tori cross section for n=2
\draw[surfC,very thick] (0,0) ellipse (3.0 and 1.75);
\node[surfC] at (2.8,1.38) {\small $T^{(1)}$};
\draw[surfA,very thick] (0,0) ellipse (2.15 and 1.28);
\node[surfA] at (2.02,0.93) {\small $T^{(2)}$};
\draw[surfB,very thick] (0,0) ellipse (1.35 and 0.82);
\node[surfB] at (0.85,0.88) {\small $A$};
\fill[white] (1.35,0) circle (0.19); \draw[thick] (1.35,0) circle (0.19);
\fill[white] (-1.35,0) circle (0.19); \draw[thick] (-1.35,0) circle (0.19);
\node[font=\small] at (0,-2.25) {nested surfaces in $E(K_2)$};
% the 2-simplex
\begin{scope}[xshift=6.6cm,yshift=-0.2cm]
  \coordinate (a) at (0,1.0);
  \coordinate (b) at (-0.95,-0.6);
  \coordinate (c) at (0.95,-0.6);
  \fill[greyish,opacity=0.18] (a)--(b)--(c)--cycle;
  \draw[thick] (a)--(b)--(c)--cycle;
  \fill[surfB] (a) circle (2.4pt); \node[above,surfB] at (a) {\small $[A]$};
  \fill[surfC] (b) circle (2.4pt); \node[below left,surfC] at (b) {\small $[T^{(1)}]$};
  \fill[surfA] (c) circle (2.4pt); \node[below right,surfA] at (c) {\small $[T^{(2)}]$};
  \node[font=\small] at (0,-2.05) {a $2$-simplex in $\ES(E(K_2))$};
\end{scope}
\end{tikzpicture}
\caption{A doubly iterated cable.  The two JSJ tori and the outermost
cabling annulus are nested, hence simultaneously disjoint, and span a
$2$-simplex.  Iterating $n$ times gives an $n$-simplex
(Proposition~\ref{prop:dimension}).}
\label{fig:iterated}
\end{figure}
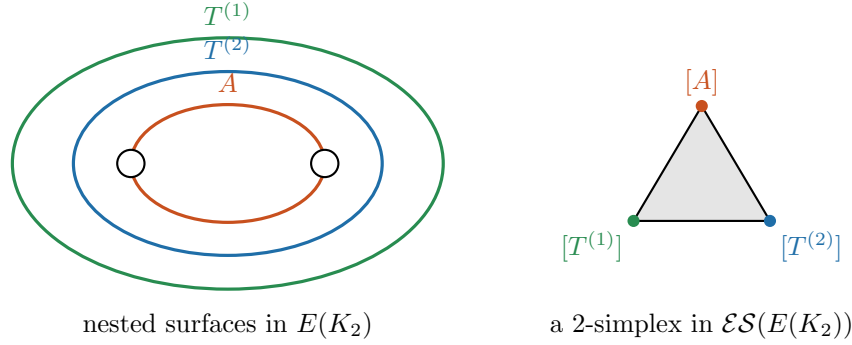

\begin{remark}
\label{rem:slope-sep-rare}
Proposition~\ref{prop:dimension} also shows that slope-separatedness
(Definition~\ref{def:slope-separated-intro}) is a genuine restriction:
a satellite exterior contains closed essential surfaces, so no
satellite knot is slope-separated.  Nor is the condition automatic for
hyperbolic knots.  By Hatcher--Thurston, a two-bridge knot whose
minimal continued-fraction expansion has at least two odd coefficients
carries connected $n$-sheeted essential surfaces for every $n$
\cite[Proposition~1]{HatcherThurston}, and one of its slope layers is
then infinite (Proposition~\ref{prop:two-bridge-finite}).  In fact the
slope-separated two-bridge knots are exactly the torus knots $T(2,q)$
and the twist knots
(Theorem~\ref{thm:two-bridge-slope-separated}), so the three-vertex
behaviour of $4_1$ is shared by precisely one infinite family.  What
survives in general is
Lemma~\ref{lem:slope-layer} and
Remark~\ref{rem:layer-preservation}: the layer structure and its
preservation by the action.
\end{remark}

\section{Three levels of structure}
\label{sec:decorations}

Section~\ref{sec:firstlook} already showed that the bare complex can
have symmetries which no homeomorphism realizes, while a very small
amount of peripheral information may remove them.  To keep the
bookkeeping light, this article uses only three named objects
(Figure~\ref{fig:ladder}):
\[
  \ES(E),\qquad \ES^{\mathrm{per}}(E),\qquad \CES(E).
\]
The first remembers only disjointness, the second adds the two pieces
of peripheral information that recur in the examples, and the third
adds characteristic-submanifold/JSJ data.  Type-only and slope-only
views are occasionally useful as diagnostics, but we do not promote
them to additional levels of the theory.

\begin{definition}[Peripheral refinement]
\label{def:peripheral}
The \emph{peripheral essential-surface complex}
$\ES^{\mathrm{per}}(E)$ is $\ES(E)$ together with two vertex data:
\begin{itemize}[leftmargin=*]
\item the homeomorphism type of the representing surface $F$, including
its genus and number of boundary components;
\item the boundary slope $\Sl(F)$ when $\partial F\ne\varnothing$.
\end{itemize}
The meridian $\mu$ belongs to the ambient peripheral coordinate system,
not to the vertex data.  It is therefore suppressed from the notation
$\ES^{\mathrm{per}}(E)$.
\end{definition}

\begin{lemma}[The optional homology refinement is binary]
\label{lem:homology-binary}
Let $F\subset E(K)$ be a connected two-sided orientable properly
embedded surface in the exterior of a nontrivial knot.  Then
\[
  [F]=0\in H_2(E,\partial E;\mathbb Z)
  \quad\Longleftrightarrow\quad
  F\text{ separates }E.
\]
If $F$ is nonseparating, then $[F]$ is primitive; since
$H_2(E,\partial E;\mathbb Z)\cong\mathbb Z$, its unoriented class is
the unique nonzero primitive class.
\end{lemma}

\begin{proof}
If $F$ separates, the closure of either complementary component gives
a relative $3$-chain whose boundary is $F$ together with a subsurface
of $\partial E$, so $[F]=0$.  Conversely, if $F$ is nonseparating,
choose an arc crossing $F$ once from one side to the other and close it
up through $E\setminus F$.  The resulting loop has algebraic
intersection $\pm1$ with $F$.  By Poincar\'e--Lefschetz duality this
pairs $[F]$ to $\pm1$, so $[F]$ is nonzero and primitive.
\end{proof}

\begin{remark}[Why homology is not a standard decoration here]
\label{rem:homology-redundancy}
Lemma~\ref{lem:homology-binary} shows that, for connected surfaces in a
knot exterior, the unoriented relative homology class $\pm[F]$ carries
only the binary information ``separating or nonseparating.''  It is a
natural optional refinement, especially in broader classes of
$3$-manifolds, but including it in the standard notation would add more
bookkeeping than information in the present setting.  We therefore do
not build it into $\ES^{\mathrm{per}}$.

There is also a useful peripheral consequence.  If $F$ is
nonseparating, then $[F]$ is a generator and the boundary map
$H_2(E,\partial E)\to H_1(\partial E)$ sends $[F]$ to
$\pm\lambda$.  Since all components of $\partial F$ have one common
slope, a nonseparating essential surface must therefore have longitude
slope $0$.  Hence every essential surface of nonzero boundary slope is
separating.
\end{remark}

\begin{remark}[Two diagnostic partial refinements]
\label{rem:partial-refinements}
In a concrete example one may temporarily keep only the surface type,
or only the boundary slope, to see which datum eliminates a false
symmetry.  We use this comparison for torus knots and the figure-eight
knot.  These are deliberately treated as \emph{partial views} of
$\ES^{\mathrm{per}}$, not as additional named objects in the main
framework.
\end{remark}

The characteristic refinement $\CES(E)$ is defined in
Section~\ref{sec:satellite}; it adds the characteristic frontier,
support, boundary-pattern, and rooted JSJ incidence data.  At the
opposite extreme, Appendix~\ref{app:atlas} defines an intentionally
over-marked hierarchy atlas as a reconstruction benchmark only.  The
latter is kept outside the main three-level framework.

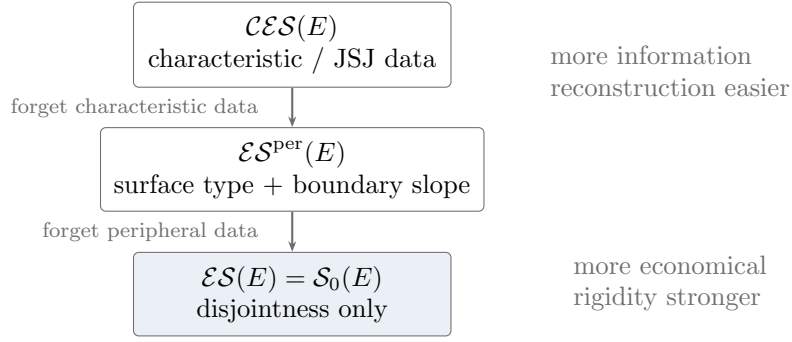
\begin{figure}[htb]
\centering
\begin{tikzpicture}[
  every node/.style={font=\small},
  box/.style={draw=greyish,rounded corners=2pt,inner sep=5pt,align=center,
             minimum width=42mm},
  ar/.style={-{Stealth[length=4pt]},greyish,thick}]
\node[box] (CES) at (0,2.25) {$\CES(E)$\\ characteristic / JSJ data};
\node[box] (ESp) at (0,0.6) {$\ES^{\mathrm{per}}(E)$\\ surface type + boundary slope};
\node[box,fill=lightfill] (ES) at (0,-1.05) {$\ES(E)=\mathcal S_0(E)$\\ disjointness only};
\draw[ar] (CES) -- node[left=4mm,font=\scriptsize,fill=white,inner sep=1pt]{forget characteristic data} (ESp);
\draw[ar] (ESp) -- node[left=4mm,font=\scriptsize,fill=white,inner sep=1pt]{forget peripheral data} (ES);
\node[greyish,align=left] at (5.0,1.9)
  {more information\\ reconstruction easier};
\node[greyish,align=left] at (5.0,-0.9)
  {more economical\\ rigidity stronger};
\end{tikzpicture}
\caption{The three named levels used in the article.  The central
question is how far down one can move while still controlling image,
kernel, and reconstruction.  Type-only and slope-only comparisons are
used locally as diagnostics, and the over-marked hierarchy atlas of
Appendix~\ref{app:atlas} is a separate benchmark rather than a fourth
level.}
\label{fig:ladder}
\end{figure}

\begin{definition}[Automorphisms at the refined levels]
\label{def:automorphisms}
An automorphism of $\ES^{\mathrm{per}}(E)$ is a simplicial
automorphism $\varphi$ of $\ES(E)$ which preserves surface type and
for which one global sign $\varepsilon\in\{\pm1\}$ satisfies
\[
  \Sl(\varphi[F])=\varepsilon\,\Sl([F])
\]
for every boundary-bearing vertex.  Once $\CES(E)$ is defined in
Section~\ref{sec:satellite}, its automorphisms are those peripheral
automorphisms which also preserve the characteristic frontier,
support, boundary-pattern, and JSJ-incidence data listed there.  We
write $\Aut(X)$ for the automorphism group at any of the three levels.
\end{definition}

There are evident forgetful maps
\[
  \CES(E)\longrightarrow\ES^{\mathrm{per}}(E)
  \longrightarrow\ES(E),
\]
all identity on the underlying vertex set.  Type-only and slope-only
views will occasionally be used as diagnostics, but not as additional
objects in the main notation.

\begin{lemma}[The natural action]
\label{lem:natural-action}
At each of the three levels, every self-homeomorphism of $E$ induces
an automorphism.  Thus, for $X\in\{\ES(E),\ES^{\mathrm{per}}(E),\CES(E)\}$,
there is a homomorphism
\[
  \Phi_X:\Mod^{\pm}(E)\longrightarrow\Aut(X).
\]
We write $\mathcal K_X(E)=\ker\Phi_X$.
\end{lemma}

\begin{proof}
A homeomorphism preserves essentiality, isotopy, simultaneous
disjointness, and surface homeomorphism type.  By
Remark~\ref{rem:GL}, it either fixes every unoriented boundary slope or
negates all of them.  The characteristic submanifold is unique up to
isotopy \cite{JS,Johannson}, so the characteristic data are preserved
as well.
\end{proof}

\begin{figure}[htb]
\centering
\begin{tikzcd}[column sep=1.5cm,row sep=0.6cm]
1 \arrow[r] & \mathcal K_X(E) \arrow[r] & \Mod^{\pm}(E)
  \arrow[r,"\Phi_X"] & \Aut\bigl(X\bigr) \\
& \text{\small\emph{invisible}} & \text{\small\emph{all symmetries}} &
  \text{\small\emph{combinatorial}}
\end{tikzcd}
\caption{The bookkeeping of the whole paper.  ``Kernel'' asks what is
lost, ``image'' asks what is gained, and the reconstruction question
asks whether the right-hand object determines $E$ at all.  The
sequence is exact at $\mathcal K_X(E)$ and at $\Mod^{\pm}(E)$ by
construction; exactness on the right --- surjectivity of $\Phi_X$ ---
is the hard Ivanov-type statement, and it fails already for
$X=\ES(E)$ by Theorem~\ref{thmA}.}
\label{fig:bookkeeping}
\end{figure}

\section{Visibility for slope-separated knots}
\label{sec:visibility}

The argument in this section uses only two facts: disjoint
boundary-bearing surfaces lie in one slope layer
(Lemma~\ref{lem:slope-layer}), and a self-homeomorphism of a knot
exterior either fixes every unoriented slope or sends every slope to
its negative (Remark~\ref{rem:GL}).

Let $K$ be slope-separated, $E=E(K)$, and write
\[
  \ES(E)=\{v_s:s\in S(K)\},\qquad \Sl(v_s)=s.
\]
For each $s$ let $\tau(s)$ be the homeomorphism type of the unique
surface represented by $v_s$.

\begin{proposition}[Combinatorial automorphisms]
\label{prop:aut-slope-separated}
If $b=|S(K)|$, then $\Aut(\ES(E))\cong S_b$.  Moreover
$\Aut(\ES^{\mathrm{per}}(E))$ is generated by the sign involution
$v_s\mapsto v_{-s}$ precisely when
\[
  S(K)=-S(K)\ne\{0\}
  \quad\text{and}\quad
  \tau(-s)=\tau(s)\ \text{for every }s\in S(K);
\]
otherwise $\Aut(\ES^{\mathrm{per}}(E))$ is trivial.
\end{proposition}

\begin{proof}
The bare complex is a discrete set of $b$ vertices.  At the peripheral
level, the slope datum forces a single global sign; the negative sign
is allowed exactly under the displayed conditions.
\end{proof}

\begin{theorem}[$=$ Theorem~\ref{thmA}]
\label{thm:visibility}
Let $K$ be a nontrivial slope-separated knot.  For
$X=\ES(E)$ or $\ES^{\mathrm{per}}(E)$, and also for $X=\CES(E)$ when
$\Fr\Char(E)=\varnothing$, every orientation-preserving mapping class
acts trivially:
\[
  \Mod^{+}(E)\subseteq\mathcal K_X(E).
\]
Furthermore
\[
  \Image\Phi_X\cong
  \begin{cases}
    \mathbb Z/2, & \text{if $K$ is amphichiral and $S(K)\ne\{0\}$},\\[2pt]
    1, & \text{otherwise.}
  \end{cases}
\]
The nontrivial element is $v_s\mapsto v_{-s}$.  For the bare complex,
if $b=|S(K)|$, then
\[
  \Aut(\ES(E))\cong S_b,
  \qquad [\Aut(\ES(E)):\Image\Phi_{\ES}]\ge b!/2.
\]
If $K$ is amphichiral then $S(K)=-S(K)$ and $b$ is odd; consequently
an even number of boundary slopes forces $K$ to be chiral and the bare
action to be trivial.
\end{theorem}

\begin{proof}
An orientation-preserving homeomorphism fixes every slope, hence fixes
every vertex because the slopes separate vertices.  An
orientation-reversing homeomorphism sends $v_s$ to $v_{-s}$; this is
nontrivial exactly when a nonzero slope occurs.  The image therefore
has order at most two and is nontrivial precisely in the amphichiral
case above.  The statement about $\Aut(\ES(E))$ follows from
Proposition~\ref{prop:aut-slope-separated}.  If
$\Fr\Char(E)=\varnothing$, then the additional characteristic data
in $\CES(E)$ are trivial: there are no frontier components, every
vertex has support in the single ambient piece, and the rooted JSJ
incidence data have one vertex.  Hence the same vertex calculation
gives the stated conclusion for $X=\CES(E)$ as well.

For the parity statement, the involution $s\mapsto-s$ fixes exactly
$0$ and the meridian slope $\infty$.  The latter is excluded by
slope-separatedness, so on the finite set $S(K)$ only $0$ is fixed and
all remaining slopes occur in pairs $\{s,-s\}$.
\end{proof}

\begin{remark}
The theorem does not require the mapping class group to be known in
advance.  It also shows that invisible elements need not be twists:
for a hyperbolic slope-separated knot there are no essential annuli or
tori, yet the entire orientation-preserving symmetry group is
invisible.  Thus finite invisible symmetries are a genuine phenomenon,
not an artefact of missing twist coordinates.
\end{remark}

\begin{corollary}[Chirality of torus knots, from surfaces]
\label{cor:torus-chiral}
A torus knot $T(p,q)$ is chiral; its action on $\ES(E)$ is trivial,
while $\Aut(\ES(E))\cong\mathbb Z/2$.
\end{corollary}

\begin{proof}
By Lemma~\ref{lem:torus-classification}, $S(T(p,q))=\{0,pq\}$ has two
elements, so Theorem~\ref{thm:visibility} applies.
\end{proof}

\begin{figure}[htb]
\centering
\begin{tikzpicture}[scale=1.0,every node/.style={font=\small}]
\foreach \x/\lab/\col in {0/{bare}/greyish, 3.2/{type only}/surfA,
                          6.4/{slope only}/surfB}{
  \node[\col] at (\x+1.1,2.15) {\lab};
}
\foreach \x/\c in {0/greyish, 3.2/surfA, 6.4/surfB}{
  \fill[\c] (\x+0.25,1.35) circle (2.2pt);
  \fill[\c] (\x+1.1,1.35) circle (2.2pt);
  \fill[\c] (\x+1.95,1.35) circle (2.2pt);
}
\node[greyish,align=center] at (1.1,0.55) {no labels\\ $\Aut\cong S_3$};
\node[surfA,align=center] at (4.3,0.55) {surface types\\ $\Aut\cong\mathbb Z/2$};
\node[surfB,align=center] at (7.5,0.55) {slopes $0,+4,-4$\\ $\Aut\cong\mathbb Z/2$};
\node[surfA] at (3.45,1.72) {\tiny $(1,1)$};
\node[surfA] at (4.3,1.72) {\tiny $(1,2)$};
\node[surfA] at (5.15,1.72) {\tiny $(1,2)$};
\node[surfB] at (6.65,1.72) {\tiny $0$};
\node[surfB] at (7.5,1.72) {\tiny $+4$};
\node[surfB] at (8.35,1.72) {\tiny $-4$};
\draw[greyish] (-0.2,-0.35) -- (9.0,-0.35);
\node[align=center] at (4.4,-0.85)
  {for the figure-eight knot, either elementary datum already\\
   removes every false symmetry};
\end{tikzpicture}
\caption{Two diagnostic partial views for the figure-eight knot.  They
are not additional levels of the theory; they simply show which part
of the peripheral information removes the false permutations.  As in
Remark~\ref{rem:invariant}, slopes are read up to the global sign
change $s\mapsto-s$, which is why the slope-only view retains the
transposition $+4\leftrightarrow-4$.}
\label{fig:decorations-fig8}
\end{figure}

\begin{corollary}[Peripheral surjectivity criterion]
\label{cor:surjectivity}
For a slope-separated knot, the action on $\ES^{\mathrm{per}}(E)$
fails to be surjective exactly when $K$ is chiral but
\[
  S(K)=-S(K)\ne\{0\},\qquad \tau(-s)=\tau(s)\ \text{for all }s.
\]
\end{corollary}

\begin{question}[A chiral knot with symmetric peripheral data]
\label{q:chiral-symmetric}
Does there exist a slope-separated chiral knot satisfying the two
conditions in Corollary~\ref{cor:surjectivity}?  Such an example would
show that even type plus slope can retain a nongeometric symmetry.
By Corollary~\ref{cor:two-bridge-question} there is no such example
among two-bridge knots.  See also Problem~\ref{prob:chiral-symmetric}.
\end{question}

\begin{remark}[$\ES^{\mathrm{per}}$ as a finite list]
\label{rem:invariant}
For a slope-separated knot, $\ES^{\mathrm{per}}(E)$ is simply
\[
  \{(s,\tau(s)):s\in S(K)\}
\]
up to the global sign change $s\mapsto-s$.  The reconstruction
question therefore begins with a very concrete finite datum.
\end{remark}

\section{Torus knots}
\label{sec:torus-knots}

We now redo the torus-knot example carefully.  Nothing in the surface
census is new: the direct classification of incompressible surfaces in
torus-knot exteriors is due to Tsau \cite{TsauTorusSurfaces}, and
Schultens's analysis of Kakimizu complexes of Seifert fibered spaces
is a closely related complex-theoretic precedent
\cite{SchultensKakimizuSeifert}.  What we add is that, once the census
is in hand, the image and kernel are immediate from
Theorem~\ref{thm:visibility}.

Throughout, $K=T(p,q)$ is a nontrivial torus knot with $1<p<q$ and
$\gcd(p,q)=1$ --- every nontrivial torus knot is of this form up to
mirror image --- and $E=E(T(p,q))$.  Then $E$ is Seifert fibered over
a disk with two exceptional fibres of orders $p$ and $q$; the
characteristic submanifold is all of $E$ and the characteristic
frontier is empty \cite{JS,Johannson}.  We write $F_0$ for the fibre
surface, a once-punctured surface of genus $(p-1)(q-1)/2$ with
boundary slope $0$, and $A$ for the cabling annulus, of boundary slope
$pq$ \cite{Rolfsen,HatcherNotes}.

\begin{lemma}[Classical census]
\label{lem:torus-classification}
$\ES(E)$ consists of exactly the two vertices $[F_0]$ and $[A]$, with
no edge.  In particular $T(p,q)$ is slope-separated with
$S=\{0,pq\}$.
\end{lemma}
\begin{proof}
The base-orbifold argument was given in Section~\ref{sec:firstlook}.
The standard classification of essential surfaces in Seifert-fibered
spaces says that every essential surface is vertical or horizontal
\cite{HatcherNotes,Waldhausen1967,TsauTorusSurfaces}.  In
$D^2(p,q)$ there is one essential vertical arc separating the two cone
points, whose preimage is the cabling annulus $A$; the horizontal
class is the fibre $F_0$.  Their slopes are $pq$ and $0$, so they do
not span an edge.  No closed or meridional essential surface occurs.
\end{proof}

\begin{lemma}[Classical mapping class group]
\label{lem:torus-mcg}
$\Mod^{\pm}(E)\cong\mathbb Z/2$, generated by the
class of the strong inversion $\sigma$, and every mapping class of $E$
is orientation-preserving.
\end{lemma}

\begin{proof}
The group
$G=\pi_1E\cong\langle x,y\mid x^p=y^q\rangle$ has $\Out(G)\cong
\mathbb Z/2$, generated by the inversion $x\mapsto x^{-1},\
y\mapsto y^{-1}$; this is due to Schreier \cite{Schreier}.  Since $E$
is Haken, Waldhausen's theorems \cite{Waldhausen} imply that homotopic
homeomorphisms are isotopic, so $\pi_0\Homeo(E)\to\Out(G)$ is
injective, and the nontrivial outer class is realized by the strong
inversion $\sigma$ of the torus knot.  Both classes contain
orientation-preserving representatives ($\mathrm{id}$ and $\sigma$).
No class is orientation-reversing: by Remark~\ref{rem:GL}(2) an
orientation-reversing homeomorphism would carry the cabling annulus,
of boundary slope $pq$, to an essential surface of slope $-pq$, whereas
$S(K)=\{0,pq\}$ by Lemma~\ref{lem:torus-classification} and $pq\ne0$.
(A comparison of free homotopy classes alone would not suffice here,
since a free homotopy need not preserve $\partial E$ and so does not
act on $H_3(E,\partial E;\mathbb Z)$.)  For classical
computations of mapping class groups of elementary $3$-manifolds
compare \cite{Bonahon,McCulloughMillerZimmermann,Johannson}.
\end{proof}

\begin{corollary}[Torus-knot action calculation]
\label{cor:torus-final}
Let $E=E(T(p,q))$.  For each of
\[
  X=\ES(E),\qquad \ES^{\mathrm{per}}(E),\qquad \CES(E),
\]
the natural action is trivial:
\[
  \mathcal K_X(E)=\Mod^{\pm}(E)=\langle[\sigma]\rangle\cong
  \mathbb Z/2,
  \qquad \Image\Phi_X=1.
\]
Moreover
\[
  \Aut(\ES^{\mathrm{per}}(E))=
  \Aut(\CES(E))=1,
\]
whereas $\Aut(\ES(E))\cong\mathbb Z/2$, generated by a transposition
induced by no homeomorphism.
\end{corollary}

\begin{proof}
By Lemma~\ref{lem:torus-classification} the knot is slope-separated
with $S=\{0,pq\}$, and it is chiral by
Corollary~\ref{cor:torus-chiral}.  Its characteristic frontier is
empty.  Theorem~\ref{thm:visibility} therefore gives the trivial image
and full kernel at all three levels.  The bare complex has two
vertices, so its automorphism group is $S_2$.  At the peripheral level
the two vertices have both different slopes and different surface
types (annulus versus positive genus), so no transposition is allowed;
the characteristic level adds no further nontrivial data in this
Seifert-fibered one-piece case.
\end{proof}

\begin{remark}[Three morals]
\label{rem:morals}
First, the kernel is already nontrivial in the simplest example, and
its generator is a finite symmetry rather than a twist; no refinement
by annular or toral twist coordinates can remove it.  Second, this is
analogous to the low-complexity hyperelliptic exceptions of the
classical theory \cite{Korkmaz,Luo}: the object is too small to detect
the whole mapping class group.  Third, the example bounds the necessary peripheral information from both sides.  The bare complex is too weak, admitting a nongeometric transposition, but either surface type alone or boundary slope alone already removes it.  This is why the main framework combines only these two elementary data and does not introduce further standard decorations.
\end{remark}

\section{The figure-eight knot}
\label{sec:figure-eight}

The figure-eight knot is the smallest hyperbolic test of the
image--kernel problem, and the first example in which the geometric
image is nontrivial.  Neither the surface census nor the symmetry
group is new: the former is extracted from Hatcher--Thurston
\cite{HatcherThurston}, and the latter, with its action, is described
in Heusener--Mu\~noz--Porti \cite[Section~9]{HeusenerMunozPorti}; see
also \cite{SakumaSurvey}.  The incompressible surfaces in the
figure-eight complement are also described in Thurston's notes
\cite[Chapter~4]{ThurstonNotes}.  Hatcher and Thurston write the figure-eight
knot as $K_{3/5}$; the continued-fraction convention used inside the
proof of Lemma~\ref{lem:figure-eight-census} is the subtractive one
fixed in Step~1 of Recipe~\ref{rec:main}.

We first record the general hyperbolic input, which is used again in
Section~\ref{sec:satellite}.

\begin{corollary}[Classical hyperbolic consequence]
\label{cor:hyperbolic-kernel}
Let $K$ be a hyperbolic knot and $E=E(K)$.  Then $\Mod^{\pm}(E)$
is finite and isomorphic to the isometry group of the complete
hyperbolic structure on $\Int E$; $E$ contains no essential annulus
or torus, so $\Fr\Char(E)=\varnothing$; and for each of the three objects $X$ the kernel $\mathcal K_X(E)$ is the finite group of
isometry classes acting trivially on $X$.  Consequently, if every
automorphism of $X$ is geometric then $\Aut(X)$ is
finite.
\end{corollary}

\begin{proof}
By Mostow--Prasad rigidity and Waldhausen's theorem the mapping class
group of a hyperbolic knot exterior is the isometry group of the
complete finite-volume structure, which is finite
\cite{Thurston,Waldhausen,Budney}.  A hyperbolic exterior is atoroidal
and anannular, so its characteristic submanifold is a collar of
$\partial E$, whose frontier consists of a boundary-parallel torus;
this is deleted as redundant in
Theorem~\ref{thm:classical-jsj-input}(2), so
$\Fr\Char(E)=\varnothing$.  In particular no essential annulus or
torus is present and the characteristic data are trivial.  The remaining statements are
immediate.
\end{proof}

\begin{lemma}[Classical census]
\label{lem:figure-eight-census}
Let $E_8=E(4_1)$.  Then $\ES(E_8)$ has exactly three vertices,
represented by
\begin{enumerate}[label=(\roman*),leftmargin=*]
\item a once-punctured torus $F$ of boundary slope $0$;
\item two twice-punctured tori $S_+$ and $S_-$ of boundary slopes
$+4$ and $-4$.
\end{enumerate}
All three are isolated, $F$ is nonseparating and represents a
generator of $H_2(E_8,\partial E_8;\mathbb Z)$ up to sign, and
$S_\pm$ are separating and represent zero.  In particular $4_1$ is
slope-separated with $S=\{0,+4,-4\}$.
\end{lemma}
\begin{proof}
This is the Hatcher--Thurston calculation recalled pictorially in
Section~\ref{sec:firstlook}.  For $K_{3/5}=4_1$, the three relevant
minimal continued-fraction expansions are
\[
  [-2,2],\qquad [2,3],\qquad [-3,-2],
\]
with boundary slopes $0,+4,-4$ \cite{HatcherThurston}.  The all-even
expansion gives the orientable once-punctured-torus fibre $F$.  The two
odd expansions first give one-sided once-punctured Klein bottles; the
frontiers of their regular neighbourhoods are the unique orientable
two-sheeted classes $S_+$ and $S_-$, each a twice-punctured torus.
Hatcher--Thurston's count gives no further connected orientable
essential surfaces, and there is no closed essential surface.  The
three slopes are distinct, so no two vertices are adjacent.  Thus
$4_1$ is slope-separated with $S=\{0,+4,-4\}$.
\end{proof}

\begin{lemma}[Peripheral action of a strong inversion]
\label{lem:strong-inversion-sign}
Let $K$ admit a strong inversion, that is, an orientation-preserving
involution $\iota$ of $(S^3,K)$ whose fixed-point set is an unknotted
circle meeting $K$ in two points.  Then $\iota|_{E(K)}$ is an
orientation-preserving mapping class acting on
$H_1(\partial E(K);\mathbb Z)$ by $-\mathrm{id}$ in the ordered
oriented basis $(\mu,\lambda)$.
\end{lemma}

\begin{proof}
The involution reverses the orientation of $K$.  Orient $\mu$ by
$\lk(\mu,K)=+1$.  Since $\iota$ preserves the orientation of $S^3$ it
preserves linking numbers, so $\lk(\iota(\mu),\iota(K))=+1$; and since
$\iota(K)$ is $K$ reversed, $\lk(\iota(\mu),K)=-1$, whence
$\iota(\mu)=-\mu$.  As $\iota|_{E(K)}$ is orientation-preserving,
Remark~\ref{rem:GL} gives $(\mu,\lambda)\mapsto\pm(\mu,\lambda)$, and
the meridian determines the sign.
\end{proof}

\begin{lemma}[Classical symmetry group]
\label{lem:figure-eight-action}
$\Mod^{\pm}(E_8)$ is dihedral of order eight, with
orientation-preserving subgroup
$\Mod^{+}(E_8)\cong\mathbb Z/2\oplus\mathbb Z/2$.
\end{lemma}

\begin{proof}
By Corollary~\ref{cor:hyperbolic-kernel} the mapping class group is
the isometry group of the complete hyperbolic structure, and the full
symmetry group of the figure-eight knot is dihedral of order eight;
see \cite[Section~9]{HeusenerMunozPorti} and, from the canonical
decomposition of $E_8$ into two regular ideal tetrahedra,
\cite[Theorem~6.2 and Example~6.3]{SakumaSurvey}.

These references identify the full group but do not single out which
index-two subgroup is orientation-preserving, so we check directly
that $\Mod^{+}(E_8)$ is a Klein four group.  The figure-eight knot
is amphichiral \cite{Rolfsen}, so $\Mod^{\pm}$ contains an
orientation-reversing class and $\Mod^{+}$ has index two, hence order
four.  Like every two-bridge knot, $4_1$ is strongly invertible: a
two-bridge knot in Schubert normal form is invariant under the
rotation by $\pi$ about an unknotted axis meeting it in two points,
and this rotation is a strong inversion $\iota$
\cite[Proposition~2.2]{HirasawaHiuraSakuma}.  By
Lemma~\ref{lem:strong-inversion-sign}, $[\iota]\in\Mod^{+}(E_8)$
acts on $H_1(\partial E_8;\mathbb Z)$ by $-\mathrm{id}$, so it is a
nontrivial involution.  If $\Mod^{+}(E_8)$ were cyclic of order
four with generator $g$, its unique involution would be $g^2$, so
$[\iota]=g^2$; but by Remark~\ref{rem:GL} the orientation-preserving
class $g$ acts by $\pm\mathrm{id}$, so $g^2$ acts by
$+\mathrm{id}\ne-\mathrm{id}$, a contradiction.  Hence
$\Mod^{+}(E_8)\cong\mathbb Z/2\oplus\mathbb Z/2$.
\end{proof}

\begin{corollary}[Figure-eight action calculation]
\label{cor:fig8-final}
Let $E_8=E(4_1)$.  Then
\begin{enumerate}[label=(\arabic*),leftmargin=*]
\item $\Aut(\ES(E_8))\cong S_3$, while the image of the natural action
is the order-two subgroup generated by
$[S_+]\leftrightarrow[S_-]$;
\item
\[
  \Aut(\ES^{\mathrm{per}}(E_8))\cong
  \Aut(\CES(E_8))\cong\mathbb Z/2;
\]
for both objects the natural action is surjective;
\item at all three levels the kernel is
\[
  \mathcal K_X(E_8)=\Mod^{+}(E_8)
  \cong\mathbb Z/2\oplus\mathbb Z/2,
\]
and the geometric image is $\mathbb Z/2$.
\end{enumerate}
\end{corollary}

\begin{proof}
By Lemma~\ref{lem:figure-eight-census}, $4_1$ is slope-separated with
$S=\{0,+4,-4\}$, and it is amphichiral.  By
Corollary~\ref{cor:hyperbolic-kernel} the characteristic frontier is
empty.  Theorem~\ref{thm:visibility} therefore gives image
$\mathbb Z/2$, generated by $v_s\mapsto v_{-s}$, and kernel
$\Mod^{+}(E_8)$, which Lemma~\ref{lem:figure-eight-action}
identifies with $(\mathbb Z/2)^2$.

The bare complex is three isolated vertices, so its automorphism group
is $S_3$.  For the peripheral refinement, the slope involution fixes
the slope-$0$ once-punctured torus and exchanges the two
slope-$\pm4$ twice-punctured tori; hence
Proposition~\ref{prop:aut-slope-separated} gives automorphism group
$\mathbb Z/2$.  The characteristic data are trivial in the hyperbolic
one-piece case, so $\CES$ has the same automorphism group.  Thus the
peripheral and characteristic actions are surjective.
\end{proof}

\begin{remark}[Why the partial views are useful here]
For the figure-eight knot, surface type alone also leaves exactly the
transposition $S_+\leftrightarrow S_-$, while slope alone leaves the
same sign involution.  Figure~\ref{fig:decorations-fig8} records this
without introducing two extra complexes into the main notation.
\end{remark}

\begin{remark}[The Kakimizu layer loses the visible symmetry]
\label{rem:fig8-kakimizu}
The figure-eight knot has a unique minimal-genus Seifert surface up to
isotopy: in the Hatcher--Thurston census the all-even expansion
$[-2,2]$ has no coefficient of absolute value greater than two, so it
carries only the single class $[F]$.  Hence $\operatorname{MS}(4_1)$
is a point and its automorphism group is trivial: the amphichiral
symmetry visible in Corollary~\ref{cor:fig8-final} is detected only
after the two non-Seifert essential surfaces are added.  This is the
opposite of the connected-sum situation of
Section~\ref{sec:composite}, where a Kakimizu subcomplex detects an
infinite-order annular twist and therefore detects it in $\ES(E)$ as
well.  The contrast is not one of incompatible complexes but of which
\emph{part} of the common essential-surface complex carries the useful
information; compare Proposition~\ref{prop:kakimizu-subcomplex}.
\end{remark}

\subsection{Beyond the first example: two-bridge knots}
\label{subsec:two-bridge-roadmap}

The proof above used only the census, chirality, and the sign rule.
This suggests a systematic two-bridge extension, but one point
requires care.  By Hatcher--Thurston, a minimal continued-fraction
expansion with at least two odd coefficients carries connected
$n$-sheeted surfaces for every $n$
\cite[Proposition~1]{HatcherThurston}, so the finite three-vertex
behaviour of $4_1$ is special and slope-separatedness cannot be
assumed; Section~\ref{sec:two-bridge} determines exactly when it
holds.  Appendix~\ref{app:recipe} gives the computational recipe;
here is the program it leads to.

\begin{problem}[Two-bridge essential-surface rigidity]
\label{prob:two-bridge-rigidity}
Let $K$ be a hyperbolic two-bridge knot and $E=E(K)$.  Use the
Hatcher--Thurston coordinates $S_n(n_1,\ldots,n_{k-1})$ to determine
the natural actions on the bare complex $\ES(E)$ and on the peripheral refinement $\ES^{\mathrm{per}}(E)$ by solving:
\begin{enumerate}[label=(B\arabic*),leftmargin=*]
\item an explicit orientability, connectedness, isotopy, and label
criterion for all carried surfaces which are vertices of $\ES(E)$;
\item the simultaneous-disjointness relation \emph{within} each
boundary-slope layer, including surfaces carried by different minimal
continued-fraction expansions with the same slope;
\item the permutation of the Hatcher--Thurston coordinates induced by
every isometry of $E$, including the exceptional symmetries of
arithmetically symmetric two-bridge knots;
\item whether type and boundary slope already suffice to make every automorphism geometric and to read the finite invisible kernel from the coordinate action.
\end{enumerate}
In particular, identify natural infinite subclasses --- for example
slope-separated ones --- for which the whole problem can be settled
uniformly by Theorem~\ref{thm:visibility}.
\end{problem}

The genuinely new input is (B2).  By Lemma~\ref{lem:slope-layer}
different slopes automatically preclude disjointness, so the entire
combinatorial content of $\ES(E)$ for a two-bridge knot lies inside
the individual slope layers.  Items (B1) and (B3) supply,
respectively, the vertex census and the geometric permutation group
against which the combinatorial automorphism group must be compared.
Corollary~\ref{cor:fig8-final} solves the problem in the first
hyperbolic example precisely because every layer is a singleton.
Theorem~\ref{thm:two-bridge-slope-separated} shows that this happens
exactly for the twist knots and the torus knots $T(2,q)$, and solves
the problem for them; Proposition~\ref{prop:two-bridge-finite} gives
the vertex census required in (B1) whenever the vertex set is finite;
Proposition~\ref{prop:two-bridge-symmetry-action} gives the action
required in (B3) for the standard Klein four group of symmetries; and
Theorem~\ref{thm:six-two} solves the whole problem for $6_2$, the
first knot for which (B2) and (B3) are genuinely needed.  In fact
(B1) and (B2) are solved for all two-bridge knots by
Theorem~\ref{thm:two-bridge-decomposition},
Lemma~\ref{lem:sheet-rule} and Proposition~\ref{prop:components}, and
(B4) is answered for all two-bridge knots with finite complex in
Theorem~\ref{thm:class-two}.

\section{Two-bridge knots: a survey of results}
\label{sec:two-bridge}

The figure-eight calculation used nothing but the Hatcher--Thurston
census.  For two-bridge knots the same input can be pushed to a
complete answer; this is done in the companion paper \cite{OzawaTB}, and
in this section we survey its results.  \emph{All statements of this
section are proved in \cite{OzawaTB}; proofs are given here only where
they are a few lines long and use the general theory of the preceding
sections.}  The main results are: the complex is the disjoint union of
the complexes carried by the individual branched surfaces, which are
computed by an explicit rule (Theorem~\ref{thm:two-bridge-decomposition}
and Lemma~\ref{lem:sheet-rule}; \cite[Theorem~A]{OzawaTB}); the
two-bridge knots fall into three classes
(Theorem~\ref{thm:three-classes}; \cite[Theorem~B]{OzawaTB}), the first
of which consists of the slope-separated ones
(Theorem~\ref{thm:two-bridge-slope-separated}); the first two classes
are treated completely (Theorems~\ref{thm:six-two}
and~\ref{thm:class-two}; \cite[Theorem~C]{OzawaTB}); for the third
the structure of the complex is determined
(Theorems~\ref{thm:six-three} and~\ref{thm:general-structure};
\cite[Theorem~D]{OzawaTB}); and the kernel of the action is determined
for every hyperbolic two-bridge knot, whereas the peripheral complex
is not rigid in the third class
(Theorems~\ref{thm:two-bridge-kernel} and~\ref{thm:class-three-aut};
\cite[Theorem~E]{OzawaTB}).  These results settle
Problem~\ref{prob:singleton-layers} and
Question~\ref{q:chiral-symmetric} for two-bridge knots, and
Problem~\ref{prob:two-bridge-rigidity} to the extent described in
Remark~\ref{rem:two-bridge-remaining}.

The paper \cite{OzawaTB} states the results for $\ES(E)$ and
$\ES^{\mathrm{per}}(E)$.  The statements about $\CES(E)$ below follow
from them: the exterior of a two-bridge knot contains no essential
torus by \textup{(HT1)}, so it is Seifert fibered or hyperbolic
\cite{Thurston}, the characteristic frontier is empty, and the
characteristic data add nothing to the peripheral data.

\subsection{The Hatcher--Thurston input}
\label{subsec:HT-input}

Throughout this section $K=K_{p/q}$ is a two-bridge knot in Schubert
normal form, with $q\ge3$ odd, $0<p<q$ and $\gcd(p,q)=1$.  Recall that
$K_{p/q}$ and $K_{p'/q}$ are isotopic if $p'\equiv p^{\pm1}\pmod q$,
and are mirror images of each other if $p'\equiv-p^{\pm1}\pmod q$
\cite{SchubertTwoBridge}.  For nonzero integers $b_1,\dots,b_k$ put
\[
  \langle b_1,\dots,b_k\rangle
  =b_1-\cfrac{1}{b_2-\cfrac{1}{\ddots-\cfrac{1}{b_k}}}\,,
\]
so that the expansions of Step~1 of Recipe~\ref{rec:main} read
$p/q=r+1/\langle b_1,\dots,b_k\rangle$.  We call
$\beta=[b_1,\dots,b_k]$ an \emph{admissible expansion} of $p/q$ if this
holds for some $r\in\mathbb Z$ and $|b_i|\ge2$ for all $i$, and write
\[
  \sigma(\beta)=\#\{i:b_i>0\}-\#\{i:b_i<0\},\qquad
  J(\beta)=\{i:|b_i|=2\}.
\]
We call $\beta$ \emph{thin} if at most one $b_i$ has $|b_i|\ne2$, that
is, if $|J(\beta)|\ge k-1$.  Exactly one admissible expansion,
$\beta_{\mathrm{ev}}$, has all $b_i$ even \cite{HatcherThurston}.

To an admissible expansion $\beta$ of length $k$ Hatcher and Thurston
associate a branched surface $\Sigma_\beta$ in $E(K)$.  It is the union
of $k$ bands $B_1,\dots,B_k$, where $B_i$ has $b_i$ half-twists (so
$B_i$ is an annulus if $b_i$ is even and a M\"obius band if $b_i$ is
odd), consecutive bands being plumbed along a square, together with a
second disk at each plumbing: the $j$-th plumbing square $Q_j^{\mathrm
{in}}$ lies in a $2$-sphere, and its complementary disk
$Q_j^{\mathrm{out}}$ in that sphere is added to the branched surface.
The two bands $B_j$ and $B_{j+1}$ leave the sphere on opposite sides,
$B_j$ along one pair of opposite edges of the square and $B_{j+1}$
along the other pair.  For $n\ge1$ and integers $0\le n_j\le n$ the
surface
\[
  S_n(n_1,\dots,n_{k-1})
\]
is the surface carried by $\Sigma_\beta$ with $n$ parallel sheets over
every band, of which, at the $j$-th plumbing, $n_j$ run over
$Q_j^{\mathrm{in}}$ and $n-n_j$ over $Q_j^{\mathrm{out}}$.  (Which of
the two disks is called inner is immaterial below.)  Since sheets
cannot cross, at each edge of the square the $n_j$ sheets nearest to
$Q_j^{\mathrm{in}}$ are the ones that turn into it.

\begin{theorem}[Hatcher--Thurston \cite{HatcherThurston}]
\label{thm:HT}
Let $K=K_{p/q}$ be a two-bridge knot and $E=E(K)$.
\begin{enumerate}[label=\textup{(HT\arabic*)},leftmargin=*]
\item $E$ contains no closed essential surface.
\item A surface with boundary in $E$, one-sided or two-sided, is
incompressible and boundary-incompressible if and only if it is
isotopic to some $S_n(n_1,\dots,n_{k-1})$ carried by $\Sigma_\beta$
for an admissible expansion $\beta$ of $p/q$.
\item Surfaces carried by $\Sigma_\beta$ and $\Sigma_{\beta'}$ with
$\beta\ne\beta'$ are not isotopic.  For fixed $\beta$ and $n$, isotopy
among the surfaces $S_n(n_1,\dots,n_{k-1})$ is the equivalence relation
generated by
\[
  S_n(\dots,n_{i-1},n_i,\dots)\ \simeq\
  S_n(\dots,n_{i-1}+1,n_i+1,\dots)\qquad\text{whenever }|b_i|=2,
\]
where for $i=1$ only $n_1$, and for $i=k$ only $n_{k-1}$, is changed.
\item All surfaces carried by $\Sigma_\beta$ have the same boundary
slope, the even integer
$2\bigl(\sigma(\beta)-\sigma(\beta_{\mathrm{ev}})\bigr)$.
\item If $S_n(n_1,\dots,n_{k-1})$ is connected, then $n=1$, or $n=2$
and some $b_i$ is odd, or at least two of the $b_i$ are odd.
\item If at least two of the $b_i$ are odd, then for every $n\ge1$ some
$S_n(n_1,\dots,n_{k-1})$ is connected.
\end{enumerate}
In particular $E$ contains no meridional essential surface, so
conditions \textup{(i)} and \textup{(ii)} of
Definition~\ref{def:slope-separated-intro} hold for every two-bridge
knot.
\end{theorem}

Here (HT1)--(HT3) are \cite[Theorem~1]{HatcherThurston}, (HT4) is
\cite[Proposition~2]{HatcherThurston}, and (HT5), (HT6) are
\cite[Proposition~1(1),(3)]{HatcherThurston}; see also
\cite[Theorem~2.6.1]{Owad} and \cite[Section~2]{CurtisEtAl}.  Note
that connectedness genuinely depends on the $n_j$ once $n\ge3$: for
instance $S_3(0,\dots,0)$ is three parallel copies of
$S_1(0,\dots,0)$ and is never connected.  The number of isotopy classes of $n$-sheeted surfaces
carried by $\Sigma_\beta$ is discussed in
Remark~\ref{rem:HT-count}.

\subsection{Elementary properties of the carried surfaces}

\begin{lemma}
\label{lem:carried-basic}
Let $\beta=[b_1,\dots,b_k]$ be an admissible expansion and
$S=S_n(n_1,\dots,n_{k-1})$.
\begin{enumerate}[label=\textup{(\arabic*)},leftmargin=*]
\item $\chi(S)=-n(k-1)$ and $S$ has exactly $n$ boundary components.
Hence, if $S$ is connected and orientable, its genus is
$1+n(k-2)/2$.
\item $S_1(n_1,\dots,n_{k-1})$ is connected, and it is orientable if
and only if every $b_i$ is even.
\item If $\varepsilon\in\{0,1\}^{k-1}$, then $S_2(2\varepsilon)$ is the
frontier of a regular neighbourhood of $S_1(\varepsilon)$.  It is
orientable, and it is connected if and only if some $b_i$ is odd.
\item Every two-sheeted surface $S_2(n_1,\dots,n_{k-1})$ is orientable,
and it is connected if some $b_i$ is odd.
\end{enumerate}
\end{lemma}

Parts (1)--(3), and the connectedness statement in (4), are contained
in \cite{HatcherThurston}: see the proof of Theorem~2(a), the remark
preceding the Corollary on p.~230, and Proposition~1(2) there.  The
orientability statement in (4), on which the vertex census below
depends, does not seem to be stated there.

\begin{lemma}[Counting isotopy classes]
\label{lem:orbits}
Let $\beta=[b_1,\dots,b_k]$ be an admissible expansion.
\begin{enumerate}[label=\textup{(\arabic*)},leftmargin=*]
\item If $\beta$ is thin, then for each $n$ all surfaces
$S_n(n_1,\dots,n_{k-1})$ are isotopic.
\item If $\beta$ is not thin, there are
$\varepsilon,\varepsilon'\in\{0,1\}^{k-1}$ such that
$S_1(\varepsilon)$ and $S_1(\varepsilon')$ are not isotopic, and
$\eta,\eta'\in\{0,1\}^{k-1}$ such that $S_2(2\eta)$ and $S_2(2\eta')$
are not isotopic.
\end{enumerate}
\end{lemma}

\subsection{The vertex census and a finiteness criterion}

\begin{proposition}[Vertices carried by one branched surface]
\label{prop:two-bridge-finite}
Let $K=K_{p/q}$ be a two-bridge knot, $E=E(K)$, and let $\beta$ be an
admissible expansion of length $k$.
\begin{enumerate}[label=\textup{(\arabic*)},leftmargin=*]
\item If $\beta=\beta_{\mathrm{ev}}$, the vertices of $\ES(E)$ carried
by $\Sigma_\beta$ are exactly the classes of the single-sheeted
surfaces $S_1(\varepsilon)$.  They are once-punctured surfaces of
genus $k/2$ and slope $0$, and there is exactly one of them if and
only if $\beta_{\mathrm{ev}}$ is thin
\cite[Corollary, p.~230]{HatcherThurston}.
\item If exactly one $b_i$ is odd, the vertices carried by
$\Sigma_\beta$ are exactly the classes of the two-sheeted surfaces
$S_2(n_1,\dots,n_{k-1})$, $n_j\in\{0,1,2\}$.  They are twice-punctured
surfaces of genus $k-1$, and there is exactly one of them if and only
if $\beta$ is thin.
\item If at least two $b_i$ are odd, $\Sigma_\beta$ carries infinitely
many vertices, all in one slope layer.
\end{enumerate}
Consequently $\ES(E)$ is finite if and only if every admissible
expansion of $p/q$ has at most one odd term, and in that case the
number of its vertices is the number of classes in \textup{(1)} plus
the numbers of classes in \textup{(2)}, each computed by
\textup{(HT3)}.
\end{proposition}

\begin{remark}
\label{rem:finite-class}
The class of two-bridge knots with finite $\ES(E)$ is much larger than
the slope-separated class determined below: it contains for instance
$6_2=K_{3/11}$ (Proposition~\ref{prop:six-two}) and $7_4=K_{4/15}$,
whose admissible expansions are $[4,4]$, $[-2,-2,-2,3]$,
$[3,-2,-2,-2]$ and $[-2,-2,-3,-2,-2]$.  Boundary slopes shared by
different branched surfaces, as for the second and third of these, are
classified by Curtis, Franczak, Leiser and Manheimer
\cite{CurtisEtAl}.
\end{remark}

\subsection{The classification of slope-separated two-bridge knots}

\begin{lemma}[Thin expansions]
\label{lem:thin}
Let $q\ge3$ be odd, $0<p<q$ and $\gcd(p,q)=1$.
\begin{enumerate}[label=\textup{(\arabic*)},leftmargin=*]
\item If every admissible expansion of $p/q$ is thin, then
$p\equiv\pm1$, $\pm2$ or $\pm2^{-1}\pmod q$.
\item The admissible expansions of $1/q$ are
\[
  \tfrac1q=0+[q]=1+[\underbrace{-2,\dots,-2}_{q-1}] .
\]
\item Let $q=2m+1$ with $m\ge2$.  The admissible expansions of $2/q$
are
\[
  \tfrac2q=0+[m,-2]=0+[m+1,2]
          =1+[\underbrace{-2,\dots,-2}_{m-1},-3] .
\]
\end{enumerate}
\end{lemma}

Recall that the knots $K_{1/q}$ are the torus knots $T(2,q)$, and that
the knots $K_{2/q}$ with $q\ge5$ are the \emph{twist knots}
$4_1,5_2,6_1,7_2,8_1,9_2,10_1,\dots$ (for $q=5,7,9,\dots$).

\begin{theorem}[Slope-separated two-bridge knots]
\label{thm:two-bridge-slope-separated}
For a two-bridge knot $K=K_{p/q}$ the following are equivalent.
\begin{enumerate}[label=\textup{(\alph*)},leftmargin=*]
\item $K$ is slope-separated.
\item Every admissible expansion of $p/q$ is thin.
\item $K$ or its mirror image is a torus knot $T(2,q)$ or a twist knot
$K_{2/q}$.
\end{enumerate}
In that case $\ES(E(K))$ is discrete with two vertices for $T(2,q)$
and three vertices for a twist knot, and $\ES^{\mathrm{per}}(E(K))$ is
given by Table~\ref{tab:two-bridge-slope-separated}.
\end{theorem}

\begin{table}[htbp]
\centering
\caption{The peripheral complex $\ES^{\mathrm{per}}(E(K))$ of the
slope-separated two-bridge knots: one vertex for each entry.  A type
$(g,b)$ is a surface of genus $g$ with $b$ boundary components.
Slopes are in the convention of Step~4 of Recipe~\ref{rec:main}; the
mirror image has the negative slopes.}
\label{tab:two-bridge-slope-separated}
\small
\renewcommand{\arraystretch}{1.25}
\begin{tabular}{@{}llll@{}}
\toprule
knot & expansion & slope & type\\
\midrule
$K_{1/q}=T(2,q)$,\ \ $q=2g+1\ge3$
  & $[-2,\dots,-2]$ & $0$ & $(g,1)$\\
  & $[q]$ & $2q$ & $(0,2)$\\
\midrule
$K_{2/q}$,\ \ $q=4c+1$,\ $c\ge1$
  & $[2c,-2]$ & $0$ & $(1,1)$\\
  & $[2c+1,2]$ & $4$ & $(1,2)$\\
  & $[-2,\dots,-2,-3]$ & $-4c$ & $(2c-1,2)$\\
\midrule
$K_{2/q}$,\ \ $q=4c-1$,\ $c\ge2$
  & $[2c,2]$ & $0$ & $(1,1)$\\
  & $[2c-1,-2]$ & $-4$ & $(1,2)$\\
  & $[-2,\dots,-2,-3]$ & $-4c-2$ & $(2c-2,2)$\\
\bottomrule
\end{tabular}
\end{table}

\begin{figure}[htb]
\centering
\begin{tikzpicture}[scale=1.0,every node/.style={font=\small}]
% axis
\draw[greyish,-{Stealth[length=5pt]}] (-6.3,0) -- (4.6,0);
\node[greyish,right] at (4.6,0) {slope};
\foreach \x/\l in {-5.6/{$-4c$},0/{$0$},1.6/{$4$}}{
  \draw[greyish] (\x,-0.08)--(\x,0.08);
}
\fill[surfA] (0,0) circle (2.6pt);
\node[surfA,above=5pt] at (0,0) {$(1,1)$};
\node[surfA,below=5pt] at (0,0) {$0$};
\fill[surfB] (1.6,0) circle (2.6pt);
\node[surfB,above=5pt] at (1.6,0) {$(1,2)$};
\node[surfB,below=5pt] at (1.6,0) {$4$};
\fill[surfC] (-5.6,0) circle (2.6pt);
\node[surfC,above=5pt] at (-5.6,0) {$(2c-1,2)$};
\node[surfC,below=5pt] at (-5.6,0) {$-4c$};
\draw[greyish,densely dashed] (-1.6,-0.08)--(-1.6,0.08);
\node[greyish,below=5pt] at (-1.6,0) {$-4$};
\node[greyish,align=center] at (-2.8,1.25)
  {the third vertex is at $-4$ only for $c=1$,\\ the figure-eight knot};
\draw[greyish,-{Stealth[length=4pt]}] (-2.3,0.82) -- (-1.68,0.14);
\end{tikzpicture}
\caption{The essential-surface complex of the twist knot
$K_{2/(4c+1)}$: three isolated vertices
(Theorem~\ref{thm:two-bridge-slope-separated}).  The slope set
$\{0,4,-4c\}$ is symmetric about $0$ only for $c=1$, so the mirror
symmetry of Figure~\ref{fig:fig8} is special to the figure-eight
knot.  The other family, $K_{2/(4c-1)}$, has slope set
$\{0,-4,-4c-2\}$, which is never symmetric.}
\label{fig:twist-knots}
\end{figure}

\begin{corollary}[Image, kernel and automorphisms]
\label{cor:two-bridge-image-kernel}
Let $K$ be a slope-separated two-bridge knot other than the
figure-eight knot, and $E=E(K)$.  Then $S(K)\ne-S(K)$, the knot $K$ is
chiral, and for each of $X=\ES(E)$, $\ES^{\mathrm{per}}(E)$, $\CES(E)$
\[
  \Image\Phi_X=1,\qquad
  \mathcal K_X(E)=\Mod^{\pm}(E)=\Mod^{+}(E).
\]
Moreover $\Aut(\ES^{\mathrm{per}}(E))=\Aut(\CES(E))=1$, so the natural
action is surjective at the peripheral and characteristic levels,
whereas $\Aut(\ES(E))$ is $S_2$ for $T(2,q)$ and $S_3$ for a twist
knot.  Either datum alone, surface type or boundary slope, already
removes every nongeometric automorphism.
\end{corollary}

\begin{proof}
By Table~\ref{tab:two-bridge-slope-separated} the slope set is
$\{0,2q\}$, $\{0,4,-4c\}$ with $c\ge2$, or $\{0,-4,-4c-2\}$, up to a
global sign; none is symmetric.  So $K$ is chiral by
Corollary~\ref{cor:chirality}, and Theorem~\ref{thm:visibility}
applies: the characteristic frontier is empty, because $E$ contains no
essential torus by (HT1) and is therefore Seifert fibered (for
$T(2,q)$) or hyperbolic \cite{Thurston}.
Proposition~\ref{prop:aut-slope-separated} gives the automorphism
groups, and $\Aut(\CES(E))\subseteq\Aut(\ES^{\mathrm{per}}(E))$.  The
last sentence holds since the slope set is not symmetric and the types
in each block of the table are pairwise distinct when $K\ne4_1$.
\end{proof}

For the hyperbolic twist knots $K_{2/q}$ with $q\ge7$ the group
$\Mod^{\pm}(E)$ is a Klein four group, see
\cite{SakumaSurvey,KodamaSakuma}, so
Corollary~\ref{cor:two-bridge-image-kernel} exhibits an infinite family
of hyperbolic knots for which a group of order four is invisible at
all three levels although every automorphism of
$\ES^{\mathrm{per}}(E)$ is geometric.

\begin{corollary}[Question~\ref{q:chiral-symmetric} for two-bridge knots]
\label{cor:two-bridge-question}
The figure-eight knot is the only slope-separated two-bridge knot with
$S(K)=-S(K)\ne\{0\}$.  In particular no two-bridge knot is a chiral
slope-separated knot with symmetric peripheral data, and for every
slope-separated two-bridge knot the action on
$\ES^{\mathrm{per}}(E)$ is surjective.
\end{corollary}

\begin{corollary}[Small knots]
\label{cor:small-table}
Among the two-bridge knots with at most ten crossings, the
slope-separated ones are exactly
\[
  3_1,\ 4_1,\ 5_1,\ 5_2,\ 6_1,\ 7_1,\ 7_2,\ 8_1,\ 9_1,\ 9_2,\ 10_1 ,
\]
and their peripheral complexes are given by
Table~\ref{tab:two-bridge-slope-separated} with $q$ the determinant.
\end{corollary}

\subsection{Two reconstruction statements}

For two knots $K,K'$ we call a simplicial isomorphism
$\varphi:\ES(E(K))\to\ES(E(K'))$ an \emph{isomorphism of peripheral
complexes} if it preserves surface type and there is a sign
$\varepsilon$ with $\Sl(\varphi(v))=\varepsilon\,\Sl(v)$ for all
boundary-bearing vertices $v$; compare
Definition~\ref{def:automorphisms}.

\begin{proposition}[Torus knots are reconstructed]
\label{prop:torus-reconstruction}
Let $K$ be a nontrivial knot such that
$\ES^{\mathrm{per}}(E(K))\cong\ES^{\mathrm{per}}(E(T(p,q)))$ for some
$1<p<q$.  Then $K$ is $T(p,q)$ or its mirror image.
\end{proposition}

\begin{proof}
By Lemma~\ref{lem:torus-classification}, $\ES(E(K))$ has exactly two
vertices: an annulus of slope $\pm pq$ and a once-punctured surface of
genus $(p-1)(q-1)/2$ and slope $0$.  The annulus is not meridional, so
it is a cabling annulus and $K$ is a torus knot or a cable knot
\cite[Lemma~1.2]{KodaOzawa}.  A cable knot has an essential companion
torus, and $\ES(E(K))$ has no torus vertex; so $K$ is, up to mirror
image, a torus knot $T(a,b)$ with $1<a<b$.  Comparing the two
complexes gives $ab=pq$ and $(a-1)(b-1)=(p-1)(q-1)$, hence $a+b=p+q$
and $\{a,b\}=\{p,q\}$.
\end{proof}

\begin{proposition}[Reconstruction among two-bridge knots]
\label{prop:two-bridge-reconstruction}
Let $K$ and $K'$ be two-bridge knots with $K$ slope-separated.  If the
slope sets $S(K)$ and $S(K')$ agree up to a global sign and
$\ES(E(K'))$ is discrete with $|S(K')|$ vertices --- in particular if
$\ES^{\mathrm{per}}(E(K))\cong\ES^{\mathrm{per}}(E(K'))$ --- then $K'$
is $K$ or its mirror image.
\end{proposition}

\begin{proof}
The hypothesis says that $K'$ satisfies condition (iii) of
Definition~\ref{def:slope-separated-intro}; conditions (i) and (ii)
hold by Theorem~\ref{thm:HT}.  So $K'$ is slope-separated and both
knots occur in Table~\ref{tab:two-bridge-slope-separated}.  Up to sign
the slope sets there are $\{0,2q\}$, $\{0,4,-4c\}$ and
$\{0,-4,-4c-2\}$.  They are distinguished by their cardinality and, in
the three-element case, by whether the two nonzero slopes have opposite
or equal signs, and then they determine $q$, respectively $c$.
\end{proof}

\subsection{The standard symmetries in Hatcher--Thurston coordinates}
\label{subsec:two-bridge-symmetries}

To go beyond the slope-separated case one needs the action of the
symmetries on the surfaces $S_n(n_1,\dots,n_{k-1})$; this is item
\textup{(B3)} of Problem~\ref{prob:two-bridge-rigidity}.  For the
standard Klein four group of symmetries it can be read off from the
proof of \cite[Theorem~1(b)]{HatcherThurston}.

We use the picture of \cite[Section~2]{HatcherThurston}: $K\subset
S^2\times[0,1]\subset S^3$, each level sphere being
$S^2_t=\mathbb R^2/\Gamma$, where $\Gamma$ is generated by the
rotations by $\pi$ about the points of $\mathbb Z^2$; the knot meets
$S^2_t$ in the four points $\mathbb Z^2/\Gamma$ for $0<t<1$, in two
arcs of slope $1/0$ for $t=1$, and in two arcs of slope $p/q$ for
$t=0$.  A slope $a/b$ is the direction of the vector $(b,a)$.  For
$v\in\mathbb Z^2$ the translation $x\mapsto x+v$ normalizes $\Gamma$,
preserves $K\cap S^2_t$ for every $t$, and induces on each level a map
depending only on $v$ modulo $2$.  For $v\not\equiv0$ it therefore
defines a level-preserving, orientation-preserving involution
$\tau_v$ of $(S^3,K)$, and
\[
  V=\{1,\tau_{(0,1)},\tau_{(1,0)},\tau_{(1,1)}\}
\]
is the familiar Klein four group of symmetries of a two-bridge knot.
The fixed points of $\tau_v$ on a level are the two points of
$(\tfrac12v+\mathbb Z^2)/\Gamma$.  They lie on $K$ only for $t=1$ and
$v\equiv(0,1)$, or $t=0$ and $v\equiv(q,p)$.  Hence
$\tau_{(0,1)}$ and $\tau_{(1,p)}$ are strong inversions, and
$\tau_{(1,p+1)}$ is the symmetry of period two whose axis misses $K$.

Let $\beta=[b_1,\dots,b_k]$ be an admissible expansion,
$p/q=r+1/\langle b_1,\dots,b_k\rangle$, and let
\[
  \frac{p_0}{q_0}=\frac r1,\qquad
  \frac{p_i}{q_i}=r+\frac1{\langle b_1,\dots,b_i\rangle}\quad(1\le i\le k)
\]
be its convergents, in lowest terms.  For $1\le i\le k-1$ put
\[
  c_i(\beta)=(q_{i-1}+q_i,\ p_{i-1}+p_i)\in(\mathbb Z/2)^2 .
\]

\begin{proposition}[Action of the standard symmetries]
\label{prop:two-bridge-symmetry-action}
For every nonzero $v\in(\mathbb Z/2)^2$ the surface
$\tau_v\bigl(S_n(n_1,\dots,n_{k-1})\bigr)$ is isotopic to
$S_n(n'_1,\dots,n'_{k-1})$, carried by the same branched surface
$\Sigma_\beta$, where
\[
  n'_i=\begin{cases} n_i, & v=c_i(\beta),\\ n-n_i, & v\ne c_i(\beta).\end{cases}
\]
\end{proposition}

\begin{corollary}[Seifert surfaces]
\label{cor:two-bridge-seifert-symmetry}
Let $K$ be a two-bridge knot.  The symmetry of period two in $V$ fixes
every vertex $[S_1(\varepsilon)]$ of $\operatorname{MS}(K)$, and each
of the two strong inversions in $V$ sends $[S_1(\varepsilon)]$ to
$[S_1(1-\varepsilon)]$, where
$1-\varepsilon=(1-\varepsilon_1,\dots,1-\varepsilon_{k-1})$.
\end{corollary}

The first statement is an instance of Edmonds's theorem
\cite{Edmonds}; the second is consistent with the dichotomy of
Hirasawa, Hiura and Sakuma \cite[Corollary~1.6]{HHSInvariant}, and
shows that a strong inversion in $V$ has an invariant minimal-genus
Seifert surface, up to isotopy, exactly when $\varepsilon$ and
$1-\varepsilon$ are equivalent under \textup{(HT3)} for some
$\varepsilon$; compare \cite{HirasawaHiuraSakuma}.  For a hyperbolic
two-bridge knot with $p^2\not\equiv\pm1\pmod q$ the group $V$ is all
of $\Mod^{\pm}(E)$ \cite{SakumaSurvey,KodamaSakuma}, so for these knots
Proposition~\ref{prop:two-bridge-symmetry-action} answers
\textup{(B3)} completely.

\subsection{Doubles of one-sided surfaces are never adjacent}

\begin{lemma}
\label{lem:doubles-not-adjacent}
Let $E$ be the exterior of a nontrivial knot and let $F_0,F_1\subset E$
be properly embedded connected one-sided surfaces with connected
boundary.  Suppose that the frontiers $D_0,D_1$ of their regular
neighbourhoods are essential and not isotopic.  Then $[D_0]$ and
$[D_1]$ are not adjacent in $\ES(E)$.
\end{lemma}

\subsection{The knot $6_2$: a complete calculation}

Outside the slope-separated class the complex need not be discrete,
even for a hyperbolic knot with finitely many essential surfaces.

\begin{proposition}[The knot $6_2$]
\label{prop:six-two}
Let $K=K_{3/11}=6_2$ and $E=E(K)$.  Then $\ES(E)$ has exactly six
vertices:
\begin{center}
\upshape\small
\renewcommand{\arraystretch}{1.2}
\begin{tabular}{@{}llll@{}}
\toprule
expansion & slope & type & vertices\\
\midrule
$[-2,-2,-2,2]$ & $0$ & $(2,1)$ & one, the fibre\\
$[3,-2,-2]$ & $2$ & $(2,2)$ & one\\
$[-2,-2,-3,-2]$ & $-4$ & $(3,2)$ & one\\
$[4,3]$ & $8$ & $(1,2)$ & three: $w_0,w_1,w_2$, where $w_i=[S_2(i)]$\\
\bottomrule
\end{tabular}
\end{center}
The edges of $\ES(E)$ are exactly $\{w_0,w_1\}$ and $\{w_1,w_2\}$.  Hence
$\ES(E(6_2))$ is not discrete, while $\ES(E(K))$ is discrete for every
knot $K$ with fewer crossings and for $6_1$.
\end{proposition}

\begin{figure}[htb]
\centering
\begin{tikzpicture}[scale=1.0,every node/.style={font=\small}]
\draw[greyish,-{Stealth[length=5pt]}] (-3.4,0) -- (7.4,0);
\node[greyish,right] at (7.4,0) {slope};
\fill[surfC] (-2.4,0) circle (2.6pt);
\node[surfC,below=5pt] at (-2.4,0) {$-4$};
\node[surfC,above=5pt] at (-2.4,0) {$(3,2)$};
\fill[surfA] (0,0) circle (2.6pt);
\node[surfA,below=5pt] at (0,0) {$0$};
\node[surfA,above=5pt] at (0,0) {$(2,1)$};
\fill[surfC] (1.2,0) circle (2.6pt);
\node[surfC,below=5pt] at (1.2,0) {$2$};
\node[surfC,above=5pt] at (1.2,0) {$(2,2)$};
% slope 8 layer
\fill[surfB,opacity=0.10] (4.2,0.45) rectangle (6.2,3.05);
\draw[surfB,opacity=0.55] (4.2,0.45) rectangle (6.2,3.05);
\draw[greyish] (5.2,-0.08)--(5.2,0.08);
\node[surfB,below=5pt] at (5.2,0) {$8$};
\coordinate (w0) at (5.2,0.8);
\coordinate (w1) at (5.2,1.75);
\coordinate (w2) at (5.2,2.7);
\draw[thick] (w0)--(w1)--(w2);
\draw[{Stealth[length=4pt]}-{Stealth[length=4pt]},greyish,thick] (5.75,0.85) to[out=35,in=-35] node[right,greyish]{$\rho$} (5.75,2.65);
\fill[surfB] (w0) circle (2.6pt); \node[surfB,left=3pt] at (w0) {$w_0$};
\fill[surfB] (w1) circle (2.6pt); \node[surfB,left=3pt] at (w1) {$w_1$};
\fill[surfB] (w2) circle (2.6pt); \node[surfB,left=3pt] at (w2) {$w_2$};
\node[surfB] at (5.2,3.3) {three vertices of type $(1,2)$};
\end{tikzpicture}
\caption{The essential-surface complex of $6_2$
(Proposition~\ref{prop:six-two} and Theorem~\ref{thm:six-two}): three
isolated vertices and a slope-$8$ layer which is a path with three
vertices.  The involution $\rho$ exchanging the two ends of the path is
induced by the symmetry of period two and by one of the two strong
inversions; the other strong inversion acts trivially.}
\label{fig:six-two}
\end{figure}

\begin{theorem}[The knot $6_2$: image and kernel]
\label{thm:six-two}
Let $K=K_{3/11}=6_2$ and $E=E(K)$, with the vertices of
Proposition~\ref{prop:six-two}.
\begin{enumerate}[label=\textup{(\arabic*)},leftmargin=*]
\item $\ES(E)$ is the disjoint union of three isolated vertices and
the path $w_0 - w_1 - w_2$; in particular $\dim\ES(E)=1$.
\item $\Aut(\ES(E))\cong S_3\times\mathbb Z/2$, of order $12$, whereas
\[
  \Aut(\ES^{\mathrm{per}}(E))=\Aut(\CES(E))=\langle\rho\rangle
  \cong\mathbb Z/2,
\]
where $\rho$ exchanges $w_0$ and $w_2$ and fixes the other four
vertices.
\item For each of $X=\ES(E)$, $\ES^{\mathrm{per}}(E)$, $\CES(E)$ the
image of $\Phi_X$ is $\langle\rho\rangle$.  The automorphism $\rho$ is
induced by the symmetry of period two, $\tau_{(1,0)}$, and by the
strong inversion $\tau_{(0,1)}$, while the strong inversion
$\tau_{(1,1)}$ acts trivially.
\end{enumerate}
Consequently the natural action is surjective at the peripheral and
characteristic levels,
\[
  \Mod^{\pm}(E)/\mathcal K_X(E)\ \cong\ \Aut(X)\ \cong\ \mathbb Z/2
  \qquad(X=\ES^{\mathrm{per}}(E),\ \CES(E)),
\]
and $\mathcal K_X(E)$ has index two in $\Mod^{\pm}(E)=\Mod^{+}(E)$.
Since $\Mod^{\pm}(E)=V$ is a Klein four group \cite{KodamaSakuma}, the
kernel is the group of order two generated by the strong inversion
$\tau_{(1,1)}$.
\end{theorem}

\begin{remark}[What the example shows]
\label{rem:six-two-open}
Theorem~\ref{thm:six-two} is the only case in this article in which
an orientation-preserving mapping class of finite order is visible
(the twist of Theorem~\ref{thmC} has infinite order): for slope-separated
knots the action sees at most the chirality
(Theorem~\ref{thm:visibility}), whereas here a symmetry of period two
and a strong inversion act nontrivially, inside a single slope layer.
It is also a case in which the peripheral refinement is exactly right:
the bare complex has twelve automorphisms, of which two are geometric,
and after adding surface type and boundary slope every automorphism is
geometric although the image is not trivial.  The kernel is again a
finite group which is not generated by twists.

The proof in \cite{OzawaTB} uses three tools which are available for every two-bridge
knot: components of carried surfaces produce edges inside a slope
layer (Proposition~\ref{prop:six-two}); doubles of one-sided
spanning surfaces are pairwise non-adjacent
(Lemma~\ref{lem:doubles-not-adjacent}); and the standard symmetries act
on the Hatcher--Thurston coordinates by
Proposition~\ref{prop:two-bridge-symmetry-action}.  The next subsections state that
the first and the third of these tools determine the complex of every
two-bridge knot.
\end{remark}

\subsection{The decomposition theorem}
\label{subsec:decomposition}

The calculation for $6_2$ is an instance of a general principle, which
reduces the whole complex of a two-bridge knot to the surfaces carried
by one branched surface at a time.  It rests on the observation that
\textup{(HT2)} and \textup{(HT3)} hold for disconnected surfaces: the
surfaces $S_n(n_1,\dots,n_{k-1})$ are not assumed connected in
\cite[Theorem~1]{HatcherThurston}, and the proof there does not use
connectedness.

For an admissible expansion $\beta$ let $\ES_\beta(E)$ be the full
subcomplex of $\ES(E)$ spanned by the vertices carried by
$\Sigma_\beta$.  A component of a carried surface is again carried,
and is determined up to isotopy by its weights \cite{FloydOertel}; if
it has $m$ sheets and $m_j$ copies of $Q_j^{\mathrm{in}}$, it is
$S_m(m_1,\dots,m_{k-1})$.

\begin{theorem}[Decomposition]
\label{thm:two-bridge-decomposition}
Let $K=K_{p/q}$ be a two-bridge knot and $E=E(K)$.
\begin{enumerate}[label=\textup{(\arabic*)},leftmargin=*]
\item $\ES(E)$ is the disjoint union of the subcomplexes $\ES_\beta(E)$,
$\beta$ ranging over the admissible expansions of $p/q$: vertices
carried by different branched surfaces are never adjacent, even if they
have the same boundary slope.
\item Distinct vertices $v_0,\dots,v_d$ of $\ES_\beta(E)$ span a
simplex exactly when some surface carried by $\Sigma_\beta$ has
$d+1$ components, representing $v_0,\dots,v_d$.
\end{enumerate}
\end{theorem}

In particular every slope layer $\ES_s(E)$ is the disjoint union of the
complexes $\ES_\beta(E)$ with
$2(\sigma(\beta)-\sigma(\beta_{\mathrm{ev}}))=s$, and
$\ES_{\beta_{\mathrm{ev}}}(E)=\operatorname{IS}(K)=\operatorname{MS}(K)$
by Proposition~\ref{prop:two-bridge-finite}(1).  For $6_2$ the theorem
gives a second proof that $w_0$ and $w_2$ are not adjacent: a carried
surface with components $S_2(0)$ and $S_2(2)$ would be $S_4(2)$, but
$S_4(2)$ consists of two copies of $S_2(1)$.

\subsection{The sheet rule}
\label{subsec:sheet-rule}

By Theorem~\ref{thm:two-bridge-decomposition} the complex is known as
soon as the components of every carried surface are known.  Hatcher and
Thurston analyse this ``only far enough'' to obtain their
Proposition~1 \cite[\S5]{HatcherThurston}; the following normal form
makes it completely explicit.

In the picture of \cite[Figs.~2 and~3]{HatcherThurston} the plumbing
spheres are parallel horizontal planes and the band $B_i$ stands
between the $(i-1)$-st and the $i$-th of them.  At each of its
plumbing squares, number the $n$ sheets over $B_i$ by $1,\dots,n$,
starting from the sheet nearest to $Q^{\mathrm{in}}$.

\begin{lemma}[Sheet rule]
\label{lem:sheet-rule}
Let $S=S_n(n_1,\dots,n_{k-1})$ be carried by $\Sigma_\beta$.
\begin{enumerate}[label=\textup{(S\arabic*)},leftmargin=*]
\item If $b_i$ is even, every sheet over $B_i$ has the same number at
all ends of the strips of $B_i$.  If $b_i$ is odd, this holds along
one strip of $B_i$, and along the other strip (the only one if
$i\in\{1,k\}$) the sheet numbered $a$ at one end is numbered $n+1-a$
at the other.
\item At the $j$-th plumbing, the sheet numbered $a$ of $B_j$ and the
sheet numbered $a'$ of $B_{j+1}$ are attached to the same copy of a
square if and only if $a+a'\equiv n_j+1\pmod n$.  The copy is a copy
of $Q_j^{\mathrm{in}}$ if $a\le n_j$, equivalently $a'\le n_j$, and a
copy of $Q_j^{\mathrm{out}}$ otherwise.
\end{enumerate}
Consequently the components of $S$ correspond to the classes of the
equivalence relation on $\{1,\dots,k\}\times\mathbb Z/n$ generated by
\[
  (j,a)\sim(j+1,\,n_j+1-a)\quad(1\le j<k),\qquad
  (i,a)\sim(i,\,1-a)\quad(b_i\ \text{odd}),
\]
and the component of a class $C$ is $S_m(m_1,\dots,m_{k-1})$, where
$m=\#\{a:(1,a)\in C\}$ and
$m_j=\#\{a\in\{1,\dots,n_j\}:(j,a)\in C\}$.
\end{lemma}

The rule can be tested against the literature: it reproduces all the
numbers of connected $n$-sheeted surfaces in the table of
\cite[p.~231]{HatcherThurston}, including the row
$7,8,4,4,8,8,12,8$ for $5/13$, and with \textup{(HT3)} it reproduces
the numbers of incompressible Seifert surfaces of all two-bridge knots
with at most ten crossings found by Kakimizu
\cite{KakimizuClassification}.

\begin{proposition}[Components and orientability]
\label{prop:components}
Let $S=S_n(n_1,\dots,n_{k-1})$ be carried by $\Sigma_\beta$.  Put
$t_1=0$ and $t_{j+1}=n_j+1-t_j$, and for each odd $b_j$ let
$r_j=(-1)^{j-1}(1-2t_j)$.
\begin{enumerate}[label=\textup{(\arabic*)},leftmargin=*]
\item If all $b_i$ are even, $S$ has $n$ components, each
single-sheeted.
\item Otherwise let $g$ be the greatest common divisor of $n$ and of
all differences $r_j-r_{j'}$ ($g=n$ if only one $b_j$ is odd).  Then
$S$ has $\lceil g/2\rceil$ components.  In particular $S$ is connected
if and only if $g\le2$.
\item A component of $S$ with $m$ sheets is orientable if and only if
all $b_i$ are even or $m$ is even.
\end{enumerate}
\end{proposition}

Part (2) contains \cite[Proposition~1]{HatcherThurston}, and part (3)
extends Lemma~\ref{lem:carried-basic}(4).  Together with
Theorem~\ref{thm:two-bridge-decomposition} and
Proposition~\ref{prop:two-bridge-symmetry-action} this answers items
\textup{(B1)}, \textup{(B2)} and, for the standard symmetries,
\textup{(B3)} of Problem~\ref{prob:two-bridge-rigidity}: \emph{the
vertices of $\ES_\beta(E)$ are the classes under \textup{(HT3)} of the
connected surfaces $S_n(n_1,\dots,n_{k-1})$ with $n=1$ if
$\beta=\beta_{\mathrm{ev}}$ and $n$ even otherwise; two of them are
adjacent exactly when they are the two components of a carried
surface; and $V$ acts by
Proposition~\ref{prop:two-bridge-symmetry-action}.}

\begin{remark}[The number of isotopy classes]
\label{rem:HT-count}
The number of isotopy classes of $n$-sheeted surfaces carried by
$\Sigma_\beta$ is the number of classes of $\{0,\dots,n\}^{k-1}$ under
the moves of \textup{(HT3)}.  It equals $(n+1)^{l-1}$, where $l\ge1$ is
the number of $b_i$ with $|b_i|>2$, as stated in
\cite[p.~230]{HatcherThurston}, when no term $\pm2$ lies between two
terms of absolute value greater than $2$, but it is larger otherwise.
For example the all-even expansion $[4,2,2,4]$ of $10/33$ has the
invariant $n_1-n_2+n_3$ and carries four Seifert surfaces, not two;
indeed $9_{10}=K_{10/33}$ has four incompressible Seifert surfaces
\cite{KakimizuClassification}.
\end{remark}

\subsection{Three classes of two-bridge knots}
\label{subsec:three-classes}

For an admissible expansion $\beta$ let $o(\beta)$ be the number of its
odd terms.  By Proposition~\ref{prop:two-bridge-finite},
$\ES_\beta(E)$ is finite exactly when $o(\beta)\le1$.  We divide the
two-bridge knots into three classes:
\begin{itemize}[leftmargin=*]
\item \emph{Class I}: every admissible expansion is thin (the
slope-separated knots of
Theorem~\ref{thm:two-bridge-slope-separated});
\item \emph{Class II}: $o(\beta)\le1$ for every admissible expansion,
but some expansion is not thin ($\ES(E)$ is finite, with at least one
edge);
\item \emph{Class III}: $o(\beta)\ge2$ for some admissible expansion
($\ES(E)$ is infinite).
\end{itemize}
Every two-bridge knot or its mirror image is $K_{p/q}$ with
$0<p<q/2$; write
\[
  \frac qp=[a_1,\dots,a_m]_+
  =a_1+\cfrac1{a_2+\cfrac1{\ddots+\cfrac1{a_m}}},
  \qquad a_i\ge1,\ a_1\ge2,\ a_m\ge2 ,
\]
for the ordinary continued fraction; $a_1+\dots+a_m$ is the crossing
number of $K$.

\begin{theorem}[The three classes]
\label{thm:three-classes}
Let $K=K_{p/q}$ be a two-bridge knot with $0<p<q/2$ and
$q/p=[a_1,\dots,a_m]_+$ as above.
\begin{enumerate}[label=\textup{(\arabic*)},leftmargin=*]
\item $K$ is in Class~I if and only if $m=1$, or $m=2$ and
$2\in\{a_1,a_2\}$.
\item $K$ is in Class~II if and only if one of the following holds:
\begin{enumerate}[label=\textup{(\alph*)}]
\item $q/p=[a,b]_+$ with $a,b\ge3$;
\item $q/p=[a,1,b]_+$ with $a,b\ge2$;
\item $q/p=[a,1,1,b]_+$ with $a,b\ge3$ odd.
\end{enumerate}
\item In all other cases $K$ is in Class~III; in particular this is so
whenever $m\ge5$.
\end{enumerate}
\end{theorem}

Since $K$ is a knot, $a$ and $b$ are not both odd in \textup{(a)} and
not both even in \textup{(b)}.  The knots in \textup{(a)} and
\textup{(b)} are the double twist knots which are not twist knots.

Table~\ref{tab:classes} lists the two-bridge knots with at most eight
crossings.  The proportion of Class~III grows quickly with the
crossing number: among the two-bridge knots with $10$ crossings, one is
in Class~I, five are in Class~II and $39$ in Class~III.

\begin{table}[htbp]
\centering
\caption{The two-bridge knots with at most eight crossings, by class.}
\label{tab:classes}
\small
\renewcommand{\arraystretch}{1.2}
\begin{tabular}{@{}lll@{}}
\toprule
class & knots & $\ES(E)$\\
\midrule
I & $3_1,4_1,5_1,5_2,6_1,7_1,7_2,8_1$ & two or three isolated vertices\\
II & $6_2,7_3,7_4,8_2,8_3,8_4,8_9$ & finite, Theorem~\ref{thm:class-two}\\
III & $6_3,7_5,7_6,7_7$, $8_6$--$8_8$, $8_{11}$--$8_{14}$ & infinite\\
\bottomrule
\end{tabular}
\end{table}

\subsection{Class II: complete results}
\label{subsec:class-two}

We write $P_r$ for the path with $r$ vertices, regarded as a
one-dimensional complex; $P_1$ is a point.

\begin{lemma}[Three local models]
\label{lem:local-models}
Let $\beta=[b_1,\dots,b_k]$ be an admissible expansion with
$o(\beta)\le1$ and with exactly two terms $b_i,b_{i'}$, $i<i'$, of
absolute value greater than $2$.
\begin{enumerate}[label=\textup{(\arabic*)},leftmargin=*]
\item If $i'=i+1$ and both terms are even, then
$\ES_\beta(E)\cong P_2$, with vertices $[S_1(n_i)]$, $n_i\in\{0,1\}$,
the other parameters being $0$.
\item If $i'=i+1$ and one of the two terms is odd, then
$\ES_\beta(E)\cong P_3$, with vertices $w_0,w_1,w_2$ in this order,
$w_c=[S_2(n_i=c)]$.
\item If $i'=i+2$ and one of the two terms is odd, then
$\ES_\beta(E)\cong P_5$, with vertices $u_{-2},\dots,u_2$ in this
order, where $u_d$ is the class of the surfaces
$S_2(n_i,n_{i+1})$ with $n_i-n_{i+1}=d$.
\end{enumerate}
In each case $\tau_v\in V$ acts on $\ES_\beta(E)$ trivially if
$v=c_i(\beta)$, and as the reflection of the path otherwise.
\end{lemma}

\begin{theorem}[Class II]
\label{thm:class-two}
Let $K=K_{p/q}$, $0<p<q/2$, be a two-bridge knot of Class~II and
$E=E(K)$.  Then $\ES(E)$ is the disjoint union of the complexes in
Table~\ref{tab:class-two}; in particular it is a graph, each of whose
components is a path with $1$, $2$, $3$ or $5$ vertices.  Let
$\iota_{\mathrm t}=\tau_{(0,1)}$ and $\iota_{\mathrm b}=\tau_{(1,p)}$
be the strong inversions whose axes meet the upper and the lower
bridges, and $\gamma=\tau_{(1,p+1)}$ the symmetry of period two.
\begin{enumerate}[label=\textup{(\arabic*)},leftmargin=*]
\item \textup{($[a,b]_+$ or $[a,1,b]_+$ with a component $P_2$.)}
$\gamma$ acts trivially and $\iota_{\mathrm t},\iota_{\mathrm b}$
exchange the two Seifert surfaces.
\item \textup{($[a,b]_+$ or $[a,1,b]_+$ with a component $P_3$.)}
One of the two strong inversions acts trivially --- namely
$\iota_{\mathrm b}$ if $a$ is even in $[a,b]_+$ or odd in
$[a,1,b]_+$, and $\iota_{\mathrm t}$ otherwise --- and the other
two elements of $V$ reflect $P_3$.
\item \textup{($[a,1,1,b]_+$.)}  $\iota_{\mathrm t}$ reflects the
components carried by $[a+1,2,-b]$ and $[a+1,3,2,\dots,2]$ and fixes
the other vertices; $\iota_{\mathrm b}$ reflects the components
carried by $[a,-2,-(b+1)]$ and $[-2,\dots,-2,-3,-(b+1)]$ and fixes the
other vertices; $\gamma$ reflects all four.  In particular $V$ acts
faithfully.
\end{enumerate}
Consequently, if $a\ne b$, so that $\Mod^{\pm}(E)=V$
\cite{SakumaSurvey,KodamaSakuma}:
\begin{itemize}[leftmargin=*]
\item in cases \textup{(1)} and \textup{(2)},
$\Image\Phi_X\cong\mathbb Z/2\cong\Aut(X)$ for
$X=\ES^{\mathrm{per}}(E),\CES(E)$, the action is surjective, and the
kernel has order two, generated by $\gamma$ in case \textup{(1)} and
by a strong inversion in case \textup{(2)};
\item in case \textup{(3)}, $\Phi_X$ is injective for all three
objects $X$, and its image is a Klein four group in
$\Aut(\ES^{\mathrm{per}}(E))\cong(\mathbb Z/2)^4$.
\end{itemize}
If $a=b$, which occurs only in cases \textup{(1)} and \textup{(3)},
then $\Mod^{\pm}(E)$ has order eight; in case \textup{(1)} the action
on $\ES^{\mathrm{per}}(E)$ is surjective, with
$\Aut(\ES^{\mathrm{per}}(E))\cong(\mathbb Z/2)^2$ and kernel
$\langle\gamma\rangle$, and in case \textup{(3)} it is injective.
\end{theorem}

\begin{table}[htbp]
\centering
\caption{The complex $\ES(E(K_{p/q}))$ for the knots of Class~II.  Each
row is one admissible expansion $\beta$, with the complex
$\ES_\beta(E)$, the common boundary slope and the common type of its
vertices.  Exponents denote repetition.  The parity conditions are on
$(a,b)$; the case $(\text{odd},\text{even})$ of $[a,b]_+$ is obtained
from $(\text{even},\text{odd})$ by reversing the expansions, and
similarly for $[a,1,b]_+$.}
\label{tab:class-two}
\small
\renewcommand{\arraystretch}{1.18}
\begin{tabular}{@{}lllll@{}}
\toprule
$q/p$ & $\beta$ & $\ES_\beta(E)$ & slope & type\\
\midrule
$[a,b]_+$, (even, even)
 & $[a,-b]$ & $P_2$ & $0$ & $(1,1)$\\
 & $[a+1,2^{\,b-1}]$ & $P_1$ & $2b$ & $(b-1,2)$\\
 & $[(-2)^{a-1},-(b+1)]$ & $P_1$ & $-2a$ & $(a-1,2)$\\
\midrule
$[a,b]_+$, (even, odd)
 & $[(-2)^{a-1},-(b+1)]$ & $P_1$ & $0$ & $(a/2,1)$\\
 & $[a,-b]$ & $P_3$ & $2a$ & $(1,2)$\\
 & $[a+1,2^{\,b-1}]$ & $P_1$ & $2(a+b)$ & $(b-1,2)$\\
\midrule
$[a,1,b]_+$, (odd, odd)
 & $[a+1,b+1]$ & $P_2$ & $0$ & $(1,1)$\\
 & $[a,(-2)^{b}]$ & $P_1$ & $-2(b+1)$ & $(b,2)$\\
 & $[(-2)^{a},b]$ & $P_1$ & $-2(a+1)$ & $(a,2)$\\
 & $[(-2)^{a-1},-3,(-2)^{b-1}]$ & $P_1$ & $-2(a+b+1)$ & $(a+b-2,2)$\\
\midrule
$[a,1,b]_+$, (odd, even)
 & $[(-2)^{a},b]$ & $P_1$ & $0$ & $(\tfrac{a+1}2,1)$\\
 & $[a+1,b+1]$ & $P_3$ & $2(a+1)$ & $(1,2)$\\
 & $[a,(-2)^{b}]$ & $P_1$ & $2(a-b)$ & $(b,2)$\\
 & $[(-2)^{a-1},-3,(-2)^{b-1}]$ & $P_1$ & $-2b$ & $(a+b-2,2)$\\
\midrule
$[a,1,1,b]_+$, (odd, odd)
 & $[(-2)^{a},2^{\,b}]$ & $P_1$ & $0$ & $(\tfrac{a+b}2,1)$\\
 & $[a+1,2,-b]$ & $P_5$ & $2(a-b+1)$ & $(2,2)$\\
 & $[a,-2,-(b+1)]$ & $P_5$ & $2(a-b-1)$ & $(2,2)$\\
 & $[a+1,3,2^{\,b-1}]$ & $P_3$ & $2(a+1)$ & $(b,2)$\\
 & $[(-2)^{a-1},-3,-(b+1)]$ & $P_3$ & $-2(b+1)$ & $(a,2)$\\
\bottomrule
\end{tabular}
\end{table}

\begin{remark}
\label{rem:class-two-examples}
For $6_2$, where $q/p=[3,1,2]_+$, the theorem gives back
Theorem~\ref{thm:six-two}.  The knots $7_3$, $8_2$, $8_4$ have the
same pattern, $P_3$ and two or three points.  For $7_4$ and $8_3$ the
complex is the edge $\operatorname{MS}(K)$ and three, respectively
two, points.  For $8_9$, where $q/p=[3,1,1,3]_+$, the complex has $17$
vertices: the fibre, two copies of $P_5$ of slopes $\pm2$ and two
copies of $P_3$ of slopes $\pm8$; here $\Mod^{\pm}(E)$, a group of
order eight, acts faithfully.  Thus the kernel of the action, which is
everything for the knots of Class~I other than $4_1$, has order two
for the double twist knots of Class~II and is trivial for the knots
$[a,1,1,b]_+$.
\end{remark}

\subsection{Class III: first results}
\label{subsec:class-three}

For an expansion with at least two odd terms the complex
$\ES_\beta(E)$ is infinite, and
Theorem~\ref{thm:two-bridge-decomposition},
Lemma~\ref{lem:sheet-rule} and Proposition~\ref{prop:components}
describe it combinatorially.  We first determine it in the simplest
case, and treat the general case in
Section~\ref{subsec:general-structure}.

\begin{proposition}[Two adjacent odd terms]
\label{prop:adjacent-odd}
Let $\beta$ be an admissible expansion in which $b_i$ and $b_{i+1}$
are odd and all other terms are $\pm2$.  Then $\ES_\beta(E)$ is an
infinite discrete set.  Its vertices are in bijection with the
rational numbers in $[0,2]$: the vertex $v_{c/\nu}$, with
$\gcd(c,\nu)=1$, is represented by the surface with $2\nu$ sheets and
$n_i=c$, the other parameters being $0$; it has genus $1+\nu(k-2)$
and $2\nu$ boundary components.  The symmetry $\tau_v$ fixes every
vertex if $v=c_i(\beta)$, and sends $v_x$ to $v_{2-x}$ otherwise.
\end{proposition}

\begin{theorem}[The knot $6_3$]
\label{thm:six-three}
Let $K=K_{5/13}=6_3$ and $E=E(K)$.  Then $\ES(E)$ is an infinite
discrete set:
\begin{center}
\upshape\small
\renewcommand{\arraystretch}{1.2}
\begin{tabular}{@{}llll@{}}
\toprule
expansion & slope & vertices & type\\
\midrule
$[-2,-2,2,2]$ & $0$ & one, the fibre & $(2,1)$\\
$[3,2,-2]$ & $2$ & one & $(2,2)$\\
$[2,-2,-3]$ & $-2$ & one & $(2,2)$\\
$[3,3,2]$ & $6$ & $v^{+}_{c/\nu}$, $c/\nu\in\mathbb Q\cap[0,2]$ & $(1+\nu,2\nu)$\\
$[-2,-3,-3]$ & $-6$ & $v^{-}_{c/\nu}$, $c/\nu\in\mathbb Q\cap[0,2]$ & $(1+\nu,2\nu)$\\
\bottomrule
\end{tabular}
\end{center}
The strong inversion $\tau_{(0,1)}$ fixes every $v^+_x$ and sends
$v^-_x$ to $v^-_{2-x}$; the strong inversion $\tau_{(1,1)}$ fixes every
$v^-_x$ and sends $v^+_x$ to $v^+_{2-x}$; the symmetry of period two
does both; and every orientation-reversing mapping class exchanges the
two infinite families.  Consequently the action of
$\Mod^{\pm}(E)$, a group of order eight, on $\ES(E)$ is faithful.
\end{theorem}

Thus $6_2$ and $6_3$, the first knots outside Class~I, already show
three different behaviours of the kernel, and
$\ES(E)$ can be an infinite complex without any edge.

\begin{figure}[htb]
\centering
\begin{tikzpicture}[scale=1.0,every node/.style={font=\small}]
\draw[greyish,-{Stealth[length=5pt]}] (-5.6,0) -- (5.9,0);
\node[greyish,right] at (5.9,0) {slope};
\foreach \x/\l in {-4.4/{$-8$},-1.2/{$-2$},0/{$0$},1.2/{$2$},4.4/{$8$}}{
  \draw[greyish] (\x,-0.08)--(\x,0.08);
  \node[greyish,below=4pt] at (\x,0) {\l};
}
\fill[surfA] (0,0) circle (2.6pt);
\node[surfA,above=4pt] at (0,0) {$(3,1)$};
\foreach \x/\c in {-1.2/surfC,1.2/surfC}{
  \draw[thick] (\x,0.7)--(\x,3.1);
  \foreach \y in {0.7,1.3,1.9,2.5,3.1}{\fill[\c] (\x,\y) circle (2.4pt);}
}
\foreach \x/\c in {-4.4/surfB,4.4/surfB}{
  \draw[thick] (\x,0.7)--(\x,1.9);
  \foreach \y in {0.7,1.3,1.9}{\fill[\c] (\x,\y) circle (2.4pt);}
}
\node[surfC] at (-1.2,3.45) {$P_5$,\ $(2,2)$};
\node[surfC] at (1.2,3.45) {$P_5$,\ $(2,2)$};
\node[surfB] at (-4.4,2.25) {$P_3$,\ $(3,2)$};
\node[surfB] at (4.4,2.25) {$P_3$,\ $(3,2)$};
\end{tikzpicture}
\caption{The essential-surface complex of the amphichiral knot
$8_9=K_{7/25}$, with $q/p=[3,1,1,3]_+$ (Theorem~\ref{thm:class-two}):
seventeen vertices in five components.  One strong inversion reflects
the two paths on the right, the other one the two paths on the left,
the symmetry of period two all four, and the orientation-reversing
symmetries exchange left and right.  The action of the symmetry group
is faithful.}
\label{fig:eight-nine}
\end{figure}

\subsection{The general structure of the complex carried by one branched surface}
\label{subsec:general-structure}

We now describe $\ES_\beta(E)$ for an arbitrary admissible expansion
$\beta=[b_1,\dots,b_k]$.  Let $o_1<\dots<o_s$ be the indices of the odd
terms, $s=o(\beta)\ge1$.  For $1\le u<s$ the \emph{$u$-th inner block}
is the set of plumbings $o_u,\dots,o_{u+1}-1$ between two consecutive
M\"obius bands; the plumbings before $o_1$ and those from $o_s$ on
form the two \emph{outer blocks}.  For a parameter vector
$c=(c_1,\dots,c_{k-1})$ put
\[
  A_u(c)=\sum_{j=o_u}^{o_{u+1}-1}(-1)^j c_j\qquad(1\le u<s),
  \qquad A(c)=(A_1(c),\dots,A_{s-1}(c)).
\]
A move of \textup{(HT3)} does not change $A(c)$, because a band with two
half-twists has both of its plumbings in the same block.  By
Proposition~\ref{prop:components}(2), for which
$r_{j'}-r_j=2\sum_{i=j}^{j'-1}(-1)^{i-1}n_i$, a surface
$S_{2\nu}(c)$ is connected if and only if
$\gcd(\nu,A_1(c),\dots,A_{s-1}(c))=1$; so the vertices of
$\ES_\beta(E)$ are the classes of the pairs $(\nu,c)$ with $\nu\ge1$,
$c\in\{0,\dots,2\nu\}^{k-1}$ and $\gcd(\nu,A(c))=1$.  For $s=1$ this
means $\nu=1$.

\begin{lemma}[Block invariants]
\label{lem:block-invariants}
Let $X$ be a component of $S=S_n(n_1,\dots,n_{k-1})$, with $m$ sheets
and weights $m_1,\dots,m_{k-1}$.  Then
\[
  \frac{A_u(m_1,\dots,m_{k-1})}{m}=\frac{A_u(n_1,\dots,n_{k-1})}{n}
  \qquad(1\le u<s).
\]
\end{lemma}

To state the adjacency rule, for $\theta\in\mathbb Z/4$ and an integer
$n\ge0$ let $D(\theta,n)$ be the number of $a\in\{1,\dots,n\}$ with
$a-\theta\equiv0,1\pmod4$ minus the number of those with
$a-\theta\equiv2,3\pmod4$; it depends only on $n$ modulo $4$ and takes
the values
\[
\begin{array}{c|cccc}
 n\bmod4 & 0&1&2&3\\\hline
 \theta=0 & 0&1&0&-1\\
 \theta=1 & 0&1&2&1\\
 \theta=2 & 0&-1&0&1\\
 \theta=3 & 0&-1&-2&-1
\end{array}
\]

\begin{theorem}[Structure of $\ES_\beta(E)$]
\label{thm:general-structure}
Let $\beta$ be an admissible expansion with $s=o(\beta)\ge1$ odd terms.
\begin{enumerate}[label=\textup{(\arabic*)},leftmargin=*]
\item \textup{(Vertices.)}  The vertices of $\ES_\beta(E)$ are the
classes under \textup{(HT3)} of the pairs $(\nu,c)$ with $\nu\ge1$,
$c\in\{0,\dots,2\nu\}^{k-1}$ and $\gcd(\nu,A(c))=1$.  The vertex
$[\nu,c]$ is represented by $S_{2\nu}(c)$ and has genus $1+\nu(k-2)$
and $2\nu$ boundary components.
\item \textup{(Fibres.)}  If two vertices $[\nu,c]$ and $[\nu',c']$
are adjacent, then $\nu=\nu'$ and $A(c)=A(c')$.  Thus $\ES_\beta(E)$ is
the disjoint union of the full subcomplexes $\mathcal K_{\nu,A}$
spanned by the vertices with given $\nu$ and $A(c)=A$.  Each
$\mathcal K_{\nu,A}$ is a finite connected flag complex, so these are
the connected components of $\ES_\beta(E)$; they are indexed by the
pairs $(\nu,A)$ with $\gcd(\nu,A)=1$ for which the polytope
\[
  \mathsf P_{\nu,A}=\{c\in[0,2\nu]^{k-1}:\ A_u(c)=A_u\ (1\le u<s)\}
\]
is nonempty, and the vertices of $\mathcal K_{\nu,A}$ are the classes
of the lattice points of $\mathsf P_{\nu,A}$.
\item \textup{(Edges.)}  Two distinct vertices $[\nu,c]$,
$[\nu,c']$ are adjacent if and only if, for suitable representatives
and with $n_j=c_j+c'_j$, there are $\theta_1,\dots,\theta_k\in\mathbb
Z/4$ such that
\[
  \theta_j+\theta_{j+1}\equiv n_j\pmod4,\qquad
  \theta_{o_u}\in\{0,2\}\ \ (1\le u\le s),\qquad
  c_j-c'_j=D(\theta_j,n_j)
\]
for all $j$.
\item \textup{(Block flips.)}  For every union $F$ of blocks, the map
replacing $c_j$ by $2\nu-c_j$ for $j\in F$ induces an automorphism of
$\ES_\beta(E)$ preserving the number of sheets.  These automorphisms
form a group $(\mathbb Z/2)^{s+1}$, and $V$ acts through four of them.
\end{enumerate}
\end{theorem}

For $s=1$ there is only $\mathcal K_{1,\varnothing}$, on the classes of
$\{0,1,2\}^{k-1}$; Lemma~\ref{lem:local-models} lists the cases $P_3$
and $P_5$.  For two adjacent odd terms $\mathsf P_{\nu,A}$ is a point,
which is Proposition~\ref{prop:adjacent-odd}.  In all examples we have
computed, $\mathcal K_{\nu,A}$ has Euler characteristic $1$ and its
dimension is that of $\mathsf P_{\nu,A}$, as if it were a
triangulation of that polytope; we have not tried to prove this.  The
same arguments apply to $\beta_{\mathrm{ev}}$, with single-sheeted
surfaces: $\operatorname{MS}(K)$ is the connected flag complex on the
classes of $\{0,1\}^{k-1}$ in which $\varepsilon$ and $\varepsilon'$
are adjacent when $S_2(\varepsilon+\varepsilon')$ has components
$S_1(\varepsilon)$, $S_1(\varepsilon')$; compare Sakuma's description
\cite{SakumaMinimalGenus}.

\subsection{The kernel for all two-bridge knots, and the additional symmetries}
\label{subsec:kernel-all}

Call a plumbing $j$ of $\beta$ \emph{essential} if there are terms
$b_i,b_{i'}$ of absolute value different from $2$ with $i\le j<i'$.
The parameters at the other plumbings can be changed freely by
\textup{(HT3)}.

\begin{lemma}
\label{lem:trivial-action}
The symmetry $\tau_v$ fixes every vertex of $\ES_\beta(E)$ if and only
if $c_j(\beta)=v$ for every essential plumbing $j$ of $\beta$.
\end{lemma}

Let $0<p<q/2$ and $q/p=[a_1,\dots,a_m]_+$.  Let $P_j/Q_j$ be the
convergents of the continued fraction $p/q=1/[a_1,\dots,a_m]_+$, with
$P_0/Q_0=0/1$.  Let
$\kappa_j\in(\mathbb Z/2)^2$ be the class of $(Q_j,P_j)$; consecutive
classes are different.

\begin{theorem}[The kernel]
\label{thm:two-bridge-kernel}
Let $K=K_{p/q}$, $0<p<q/2$, be a hyperbolic two-bridge knot,
$q/p=[a_1,\dots,a_m]_+$, and let $X$ be one of $\ES(E)$,
$\ES^{\mathrm{per}}(E)$, $\CES(E)$.
\begin{enumerate}[label=\textup{(\arabic*)},leftmargin=*]
\item If $K$ is a twist knot \textup{(}Class~I\textup{)}, then
$\mathcal K_X(E)=\Mod^{+}(E)$.
\item If $q/p=[a,b]_+$ with $a,b\ge3$ or $q/p=[a,1,b]_+$
\textup{(}the double twist knots of Class~II\textup{)}, then
$\mathcal K_X(E)$ has order two; its generator is the element of $V$
described in Theorem~\ref{thm:class-two}.
\item In all other cases, that is, if $m\ge4$, or $m=3$ and
$a_2\ge2$, the group $\Mod^{\pm}(E)$ acts faithfully on $\ES(E)$.
\end{enumerate}
In particular the action is faithful for every knot of Class~III.
\end{theorem}

\begin{proposition}[Level-reversing symmetries]
\label{prop:level-reversing}
Let $\beta=[b_1,\dots,b_k]$ be an admissible expansion of $p/q$ and
$\beta^{\mathrm{rev}}=[b_k,\dots,b_1]$.
\begin{enumerate}[label=\textup{(\arabic*)},leftmargin=*]
\item Turning the tower of plumbed bands upside down defines a
homeomorphism $g_\beta$ from $(S^3,K_{p/q})$ to $(S^3,K_{p'/q})$,
$pp'\equiv1\pmod q$, with
\[
  g_\beta\bigl(S_n(n_1,\dots,n_{k-1})\bigr)
  =S_n(n_{k-1},\dots,n_1)\quad\text{carried by }
  \Sigma_{\beta^{\mathrm{rev}}} .
\]
If $p^2\equiv1\pmod q$ it is an orientation-preserving symmetry of
$K$; if $p^2\equiv-1$, its composition with a reflection is an
orientation-reversing symmetry carrying $\Sigma_\beta$ to
$\Sigma_{-\beta^{\mathrm{rev}}}$.
\item Every symmetry $h\notin V$ of $K$ acts on the surfaces carried
by $\Sigma_\beta$ as $g_\beta$, respectively its composition with the
reflection, followed by a block flip of
Theorem~\ref{thm:general-structure}(4); if $o(\beta)\le1$ the block
flip is induced by an element of $V$.
\end{enumerate}
\end{proposition}

\begin{remark}
\label{rem:level-reversing-open}
When $o(\beta)\le1$, Proposition~\ref{prop:level-reversing} determines
the set of permutations of the vertices carried by $\Sigma_\beta$ and
$\Sigma_{\beta^{\mathrm{rev}}}$ induced by the four symmetries outside
$V$.  It does not say which of the four induces which permutation
simultaneously for two different expansions; equivalently, we have not
determined the element $g_\beta^{-1}g_{\beta'}\in V$.  This does not
affect the kernel, the order of the image, or the results of
Theorem~\ref{thm:class-two}.
\end{remark}

\subsection{Automorphism groups in Class III}
\label{subsec:aut-class-three}

\begin{theorem}[Class III is never rigid]
\label{thm:class-three-aut}
Let $K$ be a two-bridge knot of Class~III.  Then
$\Aut(\ES^{\mathrm{per}}(E))$ is uncountable, while $\Mod^{\pm}(E)$ is
a finite group acting faithfully.  More precisely, let $\beta$ be an
admissible expansion with $o(\beta)=s\ge2$.  Then:
\begin{enumerate}[label=\textup{(\arabic*)},leftmargin=*]
\item the group $(\mathbb Z/2)^{s+1}$ of block flips acts on
$\ES_\beta(E)$ preserving surface type and boundary slope, and only a
subgroup of order at most four is induced by mapping classes
preserving $\Sigma_\beta$;
\item for every $\nu\ge2$ the flip of the first inner block, applied
only to the vertices with $2\nu$ sheets, is a nontrivial automorphism
of $\ES^{\mathrm{per}}(E)$; these automorphisms generate a subgroup
isomorphic to $\prod_{\nu\ge2}\mathbb Z/2$.
\end{enumerate}
For $K=6_3$,
\[
  \Aut\bigl(\ES^{\mathrm{per}}(E(6_3))\bigr)\ \cong\
  \Bigl(\prod_{\nu\ge1}S_{N_\nu}\Bigr)\wr\mathbb Z/2 ,
  \qquad N_1=3,\quad N_\nu=2\varphi(\nu)\ \ (\nu\ge2),
\]
where $\varphi$ is Euler's function, and the image of
$\Mod^{\pm}(E)$ has order eight.
\end{theorem}

\begin{remark}[What would restore rigidity]
\label{rem:block-refinement}
The proof of Theorem~\ref{thm:class-three-aut} in \cite{OzawaTB}
shows what the peripheral refinement fails to see in
Class~III: the block invariants $A(c)/\nu$ of
Lemma~\ref{lem:block-invariants}, that is, the connected components of
the complex, are not determined by surface type and boundary slope.
For $6_3$ the invariant is the rational number $c/\nu\in[0,2]$ of
Theorem~\ref{thm:six-three}.  It would be interesting to describe it
intrinsically, for example through the intersection of the surface
with the M\"obius bands of $\Sigma_\beta$, and to decide whether the
complex decorated by it is rigid.
\end{remark}

\begin{remark}[What remains for two-bridge knots]
\label{rem:two-bridge-remaining}
With Theorems~\ref{thm:general-structure},
\ref{thm:two-bridge-kernel} and~\ref{thm:class-three-aut} the complex,
the kernel of the action and the failure of rigidity are determined for
every two-bridge knot.  Three points remain: whether each component
$\mathcal K_{\nu,A}$ is a triangulation of the polytope
$\mathsf P_{\nu,A}$; the correlation between different expansions in
the action of the symmetries outside $V$
(Remark~\ref{rem:level-reversing-open}); and an intrinsic description
of the block invariants (Remark~\ref{rem:block-refinement}).
\end{remark}

\section{Connected sums: a visible twist}
\label{sec:composite}

The first decomposed phenomenon is an essential decomposing annulus.
The underlying surface theory is classical: Eisner proved that a
connected sum of two nonfibered knots has infinitely many pairwise
nonisotopic minimal spanning surfaces \cite{Eisner}; Banks identified
the full product structure and its integer winding coordinate
\cite{BanksConnectedSum}; and Johnson--Pelayo--Wilson place twisting
actions in the broader JSJ description of Kakimizu complexes
\cite{JohnsonPelayoWilson}.  We record the classical uniqueness of the
annulus and then spell out the natural action of its Dehn twist on
Banks's coordinate.  This is a concrete model of the
frontier-crossing issue isolated in
Section~\ref{subsec:frontier-problem}.

Let $K=K_1\mathbin\#K_2$.  A decomposing sphere $S\subset S^3$ meets
$K$ in two points, and $A=S\cap E(K)$ is a properly embedded essential
annulus whose boundary components are meridians; we call it the
\emph{standard decomposing annulus}.

\begin{theorem}[Classical annulus classification]
\label{thm:composite-annulus}
Let $K=K_1\mathbin\#K_2$ with $K_1,K_2$ nontrivial prime knots, not
necessarily distinct, and $E=E(K)$.  Then
\begin{enumerate}[label=(\arabic*),leftmargin=*]
\item every essential annulus in $E$ is isotopic to $A$;
\item $[A]$ is the unique annulus vertex of
$\ES^{\mathrm{per}}(E)$ and is fixed by every peripheral
automorphism;
\item every automorphism of the characteristic refinement
$\CES(E)$ also fixes $[A]$.
\end{enumerate}
\end{theorem}

\begin{proof}
An essential annulus in a knot exterior is either a cabling annulus,
in which case the knot is a torus knot or a cable knot, or it is
meridional and extends across $N(K)$ to a decomposing sphere; see for
example \cite[Lemma~1.2]{KodaOzawa} and the JSJ discussion in
\cite{Budney}.  The knot $K_1\mathbin\#K_2$ is composite, hence
neither a torus knot nor a cable knot, since nontrivial cable knots
are prime \cite{Schubert,Rolfsen}.  So every essential annulus in $E$
extends to a decomposing sphere.

Let $S$ be the standard decomposing sphere and $S'$ any other.  The
disjointness step in the standard proof of Schubert's prime
decomposition theorem \cite{Schubert} isotopes $S'$ in $(S^3,K)$ to be
disjoint from $S$.  If the two spheres were not parallel in
$(S^3,K)$, then $S'$ would lie in one of the two complementary balls
cut off by $S$ and would give a nontrivial connected-sum decomposition
of the corresponding prime factor, a contradiction.  Hence $S'$ is
isotopic to $S$.  This does not require $K_1\ne K_2$: when
$K_1\cong K_2$ the possible exchange of the two equal factors does not
create a second isotopy class of the decomposing sphere.  Intersecting
the isotopy with $E$ proves (1); (2) and (3) follow since surface type
is an automorphism-invariant label.
\end{proof}

The exterior also contains the standard swallow--follow torus obtained
by tubing $A$ along one of the two annuli of
$\partial E\setminus\partial A$; for nontrivial factors this torus is
essential.  Thus Theorem~\ref{thm:composite-annulus} identifies only
the unique annulus vertex, not the whole complex.

\subsection{The twist on Banks's winding coordinate}

Choose a product neighbourhood
$N(A)\cong S^1\times[0,1]\times[-1,1]$ with
$A=S^1\times[0,1]\times\{0\}$ and
$N(A)\cap\partial E=S^1\times\{0,1\}\times[-1,1]$.  A positive Dehn
twist about $A$ is
\[
  t_A(z,s,u)=\bigl(e^{2\pi i\rho(u)}z,\,s,\,u\bigr),
  \qquad \rho(-1)=0,\ \rho(1)=1,
\]
extended by the identity outside $N(A)$.  The endpoint rotations are
both the identity, so this is a homeomorphism of $E$; on $\partial E$
the two induced twists occur along oppositely oriented meridians and
cancel, so $t_A|_{\partial E}$ is isotopic to the identity and
preserves every boundary slope.  Write $\tau_A=[t_A]$.  Reversing the
orientation convention replaces $\tau_A$ by its inverse.

\begin{theorem}[Banks's connected-sum theorem \cite{BanksConnectedSum}]
\label{thm:banks-connected-sum}
Let $L_1,L_2$ be non-split, nonfibered links in $S^3$ and
$L=L_1\mathbin\#L_2$.  Banks constructs an ordered product
triangulation and a simplicial isomorphism whose realization gives
\[
  |\operatorname{MS}(L)|\cong
  |\operatorname{MS}(L_1)|\times|\operatorname{MS}(L_2)|\times\mathbb R .
\]
On vertices, after choices of ordered representatives, this is a
bijection
\[
  \Psi:\operatorname{V}(\operatorname{MS}(L_1))\times
       \operatorname{V}(\operatorname{MS}(L_2))\times\mathbb Z
       \longrightarrow\operatorname{V}(\operatorname{MS}(L)),
\]
the integer recording the winding of the joining rectangle in the
composing region.  The analogous homeomorphism holds for all
incompressible spanning surfaces,
$|\operatorname{IS}(L)|\cong|\operatorname{IS}(L_1)|\times
|\operatorname{IS}(L_2)|\times\mathbb R$.
\end{theorem}

For knots the hypotheses reduce exactly to the nonfibered hypotheses
used below: every knot is non-split, and a nontrivial prime knot is an
admissible factor.  The next statement is the action calculation.  We
have not found its displayed formula stated verbatim in the cited
literature, but it is an immediate naturality consequence of the
winding-coordinate construction and should not be read as a new
classification theorem.  Related twist actions on Kakimizu complexes,
along JSJ tori, are used by Johnson--Pelayo--Wilson and by Banks to
determine the quasi-isometry type
\cite{JohnsonPelayoWilson,BanksQuasi}.

\begin{proposition}[Translation by annular spinning; $=$ Theorem~\ref{thmC}]
\label{prop:kakimizu-spinning}
Let $K=K_1\mathbin\#K_2$ with $K_1,K_2$ nontrivial nonfibered prime
knots.  In the vertex coordinates of
Theorem~\ref{thm:banks-connected-sum}, with the sign convention above,
\[
  \tau_A\cdot\Psi(R_1,R_2,n)=\Psi(R_1,R_2,n+1).
\]
Consequently $\langle\tau_A\rangle\cong\mathbb Z$ acts freely on the
vertex set of $\operatorname{MS}(K)$.  By
Proposition~\ref{prop:kakimizu-subcomplex}, these are also vertices of
$\ES(E)$, and therefore
\[
  \langle\tau_A\rangle\cap
  \ker\bigl(\Mod^{\pm}(E)\to\Aut(\operatorname{MS}(K))\bigr)=1,
  \qquad
  \langle\tau_A\rangle\cap\mathcal K_{\ES}(E)=1 .
\]
In particular annular spinning is not an invisible symmetry of the
full essential-surface complex.
\end{proposition}

\begin{proof}
Banks's vertex map joins fixed representatives in the two factor
exteriors by a rectangle winding $n$ times in the composing region
\cite[Definition~7.1 and Section~7]{BanksConnectedSum}.  We verify
naturality in the model neighbourhood fixed above.  The composing
region is a collar of the decomposing annulus, which we take to be
$N(A)\cong S^1\times[0,1]\times[-1,1]$, with the two factor exteriors
attached along $S^1\times[0,1]\times\{\pm1\}$.  A representative of
$\Psi(R_1,R_2,n)$ meets $N(A)$ in the joining rectangle, a spanning
band meeting each level annulus $S^1\times[0,1]\times\{u\}$ in a
single spanning arc, with total winding $n$ in the $S^1$ direction
measured from $u=-1$ to $u=1$ relative to the reference rectangle.
The homeomorphism $t_A$ is the identity outside $N(A)$, so it fixes
$R_1$ and $R_2$ pointwise; on $N(A)$ it rotates the level $u$ by
$2\pi\rho(u)$, so it carries a rectangle of winding $n$ to one of
winding $n+1$ relative to the same reference.  Iterating gives
$\tau_A^m\Psi(R_1,R_2,n)=\Psi(R_1,R_2,n+m)$, and injectivity of
Banks's vertex map shows this equals $\Psi(R_1,R_2,n)$ only for $m=0$.
\end{proof}

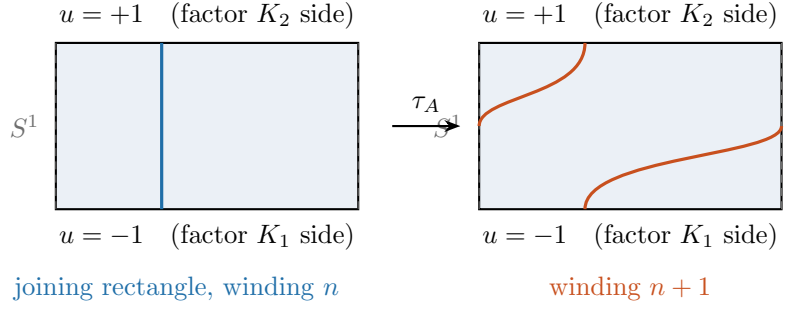
\begin{figure}[htb]
\centering
\begin{tikzpicture}[scale=1.0,every node/.style={font=\small}]
% the composing region as a cylinder, drawn flattened
\foreach \x/\lab in {0/{$n$}, 5.6/{$n+1$}}{
  \begin{scope}[xshift=\x cm]
    \fill[lightfill] (0,0) rectangle (4.0,2.2);
    \draw[thick] (0,0) rectangle (4.0,2.2);
    \node[below] at (2.0,-0.05) {$u=-1$ \ \ (factor $K_1$ side)};
    \node[above] at (2.0,2.25) {$u=+1$ \ \ (factor $K_2$ side)};
    \draw[greyish,densely dashed] (0,0)--(0,2.2);
    \draw[greyish,densely dashed] (4.0,0)--(4.0,2.2);
  \end{scope}
}
% winding n = 0 (straight) on the left
\node[greyish,left] at (-0.08,1.1) {$S^1$};
\draw[surfA,very thick] (1.4,0) -- (1.4,2.2);
\node[surfA] at (1.6,-1.05) {joining rectangle, winding $n$};
% winding n+1 on the right: goes once around
\draw[surfB,very thick] (5.6+1.4,0) .. controls (5.6+1.4,0.7) and (5.6+4.0,0.7) .. (5.6+4.0,1.1);
\draw[surfB,very thick] (5.6+0.0,1.1) .. controls (5.6+0.0,1.5) and (5.6+1.4,1.5) .. (5.6+1.4,2.2);
\node[surfB] at (5.6+2.0,-1.05) {winding $n+1$};
\draw[-{Stealth[length=6pt]},thick] (4.3,1.1) -- (5.3,1.1);
\node[above] at (4.8,1.15) {$\tau_A$};
\end{tikzpicture}
\caption{Proposition~\ref{prop:kakimizu-spinning}.  The Dehn twist
along the decomposing annulus is the identity on both factor
exteriors, and increases by one the winding of the joining rectangle
in the composing region.  Hence it is not invisible on the Kakimizu
layer, even though it fixes every surface lying in a single factor.
This is the smallest concrete instance of the frontier-crossing
phenomenon of Figure~\ref{fig:frontier}.}
\label{fig:winding}
\end{figure}

\begin{remark}[The winding coordinate as global compatibility data]
\label{rem:winding-globality}
Cutting a minimal-genus Seifert surface along $A$ leaves the two
factor coordinates unchanged; changing the gluing by $\tau_A^m$
changes only the relative winding integer.  Thus, for this family of
frontier-crossing surfaces, the integer is exactly the discrepancy
between two fixed local restrictions which must be controlled when
they are reglued.  This is an interpretation, not a further assertion:
no claim is made that the coordinate is intrinsically definable from
$\ES^{\mathrm{per}}(E)$.
\end{remark}

\begin{remark}[Scope]
\label{rem:composite-scope}
Theorem~\ref{thm:composite-annulus} verifies a concrete annular
fragment of intrinsic recognition;
Theorem~\ref{thm:banks-connected-sum} supplies the coordinate model;
and Proposition~\ref{prop:kakimizu-spinning} determines the action of
the annular-twist direction and shows that this infinite cyclic
subgroup is not invisible.  It does not compute the full kernel:
mapping classes coming from the two factor exteriors, factor-exchange
symmetries when $K_1\cong K_2$, and possible finite invisible
symmetries remain.  The nonfibered hypothesis is exactly the regime in
which Banks's model contains the extra $\mathbb Z$ coordinate; if one
factor is fibered, the Kakimizu complex is isomorphic to that of the
other factor and the argument does not determine whether $\tau_A$ is
invisible.  That is a separate, and possibly different, kernel
phenomenon.
\end{remark}

\section{Satellites, the JSJ decomposition, and the frontier problem}
\label{sec:satellite}

Sections~\ref{sec:torus-knots}--\ref{sec:composite} concerned
exteriors whose characteristic frontier is either empty or a single
annulus.  In general a knot exterior decomposes into pieces of
different geometric types, and this is where the three-dimensional
problem departs decisively from the curve-complex model.  This section
collects the classical input, defines the characteristic refinement
$\CES(E)$ used in the statements above, and explains --- in
pictures, with the formal details deferred to
Appendix~\ref{app:realization} --- why realizing a combinatorial
automorphism by a homeomorphism is not a local problem.

\subsection{Classical input}

\begin{theorem}[JSJ input for a knot exterior]
\label{thm:classical-jsj-input}
Let $K$ be a nontrivial knot and $E=E(K)$.  Then:
\begin{enumerate}[label=(\arabic*),leftmargin=*]
\item there is a finite minimal family
$\mathcal T_{\mathrm{JSJ}}(E)=T_1\cup\cdots\cup T_n$ of pairwise
disjoint, pairwise nonparallel essential tori, unique up to isotopy,
such that every component of $E\sslash\mathcal T_{\mathrm{JSJ}}(E)$ is
Seifert fibered or atoroidal;
\item the characteristic submanifold $\Char(E)$ is unique up to
isotopy, and, after deleting boundary-parallel and parallel
redundancies, its frontier $\Fr\Char(E)$ is a finite union of
essential annuli and tori; it contains the Seifert and $I$-bundle
parts through which essential annuli and tori are carried;
\item every self-homeomorphism of $E$ preserves
$\mathcal T_{\mathrm{JSJ}}(E)$, $\Char(E)$, and their incidence data
up to isotopy;
\item by geometrization, each non-Seifert JSJ piece has interior
carrying a complete finite-volume hyperbolic metric.
\end{enumerate}
\end{theorem}

\begin{proof}
A nontrivial knot exterior is compact, orientable, irreducible, and
has incompressible torus boundary.  The assertions about the canonical
torus system and the characteristic submanifold are the
Jaco--Shalen--Johannson theory \cite{JS,Johannson,NeumannSwarup};
canonicality gives invariance under homeomorphisms.  The last
assertion is the geometrization consequence for atoroidal non-Seifert
pieces.  See also Budney's knot- and link-exterior formulation
\cite{Budney}.
\end{proof}

\begin{definition}[Rooted JSJ tree]
\label{def:rooted-jsj-tree}
The \emph{rooted JSJ tree} $\Gamma_{\mathrm{JSJ}}(E)$ is the dual
graph of $E\sslash\mathcal T_{\mathrm{JSJ}}(E)$: vertices are JSJ
pieces, edges are JSJ tori, and the root is the vertex of the piece
meeting $\partial E$.  Each vertex is labelled by its Seifert or
hyperbolic type, and the incidence of its boundary tori is retained.
\end{definition}

\begin{lemma}
\label{lem:jsj-tree}
$\Gamma_{\mathrm{JSJ}}(E)$ is a finite tree, and the natural action
induces
\[
  \Psi_{\mathrm{JSJ}}:\Mod^{\pm}(E)\longrightarrow
  \Aut_{\mathrm{root,lab}}\Gamma_{\mathrm{JSJ}}(E).
\]
\end{lemma}

\begin{proof}
Since $H_2(E;\mathbb Z)=0$, every embedded torus in $E$ separates, so
the dual graph of a finite collection of separating tori is a tree.
Theorem~\ref{thm:classical-jsj-input}(3) gives the action, and the
root is preserved as the unique piece incident to $\partial E$.
\end{proof}

\begin{figure}[htb]
\centering
\begin{tikzpicture}[scale=1.0,every node/.style={font=\small}]
% nested picture of a satellite exterior
\begin{scope}
  \draw[thick,greyish] (0,0) ellipse (2.6 and 1.6);
  \node[greyish,above] at (0,1.66) {$\partial E$};
  \draw[surfA,very thick] (0,0) ellipse (1.9 and 1.15);
  \node[surfA] at (-1.52,0.98) {$T_1$};
  \draw[surfB,very thick] (-0.75,0) ellipse (0.75 and 0.55);
  \node[surfB] at (-0.75,0.72) {$T_2$};
  \draw[surfC,very thick] (0.9,0) ellipse (0.62 and 0.45);
  \node[surfC] at (0.9,0.63) {$T_3$};
  \node[greyish] at (-2.25,0) {\small $P_0$};
  \node[greyish] at (-0.75,0) {\small $P_2$};
  \node[greyish] at (0.9,0) {\small $P_3$};
  \node[greyish] at (0,-0.75) {\small $P_1$};
\end{scope}
% the tree
\begin{scope}[xshift=6.5cm,yshift=-0.6cm]
  \coordinate (r) at (0,2.0);
  \coordinate (p1) at (0,1.0);
  \coordinate (p2) at (-0.9,0.0);
  \coordinate (p3) at (0.9,0.0);
  \draw[thick] (r)--(p1); \draw[thick] (p1)--(p2); \draw[thick] (p1)--(p3);
  \fill (r) circle (2.4pt);  \node[right=3pt] at (r) {$P_0$ \ (root)};
  \fill[surfA] (p1) circle (2.4pt); \node[right=3pt,surfA] at (p1) {$P_1$};
  \fill[surfB] (p2) circle (2.4pt); \node[left=3pt,surfB] at (p2) {$P_2$};
  \fill[surfC] (p3) circle (2.4pt); \node[right=3pt,surfC] at (p3) {$P_3$};
  \node[surfA,left] at (-0.1,1.5) {$T_1$};
  \node[surfB,left] at (-0.55,0.5) {$T_2$};
  \node[surfC,right] at (0.55,0.5) {$T_3$};
\end{scope}
\end{tikzpicture}
\caption{A satellite exterior and its rooted JSJ tree.  The root is
the piece $P_0$ meeting $\partial E$, so it is preserved by every mapping
class.  The tree alone is not a complete invariant: Budney's
companionship graph adds vertex labels coming from companion links and
the peripheral data needed for gluing \cite{Budney}.}
\label{fig:jsjtree}
\end{figure}
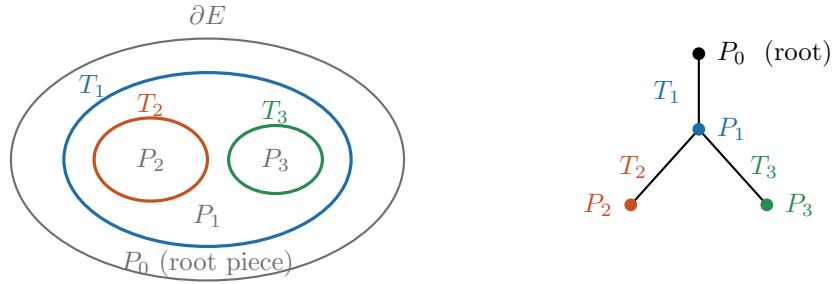

\begin{remark}[Companionship data and splicing]
\label{rem:companionship-tree}
The abstract rooted tree is not a complete knot invariant.  Budney
enriches it by vertex labels coming from companion links and by the
peripheral data needed for gluing; the resulting companionship graph
is a finite acyclic labelled graph giving a bijective encoding of
links in $S^3$, naturally rooted for knots.  In this language cabling,
connected sum, Whitehead doubling and general satellite constructions
are instances of splicing \cite{Budney}.  An intrinsic recognition
theorem must therefore recover not merely an unlabelled tree but the
piece labels and the peripheral attachment data.
\end{remark}

\begin{remark}[Group-theoretic and algorithmic shadows]
\label{rem:jsj-group-algorithm}
Van Kampen's theorem expresses $\pi_1E$ as the fundamental group of
the JSJ graph of groups with edge groups $\mathbb Z^2$, and
Scott--Swarup's canonical splittings over $\mathbb Z$ and
$\mathbb Z^2$ recover a closely related group-theoretic structure
\cite{ScottSwarup}; the peripheral meridian is still needed to return
from the group to the knot.  Algorithmically, normal-surface methods
compute prime and JSJ decompositions from a triangulation
\cite{JacoTollefson}.  This gives an external computability benchmark,
but does not solve the intrinsic problem of recognizing the same
structure from $\ES(E)$ alone.
\end{remark}

\begin{remark}[The Bonahon--Siebenmann layer]
\label{rem:bonahon-siebenmann-distinction}
The Bonahon--Siebenmann characteristic toric splitting is formulated
for the orbifold pair $(S^3,K)$
\cite{BonahonSiebenmannToric,BonahonSiebenmannKnots}.  A Conway sphere
becomes the Euclidean $2$-orbifold $S^2(2,2,2,2)$ and appears in the
exterior as an essential four-punctured sphere; it is therefore not an
edge of the ordinary torus JSJ tree.  This is a related
equivariant/orbifold layer which may be added to a hierarchy-level
object, but it should not be conflated with
$\mathcal T_{\mathrm{JSJ}}(E)$.
\end{remark}

\subsection{What the complex already recognizes}

One classical piece of the recognition problem can be read off from the
surface-type component of the peripheral data with no extra input.
Following Neumann--Swarup,
call an essential annulus or torus $C\subset E$ \emph{W-canonical} if
every essential annulus or torus in $E$ can be isotoped off $C$
\cite[Definition~2.1]{NeumannSwarup}.  Call a vertex $v$ of
$\ES^{\mathrm{per}}(E)$ \emph{annulus--torus universal} if its type
is annulus or torus and it is adjacent to every other annulus or torus
vertex.

\begin{proposition}[Recognition of the W-canonical subsystem]
\label{prop:W-canonical-recognition}
Let $E$ be a nontrivial knot exterior.
\begin{enumerate}[label=(\arabic*),leftmargin=*]
\item A vertex of $\ES^{\mathrm{per}}(E)$ is annulus--torus universal
if and only if it is represented by a W-canonical annulus or torus.
\item The set of W-canonical vertices is intrinsic in
$\ES^{\mathrm{per}}(E)$ and in every richer refinement; hence it
is preserved by every compatible automorphism.
\item The W-canonical vertices span a single finite simplex, and a
simultaneously disjoint choice of representatives is a W-system in the
sense of Neumann--Swarup, determining the W-decomposition up to
isotopy.
\end{enumerate}
\end{proposition}

\begin{proof}
(1) is the translation of the definition into the edge relation, an
edge meaning that the two classes admit disjoint representatives.  (2)
holds because the type labels and adjacency are preserved by
label-compatible automorphisms.  For (3), any two W-canonical vertices
are adjacent, so by Proposition~\ref{prop:ES-flag} every finite
collection spans a simplex; since $\ES(E)$ is finite-dimensional
(Lemma~\ref{lem:finite-dim}) there are only finitely many, and they
span a single simplex.  Distinct vertices have nonparallel
representatives, so the resulting collection is the maximal pairwise
nonparallel canonical system defining the W-decomposition, unique up
to isotopy by Neumann--Swarup's Lemmas~2.2 and~2.3
\cite{NeumannSwarup}.
\end{proof}

\begin{remark}[What remains]
\label{rem:W-versus-characteristic}
Proposition~\ref{prop:W-canonical-recognition} extracts a classical
part of the recognition problem from the peripheral refinement.  The
W-decomposition is related to, but must not be conflated with, the
ordinary torus JSJ decomposition or the full characteristic
submanifold.  The proposition does not identify the characteristic
pieces, their Seifert or $I$-bundle structures, the rooted incidence
data, or the peripheral attaching maps; nor does it recognize annulus
and torus types in the bare complex.  These are the remaining
contents of the intrinsic reconstruction problem.
\end{remark}

\subsection{The characteristic refinement}

\begin{definition}[Minimal position and characteristic support]
\label{def:support}
A properly embedded surface $F\subset E$ is in \emph{minimal
position} with respect to $\Fr\Char(E)$ if it is transverse to
$\Fr\Char(E)$ and the number of components of $F\cap\Fr\Char(E)$ is
minimal in the isotopy class of $F$.  The \emph{characteristic
support} $\supp_{\Char}(F)$ is the collection of pieces of
$E\sslash\Fr\Char(E)$ met by \emph{every} minimal-position
representative of $[F]$, considered up to isotopy of the
characteristic decomposition.
\end{definition}

\begin{remark}
\label{rem:support-well-defined}
Minimality already forces every component of $F\cap\Fr\Char(E)$ to be
essential in both $F$ and the frontier component containing it, by an
innermost-disk or outermost-arc argument using irreducibility of $E$
and incompressibility of $F$ and of the frontier surfaces.  The
universal quantifier in Definition~\ref{def:support} makes the support
manifestly isotopy-invariant without requiring a uniqueness statement
for minimal-position representatives.  If $F$ is isotopic to a
frontier component, its representatives can be pushed to either side,
so $\supp_{\Char}(F)=\varnothing$; this is consistent with, and
explains, the separate frontier datum below.
\end{remark}

\begin{definition}[Characteristic refinement]
\label{def:CES}
The \emph{characteristic refinement} $\CES(E)$ is
$\ES^{\mathrm{per}}(E)$ together with the following characteristic data:
\begin{enumerate}[label=(\arabic*),leftmargin=*]
\item labels for vertices represented by frontier annuli and tori of
$\Char(E)$, with the ordinary JSJ tori distinguished;
\item the characteristic support $\supp_{\Char}(F)$ of each vertex;
\item for each vertex, the set of boundary patterns induced on the
frontier components by its minimal-position representatives, up to
pattern isotopy and without twist coordinates;
\item the rooted labelled JSJ tree $\Gamma_{\mathrm{JSJ}}(E)$ together
with the incidence of characteristic pieces and frontier components.
\end{enumerate}
The boundary-pattern datum in (3) is quantified over minimal-position
representatives for the
reason given in Remark~\ref{rem:support-well-defined}; in practice one
expects a single pattern-isotopy class, but nothing below depends on
that.  The rooted-tree datum in (4) records only the ordinary
JSJ/characteristic
incidence, not the distinct Bonahon--Siebenmann layer of
Remark~\ref{rem:bonahon-siebenmann-distinction}.
\end{definition}

\begin{warning}
$\CES(E)$ is not claimed to be intrinsic to $\ES(E)$: it contains
characteristic data by fiat.  The intrinsic-recognition problem asks
whether they are forced by a more economical object.
\end{warning}

\subsection{Why realization is not a local problem}
\label{subsec:frontier-problem}

Suppose an automorphism $\psi$ of $\CES(E)$ has been shown to
preserve the characteristic frontier and to be realized, on each
piece, by a patterned homeomorphism.  Does it follow that $\psi$ is
induced by a homeomorphism of $E$?  Not automatically, and the
obstruction is easy to see (Figure~\ref{fig:frontier}).

Cut $E$ along $\Fr\Char(E)$ into pieces $P_1,\dots,P_n$.  Piece
homeomorphisms $h_i$ can always be adjusted in collars so as to glue,
\emph{provided} their boundary restrictions agree on each frontier
component; the discrepancy is measured by a pattern-isotopy class
$[\delta_Q]$ for each frontier component $Q$
(Definition~\ref{def:pg}(G4)).  Even when all discrepancies vanish and
a homeomorphism $h$ is obtained, $h$ is only known to agree with
$\psi$ on vertices whose support lies in a single piece: composing $h$
with a Dehn twist supported in a collar of a frontier annulus or torus
changes its action precisely on the vertices represented by
frontier-crossing surfaces, and fixes every single-piece vertex.  We
call this the \emph{frontier-crossing compatibility problem}; it is
condition (G5) of Definition~\ref{def:pg}.

\begin{keybox}
This is the essential new difficulty in dimension three, with no
analogue in the curve-complex setting.  There, a homeomorphism
realizing a combinatorial automorphism on a pants decomposition
already realizes it, because there is nothing left to twist by that
acts trivially on the visible data.  In a knot exterior there is:
Proposition~\ref{prop:kakimizu-spinning} exhibits exactly such a
twist, and shows that its effect is visible only on surfaces crossing
the frontier.
\end{keybox}

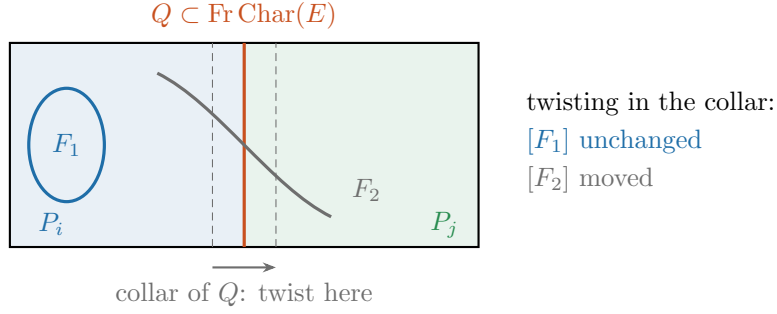
\begin{figure}[htb]
\centering
\begin{tikzpicture}[scale=1.0,every node/.style={font=\small}]
% two pieces glued along frontier annulus Q
\fill[surfA,opacity=0.10] (-3.1,-1.35) rectangle (0,1.35);
\fill[surfC,opacity=0.10] (0,-1.35) rectangle (3.1,1.35);
\draw[thick] (-3.1,-1.35) rectangle (3.1,1.35);
\draw[surfB,very thick] (0,-1.35) -- (0,1.35);
\node[surfB,above] at (0,1.4) {$Q\subset\Fr\Char(E)$};
\node[surfA] at (-2.55,-1.05) {$P_i$};
\node[surfC] at (2.65,-1.05) {$P_j$};
% a surface living in one piece
\draw[surfA,very thick] (-2.35,0) ellipse (0.5 and 0.75);
\node[surfA] at (-2.35,0) {$F_1$};
% a surface crossing
\draw[greyish,very thick] (-1.15,0.95) .. controls (-0.35,0.55) and (0.35,-0.55) .. (1.15,-0.95);
\node[greyish] at (1.62,-0.6) {$F_2$};
% collar and twist
\draw[greyish,densely dashed] (-0.42,-1.35) -- (-0.42,1.35);
\draw[greyish,densely dashed] (0.42,-1.35) -- (0.42,1.35);
\draw[-{Stealth[length=5pt]},thick,greyish] (-0.42,-1.62) -- (0.42,-1.62);
\node[greyish,below] at (0,-1.68) {collar of $Q$: twist here};
% conclusion
\node[align=left,anchor=west] at (3.6,0.55)
  {twisting in the collar:};
\node[align=left,anchor=west,surfA] at (3.6,0.05)
  {$[F_1]$ unchanged};
\node[align=left,anchor=west,greyish] at (3.6,-0.45)
  {$[F_2]$ moved};
\end{tikzpicture}
\caption{The frontier-crossing compatibility problem.  A Dehn twist
supported in a collar of a frontier component $Q$ fixes every vertex
supported in a single piece, but moves vertices represented by
surfaces crossing $Q$.  Hence realizing an automorphism piece by piece
does not suffice: a global condition, (G5) of
Definition~\ref{def:pg}, must be imposed or derived.  Compare
Figure~\ref{fig:winding}, which is this picture with $Q$ the
decomposing annulus of a connected sum.}
\label{fig:frontier}
\end{figure}

\begin{warning}
\label{warn:no-shortcut}
One cannot simply assert that a mapping class acting trivially on
essential-surface data is isotopic to the identity modulo frontier
twists.  Such a statement is itself a three-dimensional Ivanov-type
rigidity assertion, and it is false verbatim: for a torus-knot
exterior the strong inversion acts trivially on every object
considered here (Corollary~\ref{cor:torus-final}), and for the
figure-eight knot an entire Klein four group does
(Corollary~\ref{cor:fig8-final}), while no frontier exists at all.  In
this article such statements are treated as problems or hypotheses, not
as inputs hidden inside a proof.
\end{warning}

\begin{problem}[Kernel problem]
\label{prob:kernel}
Determine $\mathcal K_X(E)$ for natural choices of $X$,
especially $X=\ES(E)$, $\ES^{\mathrm{per}}(E)$, $\CES(E)$.  In
particular decide when it is trivial, when it is generated by explicit
twists along annuli or tori invisible to $X$, and when it also
contains finite symmetries acting trivially on all essential-surface
data.  For hyperbolic knots the qualitative answer is unconditional by
Corollary~\ref{cor:hyperbolic-kernel} --- the kernel is a finite group
of invisible isometries, and the content is to say which --- so the
twist part of the problem concerns decomposed nonhyperbolic exteriors,
including satellite and composite knots.
\end{problem}

For two-bridge knots the kernel is determined in
Theorem~\ref{thm:two-bridge-kernel}.

\section{Toward rigidity: three tasks}
\label{sec:reduction}

The examples suggest a three-dimensional analogue of Ivanov rigidity,
but they also show why it cannot be a single statement.  For a chosen
essential-surface object $X$ --- usually one of $\ES(E)$,
$\ES^{\mathrm{per}}(E)$, or $\CES(E)$ --- three tasks must be kept
separate.

\begin{keybox}
\begin{enumerate}[label=(A\arabic*),leftmargin=*]
\item \emph{Recognition}: recover the characteristic frontier, the
pieces, and the relevant peripheral data intrinsically from $X$.
\item \emph{Realization}: decide whether a combinatorial automorphism
is induced by compatible homeomorphisms of the pieces and hence by a
global homeomorphism of $E$.
\item \emph{Kernel}: determine the mapping classes acting trivially on
$X$.
\end{enumerate}
\end{keybox}

The second task is the direct analogue of the hard
``automorphisms are geometric'' step on surfaces, with an additional
three-dimensional difficulty: piecewise realizations must agree across
the characteristic frontier.  Figure~\ref{fig:frontier} shows the
obstruction, and Appendix~\ref{app:realization} records a precise set
of compatibility data for readers who need it.

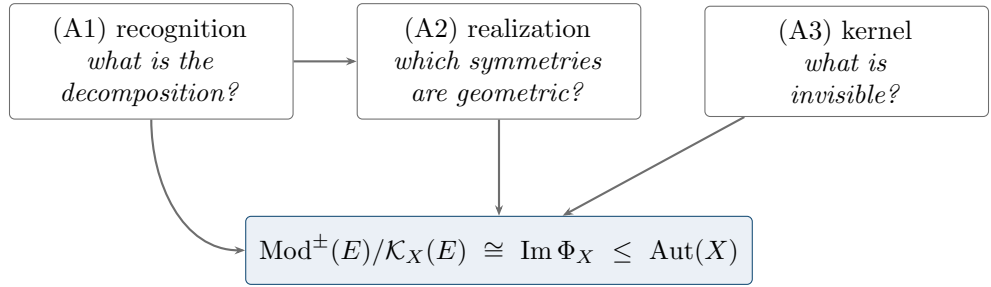
\begin{figure}[htb]
\centering
\begin{tikzpicture}[
  every node/.style={font=\small},
  box/.style={draw=greyish,rounded corners=2pt,inner sep=5pt,align=center,
              minimum height=9mm},
  hi/.style={box,draw=surfA!70!black,fill=lightfill},
  ar/.style={-{Stealth[length=4pt]},thick,greyish}]
\node[box,text width=34mm] (a1) at (0,0)
  {(A1) recognition\\ \emph{what is the}\\ \emph{decomposition?}};
\node[box,text width=34mm] (a2) at (4.6,0)
  {(A2) realization\\ \emph{which symmetries}\\ \emph{are geometric?}};
\node[box,text width=34mm] (a3) at (9.2,0)
  {(A3) kernel\\ \emph{what is}\\ \emph{invisible?}};
\node[hi] (c) at (4.6,-2.5)
  {$\Mod^{\pm}(E)/\mathcal K_X(E)\ \cong\ \Image\Phi_X\ \le\ \Aut(X)$};
\draw[ar] (a1) -- (a2);
\draw[ar] (a2) -- (c);
\draw[ar] (a3) -- (c);
\draw[ar] (a1) to[out=-90,in=180] (c);
\end{tikzpicture}
\caption{The three tasks.  Recognition supplies the geometric
structure to be realized, realization determines the image, and kernel
determination tells what the action loses.  If realization is
surjective, the last inclusion becomes equality by the first
isomorphism theorem.}
\label{fig:packages}
\end{figure}

\begin{remark}[What is formal and what is not]
\label{rem:formal}
Once realization is surjective and the kernel is known, the
isomorphism $\Mod^{\pm}(E)/\mathcal K_X(E)\cong\Aut(X)$ is merely the
first isomorphism theorem.  The mathematical content lies in proving
recognition, global realization, and kernel identification.  In
particular, a mapping class which fixes every essential-surface vertex
cannot simply be declared trivial by Waldhausen--Johannson rigidity;
one must first control its induced action on $\pi_1E$ and the
characteristic structure.  See \cite{Waldhausen,Johannson,
KalliongisMcCullough} for the relevant background.
\end{remark}

At the opposite extreme one may record an entire marked Haken
hierarchy, including terminal patterned balls and strictly compatible
gluing maps.  Appendix~\ref{app:atlas} defines this deliberately
over-marked \emph{marked hierarchy atlas}.  By construction an atlas
isomorphism can be reassembled into a homeomorphism of the exterior.
The point of the benchmark is only to calibrate sufficiency: the
research problem is to determine how much of that information is
already forced by the much smaller objects $\ES(E)$,
$\ES^{\mathrm{per}}(E)$, or $\CES(E)$.

\section{Problems}
\label{sec:problems}

The following list is not meant to be exhaustive; it records questions
that arise naturally from the viewpoint of this article.  We divide it
into three kinds.  \emph{Computational projects}
(Section~\ref{subsec:prob-computational}) can be attacked with tools
that already exist, and several of them are suitable as a first
research project; Appendix~\ref{app:recipe} gives the recipe needed
for the two-bridge cases.  \emph{Open questions}
(Section~\ref{subsec:prob-open}) are specific questions whose answer we
do not know in general; several of them are settled for two-bridge
knots in Section~\ref{sec:two-bridge}, as recorded in the status
lines.  \emph{Programmatic problems}
(Section~\ref{subsec:prob-programme}) indicate directions rather than
precise statements; they are not conjectures.  Each item ends with a
short statement of its status, to the best of our knowledge.
Problems~\ref{prob:two-bridge-rigidity} and~\ref{prob:kernel}, stated
in the body of the article, belong to the second and third kinds
respectively.

\subsection{Computational projects}
\label{subsec:prob-computational}

\begin{problem}[Census and complex for small knots]
\label{prob:census}
For each knot with at most nine crossings, decide whether it is
slope-separated, and when it is, write down the peripheral
slope--type list of Remark~\ref{rem:invariant}.  By
Theorem~\ref{thm:visibility} this determines the image and the kernel
of the actions on the bare and peripheral levels at once, with no
mapping class group computation.
\end{problem}

\noindent\textit{Status.}  Solved for two-bridge knots and torus
knots.  By Theorem~\ref{thm:two-bridge-slope-separated} and
Corollary~\ref{cor:small-table} the slope-separated two-bridge knots
in this range are $3_1,4_1,5_1,5_2,6_1,7_1,7_2,8_1,9_1,9_2$, with the
slope--type lists of Table~\ref{tab:two-bridge-slope-separated}, and no
other two-bridge knot is slope-separated; the torus knot
$8_{19}=T(3,4)$ is slope-separated by
Lemma~\ref{lem:torus-classification}.  The remaining knots, the first
of which is $8_5$, are open and computational.  For Montesinos
knots it is the Hatcher--Oertel algorithm \cite{HatcherOertel},
together with Oertel's results on closed incompressible surfaces
\cite{OertelClosed}.  These cover most knots in this range; the
remaining ones require normal-surface methods \cite{JacoTollefson}.
Amphichirality, which is the only symmetry input needed, is recorded
in the symmetry tables \cite{KodamaSakuma,HenryWeeks}.
We are not aware of a published table of this kind.

\begin{problem}[Dimension of small satellite examples]
\label{prob:dimension}
Proposition~\ref{prop:dimension} gives $\dim\ES(E)\ge n$ for an
$n$-fold iterated cable, and Lemma~\ref{lem:finite-dim} gives the
Kneser--Haken upper bound $h(E)-1$.  Compute $\dim\ES(E)$ exactly for
iterated torus knots and for connected sums of two torus knots.
\end{problem}

\noindent\textit{Status.}  Open in this form, but expected to
be computable: the JSJ decompositions of these exteriors are explicit
\cite{Budney}, and their pieces are Seifert fibered, where essential
surfaces are classified \cite{HatcherNotes,Waldhausen1967}.  The
difficulty is to control essential surfaces that cross the JSJ tori.

\subsection{Open questions}
\label{subsec:prob-open}

\begin{problem}[Chiral knots with symmetric peripheral data]
\label{prob:chiral-symmetric}
This is Question~\ref{q:chiral-symmetric}: does there exist a
slope-separated chiral knot with $S(K)=-S(K)\neq\{0\}$ and
$\tau(-s)=\tau(s)$ for every $s$?  By
Corollary~\ref{cor:surjectivity} this is exactly the question of
whether type and slope together can retain a nongeometric
automorphism.
\end{problem}

\noindent\textit{Status.}  Open in general.  Torus knots do
not qualify, since $S(T(p,q))=\{0,pq\}$ is not symmetric, and no
two-bridge knot qualifies: by
Corollary~\ref{cor:two-bridge-question} the only slope-separated
two-bridge knot with symmetric slope set is the amphichiral knot
$4_1$.  The natural place to search is now among Montesinos knots
without closed essential surfaces
\cite{OertelClosed,HatcherOertel}.

\begin{problem}[Which layers are singletons?]
\label{prob:singleton-layers}
Give a criterion, for a family of knots, under which each boundary
slope carries exactly one essential surface.  A weaker version: give
a criterion for the slope-$0$ layer $\ES_0(E)$ to consist of a single
vertex.
\end{problem}

\noindent\textit{Status.}  Solved for two-bridge knots: every
slope layer is a singleton exactly for the torus knots $T(2,q)$ and
the twist knots (Theorem~\ref{thm:two-bridge-slope-separated}), and by
Proposition~\ref{prop:two-bridge-finite} and (HT4) the layer
$\ES_0(E)$ is a single vertex exactly when the all-even expansion is
thin and no other admissible expansion $\beta$ has
$\sigma(\beta)=\sigma(\beta_{\mathrm{ev}})$.  Partially known
otherwise.  For vertices with one boundary component the weak
version asks whether the Kakimizu complex $\operatorname{IS}(K)$ is a
single point, which is classical for fibered knots and is accessible
for two-bridge knots through Hatcher--Thurston's classification
\cite{Kakimizu,HatcherThurston}; Eisner and Wilson give knots for
which the slope-$0$ layer is infinite \cite{Eisner,WilsonInfinite}.  Vertices of slope $0$ with several
boundary components are not controlled by Kakimizu theory, and we do
not know a general criterion for them.

\begin{problem}[How much peripheral information is needed?]
\label{prob:minimal-decoration}
The peripheral refinement $\ES^{\mathrm{per}}(E)$ records surface type
and boundary slope.  For natural families of knots, decide whether one
of the two data can be discarded without creating nongeometric
automorphisms.
\end{problem}

\noindent\textit{Status.}  Solved for slope-separated knots,
open otherwise.  For every slope-separated two-bridge knot either datum
alone suffices (Figure~\ref{fig:decorations-fig8} and
Corollary~\ref{cor:two-bridge-image-kernel}).  For a general
slope-separated knot, Proposition~\ref{prop:aut-slope-separated} and
Theorem~\ref{thm:visibility} show that the slope alone suffices unless
$K$ is chiral with $S(K)=-S(K)\ne\{0\}$, which is the situation of
Problem~\ref{prob:chiral-symmetric}.  We know no reason to expect a
uniform answer in larger families.

\begin{problem}[From W-canonical surfaces to the characteristic frontier]
\label{prob:W-to-frontier}
Proposition~\ref{prop:W-canonical-recognition} recognizes the
W-canonical annulus and torus vertices in $\ES^{\mathrm{per}}(E)$ by
universal adjacency.  Are annulus and torus types, and hence this
subsystem, intrinsically recognizable in the bare complex $\ES(E)$?
Starting from the recognized W-system, what further data are needed
to recover the characteristic frontier, the characteristic pieces and
their attaching data?
\end{problem}

\noindent\textit{Status.}  The first question has a negative
answer in general: for a torus knot the transposition of the two
vertices of $\ES(E)$ exchanges the cabling annulus with the fibre
surface (Corollary~\ref{cor:torus-final}), so the annulus type is not
intrinsic in the bare complex; similarly the bare automorphisms of the
twist-knot complexes preserve neither genus nor the number of boundary
components.  We do not know whether the torus type is intrinsic in the
bare complex.  The remaining questions are open.

\begin{problem}[JSJ tree and companionship recognition]
\label{prob:jsj-recognition}
Can one recover intrinsically the rooted labelled tree
$\Gamma_{\mathrm{JSJ}}(E)$, including the Seifert/hyperbolic labels
and the peripheral attachment data needed for Budney's companionship
graph \cite{Budney}?  Which part of the companionship data is forced
by $\ES(E)$, $\ES^{\mathrm{per}}(E)$ or $\CES(E)$?
\end{problem}

\noindent\textit{Status.}  Open.  By construction $\CES(E)$
carries the characteristic data, so the question is nontrivial only
for $\ES(E)$ and $\ES^{\mathrm{per}}(E)$.

\begin{problem}[Remaining kernel for a two-factor connected sum]
\label{prob:connected-sum-kernel}
For $K=K_1\mathbin\#K_2$ as in Theorem~\ref{thm:composite-annulus},
compute the full kernels of the actions on $\ES^{\mathrm{per}}(E)$,
$\operatorname{MS}(K)$ and $\operatorname{IS}(K)$.  In particular,
determine the contribution of the factor mapping class groups, of
factor exchange when $K_1\cong K_2$, and of finite invisible
symmetries; determine what happens when one factor is fibered; and
decide whether the winding coordinate can be characterized
intrinsically in $\ES^{\mathrm{per}}(E)$.
\end{problem}

\noindent\textit{Status.}  Partially known.  When both factors
are nonfibered, Proposition~\ref{prop:kakimizu-spinning} shows that
the annular twist subgroup acts freely on $\operatorname{MS}(K)$; the
remaining parts are open.

\begin{problem}[Images and kernels on the Kakimizu layers]
\label{prob:kakimizu-layers}
Determine and compare the exact images and kernels of
\[
  \Mod^{\pm}(E(K))\to\Aut(\operatorname{MS}(K)),
  \qquad
  \Mod^{\pm}(E(K))\to\Aut(\operatorname{IS}(K)),
\]
decide which twisting directions are faithfully detected on each
layer, and determine which further slopes or closed essential surfaces
are needed to reconstruct the characteristic structure.
\end{problem}

\noindent\textit{Status.}  Partially known.  The complexes are
connected \cite{ScharlemannThompson,Kakimizu}; $\operatorname{MS}(K)$ is
simply connected \cite{SchultensSimplyConnected} and contractible, and
every finite group of symmetries fixes a simplex of it
\cite{PrzytyckiSchultens}; $\operatorname{IS}(K)$ is contractible for
every nontrivial knot, answering a question of Przytycki and Schultens
\cite{PanYouZhouChen}; finite symmetries have invariant Seifert
surfaces in many cases \cite{Edmonds,HHSInvariant}; there is an abelian JSJ-twist action giving
coarse $\mathbb Z^n$ geometry for $\operatorname{MS}(K)$, with
the exponent determined by Banks \cite{JohnsonPelayoWilson,BanksQuasi}; and connected sums have explicit
coordinates \cite{BanksConnectedSum}.  Exact images and kernels are
not known in general.  Remark~\ref{rem:fig8-kakimizu} and
Proposition~\ref{prop:kakimizu-spinning} show that
$\operatorname{MS}(K)$ and $\ES(E)$ can see different parts of the
symmetry.

\begin{problem}[The boundary-slope shadow]
\label{prob:slope-shadow}
The boundary-slope set records only the image of $[F]\mapsto\Sl(F)$
and forgets disjointness, surface type and characteristic support.
How much of $\ES(E)$ is determined by it?
\end{problem}

\noindent\textit{Status.}  Open in general.  By
Remark~\ref{rem:invariant} the slope set together with surface types
determines $\ES^{\mathrm{per}}(E)$ for slope-separated knots, so the
first question is how far beyond that class this persists.  Within
the class of two-bridge knots, a slope-separated knot is determined up
to mirror image by its slope set together with the fact that the
complex is discrete
(Proposition~\ref{prop:two-bridge-reconstruction}).

\subsection{Programmatic problems}
\label{subsec:prob-programme}

The problems in this subsection are deliberately broad.  Each is open
apart from the first cases indicated, and each would need to be
specialized before it could be formulated as a conjecture.

\begin{problem}[Local geometricity]
\label{prob:local-geometricity}
Suppose an automorphism of an economical essential-surface complex
preserves the characteristic frontier.  Under what hypotheses is its
restriction to each characteristic or complementary piece induced by
a patterned homeomorphism of that piece?
\end{problem}

\begin{problem}[Globality across the frontier]
\label{prob:globality}
Suppose an automorphism admits a realization datum satisfying
(G1)--(G4) of Definition~\ref{def:pg}.  Find natural conditions
ensuring that a glued realization can be chosen satisfying the
globality condition (G5), that is, inducing the automorphism on the
vertices represented by frontier-crossing surfaces.
Figure~\ref{fig:frontier} shows the obstruction, and
Figure~\ref{fig:winding} the smallest example.
\end{problem}

\begin{problem}[Action on the fundamental group]
\label{prob:pi1}
If a mapping class fixes all vertices of $X$, does it induce the
identity in $\Out(\pi_1E)$, possibly modulo a controlled subgroup?
This is the step needed before Waldhausen--Johannson rigidity can be
applied, and by Remark~\ref{rem:formal} it cannot be assumed.
\end{problem}

\begin{problem}[Reconstruction]
\label{prob:reconstruction}
Determine when an isomorphism
\[
  \ES^{\mathrm{per}}(E(K))\cong\ES^{\mathrm{per}}(E(K'))
\]
is induced by a homeomorphism of exteriors.  The marked hierarchy atlas of
Appendix~\ref{app:atlas} is a deliberately over-marked sufficient
object; the question is how far down the three levels of
Figure~\ref{fig:ladder} the conclusion survives.  For slope-separated
knots this reduces to Problem~\ref{prob:slope-shadow}.
\end{problem}

\noindent\textit{Status.}  The first cases are
Proposition~\ref{prop:torus-reconstruction}, which reconstructs torus
knots among all knots, and
Proposition~\ref{prop:two-bridge-reconstruction}, which reconstructs
the slope-separated two-bridge knots among two-bridge knots.  We do not
know whether a twist knot is determined by its peripheral complex among
all knots.

\begin{problem}[An economical three-dimensional Ivanov problem]
\label{prob:3d-ivanov}
For a nontrivial knot exterior, determine the weakest natural
essential-surface object $X$ for which one can compute the image and
kernel, decide whether every automorphism is geometric, and
reconstruct the exterior from the isomorphism type of $X$.  How much
of the peripheral and characteristic information in
$\ES^{\mathrm{per}}(E)$ and $\CES(E)$ is genuinely necessary?
\end{problem}

\section{Further reading}
\label{sec:further}

We list, for a reader who wants to go further, the sources on which
each part of the article rests.

\smallskip
\emph{Curve complexes and their rigidity.}  Ivanov's theorem and its
extensions by Korkmaz and Luo \cite{Ivanov,Korkmaz,Luo}; the survey of
Aramayona--Souto \cite{AramayonaSoutoSurvey} and, for the
metaconjectural viewpoint, McCarthy--Papadopoulos and
Brendle--Margalit \cite{McCarthyPapadopoulos,BrendleMargalit}.  For
three-dimensional analogues see Korkmaz--Schleimer,
Charitos--Papadoperakis--Tsapogas and Aramayona--Souto
\cite{KorkmazSchleimer,CharitosPapadoperakisTsapogas,AramayonaSoutoSplittings}.

\smallskip
\emph{Essential surfaces in $3$-manifolds.}  Hempel's and Hatcher's
expositions \cite{Hempel,HatcherNotes}, Waldhausen's foundational
papers \cite{Waldhausen1967,Waldhausen}, and Jaco--Shalen and
Johannson for the characteristic submanifold \cite{JS,Johannson}.
Hatcher's finiteness theorem for boundary slopes
\cite{HatcherBoundaryCurves} underlies all the slope bookkeeping.
Budney's article \cite{Budney} is the most convenient reference for
the JSJ decomposition of knot exteriors.

\smallskip
\emph{Complexes of surfaces.}  Schultens's surface complex
\cite{SchultensSurfaceComplex}, Kakimizu's complexes
\cite{Kakimizu,KakimizuClassification} and their connectivity
\cite{ScharlemannThompson}, the distance of Sakuma--Shackleton
\cite{SakumaShackleton}, simple connectivity and contractibility
\cite{SchultensSimplyConnected,PrzytyckiSchultens,PanYouZhouChen},
invariant Seifert surfaces \cite{Edmonds,HHSInvariant}, coarse geometry
\cite{JohnsonPelayoWilson,BanksQuasi}, Banks's connected-sum and
split-link models \cite{BanksConnectedSum,BanksSplit}, Kakimizu
complexes of Seifert fibered spaces \cite{SchultensKakimizuSeifert},
and the related complexes of
\cite{ZhangGuo,CharitosPapadoperakisTsapogas,CharitosIncompressible,KorkmazSchleimer,Schleimer}.

\smallskip
\emph{Censuses for model knots.}  Tsau for torus knots
\cite{TsauTorusSurfaces}; Hatcher--Thurston for two-bridge knots
\cite{HatcherThurston}, with the slope conventions made explicit in
\cite{CurtisEtAl} and the classification recalled in \cite{Owad};
Schubert's normal form \cite{SchubertTwoBridge}; branched surfaces
\cite{FloydOertel}; Sakuma's survey, Heusener--Mu\~noz--Porti, and the
tables of Kodama--Sakuma and Henry--Weeks for symmetry groups
\cite{SakumaSurvey,HeusenerMunozPorti,KodamaSakuma,HenryWeeks}.

\section*{Acknowledgements}

\paragraph{Use of generative AI}
The author used ChatGPT (OpenAI) and Claude (Anthropic) as interactive
aids for exposition, alternative formulations, figure preparation,
and preliminary consistency checks.  All mathematical statements,
proofs, computations, and references were independently checked by
the author, who takes full responsibility for the originality and
correctness of the manuscript.

\appendix

\section{Realization data for piecewise geometric automorphisms}
\label{app:realization}

This appendix contains the technical form of the realization
hypotheses used in Section~\ref{sec:reduction}.  The definitions are
deliberately cautious: they do not claim that all automorphisms of
$\CES(E)$ are realized by homeomorphisms, and they do not identify
the kernel of the action.  Figure~\ref{fig:frontier} is the picture to
keep in mind.

Let $E\,|\,\Fr$ denote $E$ cut along $\Fr\Char(E)$, with quotient map
$\pi:E\,|\,\Fr\to E$ and pieces $P_1,\dots,P_n$.  Each frontier
component $Q$ has exactly two preimage copies $Q^{+}\subset\partial
P_i$ and $Q^{-}\subset\partial P_j$ (possibly $i=j$), and we write
\[
  \tau_Q=(\pi|_{Q^{-}})^{-1}\circ(\pi|_{Q^{+}}):Q^{+}\to Q^{-}
\]
for the gluing identification.

\begin{definition}[Piecewise geometric automorphism]
\label{def:pg}
An automorphism $\psi\in\Aut(\CES(E))$ is \emph{piecewise
geometric} if it admits a \emph{realization datum} consisting of the
following.
\begin{enumerate}[label=(G\arabic*),leftmargin=*]
\item $\psi$ preserves the set of vertices represented by
characteristic frontier annuli and tori, and the accompanying data
include a bijection $\sigma$ of
the pieces, a bijection $\rho$ of the frontier components, and a
bijection
\[
  \widehat\rho:\{Q^{+},Q^{-}\}\to\{\rho(Q)^{+},\rho(Q)^{-}\}
\]
for every frontier component $Q$, all compatible with the support
labels.  In the dual graph of the decomposition, $\sigma$ is a vertex
bijection, $\rho$ an edge bijection, and $\widehat\rho$ the
accompanying half-edge bijection.  The last is genuinely additional
data: it is explicitly allowed to exchange the two sides of a frontier
component, as happens when $\psi$ swaps two homeomorphic pieces glued
along $Q$, for instance the two factor pieces of the exterior of
$K_1\mathbin\#K_1$.
\item The triple $(\sigma,\rho,\widehat\rho)$ is an isomorphism of the
dual graph respecting half-edge incidence: if $Q^{\varepsilon}\subset
\partial P_i$ then $\widehat\rho(Q^{\varepsilon})\subset
\partial P_{\sigma(i)}$.
\item For every piece $P_i$ a patterned homeomorphism
$h_i:P_i\to P_{\sigma(i)}$ realizing the correspondence induced by
$\psi$ between surface data supported in $P_i$ and in $P_{\sigma(i)}$,
and satisfying $h_i(Q^{\varepsilon})=\widehat\rho(Q^{\varepsilon})$
for every frontier copy $Q^{\varepsilon}\subset\partial P_i$.
\item For each frontier component $Q$, with $Q^{+}\subset\partial P_i$
and $Q^{-}\subset\partial P_j$, write $R=\rho(Q)$ and let
\[
  \tau_Q^{\widehat\rho}:\widehat\rho(Q^{+})\to\widehat\rho(Q^{-})
\]
be the gluing identification of $R$ read in the direction prescribed
by $\widehat\rho$: explicitly $\tau_Q^{\widehat\rho}=\tau_R$ if
$\widehat\rho$ preserves the two sides, and
$\tau_Q^{\widehat\rho}=\tau_R^{-1}$ if it exchanges them.  The
\emph{frontier discrepancy}
\[
  \delta_Q=\bigl(h_j|_{Q^{-}}\bigr)\circ\tau_Q\circ
           \bigl(h_i|_{Q^{+}}\bigr)^{-1}\circ
           \bigl(\tau_Q^{\widehat\rho}\bigr)^{-1}:
  \widehat\rho(Q^{-})\to\widehat\rho(Q^{-})
\]
is a pattern-preserving self-homeomorphism, whose pattern-isotopy
class $[\delta_Q]$ is the \emph{frontier discrepancy class} of the
datum at $Q$.  The datum is \emph{compatible} if the piece
homeomorphisms can be chosen, among all systems satisfying (G3), so
that every frontier discrepancy class is trivial.
\item \emph{Globality}: for some compatible choice as in (G4), the
glued homeomorphism $h:E\to E$ produced by Lemma~\ref{lem:gluing}
induces $\psi$ on all of $\CES(E)$, that is
$\Phi_{\CES}([h])=\psi$.
\end{enumerate}
Realization data compose and invert, so the piecewise geometric
automorphisms form a subgroup
$\Aut^{\mathrm{pg}}(\CES(E))$.
\end{definition}

\begin{remark}[Why (G5) is not redundant]
\label{rem:G5}
Conditions (G1)--(G4) produce, via Lemma~\ref{lem:gluing}, a
homeomorphism $h$ agreeing with $\psi$ on every vertex whose support
lies in a single piece.  They do \emph{not} force agreement on
vertices represented by surfaces crossing the frontier: composing $h$
with twists supported in collars of frontier annuli and tori changes
its action precisely on such vertices while fixing all single-piece
ones.  See Figure~\ref{fig:frontier} for the picture and
Proposition~\ref{prop:kakimizu-spinning} for a computed instance.
\end{remark}

\begin{remark}
Definition~\ref{def:pg} is not intrinsic in the sense of Ivanov
theory, since it refers to the characteristic pieces of $E$.  This is
intentional: it isolates geometric realization from intrinsic
recognition.  The frontier discrepancy classes are not part of the
combinatorial automorphism; they arise only after choosing the piece
homeomorphisms in (G3).
\end{remark}

\begin{lemma}[Gluing patterned piece maps]
\label{lem:gluing}
Let $E$ be cut along $\Fr\Char(E)$ into pieces $P_1,\dots,P_n$ with
gluing identifications $\tau_Q$.  Suppose given bijections
$\sigma,\rho,\widehat\rho$ as in (G1)--(G2) and patterned
homeomorphisms $h_i:P_i\to P_{\sigma(i)}$ sending each frontier copy
$Q^{\varepsilon}$ to $\widehat\rho(Q^{\varepsilon})$, such that every
frontier discrepancy class $[\delta_Q]$ is trivial.  Then the $h_i$,
after modification by pattern isotopies supported in pairwise disjoint
collars of the frontier copies, glue to a homeomorphism $h:E\to E$,
so that $[h]\in\Mod^{\pm}(E)$.
\end{lemma}

\begin{proof}
Fix a frontier component $Q$ with copies $Q^{+}\subset\partial P_i$
and $Q^{-}\subset\partial P_j$.  The maps $h_i,h_j$ descend along
$\pi$ near $Q$ if and only if $h_j\circ\tau_Q=
\tau_Q^{\widehat\rho}\circ h_i$ on $Q^{+}$, which is precisely
$\delta_Q=\mathrm{id}$, and this makes sense whether or not
$\widehat\rho$ exchanges the two sides.  By hypothesis $\delta_Q$ is
pattern-isotopic to the identity, so the two boundary restrictions
differ by a pattern-preserving homeomorphism pattern-isotopic to the
identity.  Such a difference is absorbed in a collar: choose a collar
$Q^{-}\times[0,1]\subset P_j$ and replace $h_j$ there by the map which
at level $t$ interpolates along a chosen pattern isotopy between the
corrected boundary map at $t=0$ and the original $h_j$ at $t=1$.
Collars of distinct frontier copies may be chosen pairwise disjoint,
so corrections at different components do not interfere.  After these
modifications the disjoint union of the piece maps commutes with all
gluing identifications, hence descends to a homeomorphism of
$E=(E\,|\,\Fr)/\{\tau_Q\}$.
\end{proof}

\begin{proposition}[Realization of piecewise geometric automorphisms]
\label{prop:realization}
Every piecewise geometric automorphism
$\psi\in\Aut^{\mathrm{pg}}(\CES(E))$ is realized by a mapping
class in $\Mod^{\pm}(E)$; equivalently
$\Aut^{\mathrm{pg}}(\CES(E))\subseteq\Image\Phi_{\CES}$.
\end{proposition}

\begin{proof}
Choose a realization datum with a compatible system of piece
homeomorphisms as in (G4).  By Lemma~\ref{lem:gluing} these glue to a
homeomorphism $h$ of $E$, so $[h]\in\Mod^{\pm}(E)$; condition (G5)
states precisely that $\Phi_{\CES}([h])=\psi$.
\end{proof}

\begin{remark}
With (G5) included in the definition, Proposition~\ref{prop:realization}
is close to a tautology.  Its role is to fix precisely what
``geometric realization'' must mean in
the roadmap of Section~\ref{sec:reduction}, and to locate the honest content in
(G1)--(G5) rather than in the gluing step.  Without (G5) the argument
yields only a mapping class agreeing with $\psi$ on single-piece
vertices, by Remark~\ref{rem:G5}.
\end{remark}

\begin{definition}[Piecewise geometric automorphisms of economical objects]
Let $X$ be an economical essential-surface object equipped with a
forgetful map $\CES(E)\to X$.  An automorphism
$\bar\psi\in\Aut(X)$ is \emph{piecewise geometric relative
to a reconstructed characteristic structure} if the characteristic
frontier and pieces are intrinsically reconstructed from $X$ and
$\bar\psi$ admits a lift $\widetilde\psi\in\Aut(\CES(E))$
along the forgetful map which is piecewise geometric in the sense of
Definition~\ref{def:pg}.  If the lift is not unique, the phrase means
that at least one such lift, together with a realization datum
satisfying (G1)--(G5), has been chosen.  Asserting that every
automorphism of an economical object is piecewise geometric is
therefore the conjunction of four statements: existence of a lift;
realization on each piece; compatibility of frontier discrepancy
classes; and globality.  The second is the three-dimensional analogue
of the hard part of the Ivanov theorem.  All four are hypotheses in
this article, not conclusions.
\end{definition}

\section{The marked hierarchy atlas}
\label{app:atlas}

This appendix records the over-marked reconstruction object of
Section~\ref{sec:reduction}.  It is a benchmark, not a candidate for
the final Ivanov complex.

\begin{definition}[Marked hierarchy atlas]
\label{def:atlas}
The \emph{marked hierarchy atlas} $\HA(E,\mu)$ consists of:
\begin{enumerate}[label=(M\arabic*),leftmargin=*]
\item the set of all complete Haken hierarchy records of $E$, up to
isotopy;
\item for each record, the patterned $3$-balls obtained at the end of
that hierarchy;
\item for each record, the incidence relation among the patterned
faces;
\item for each record, chosen PL gluing maps along all faces together
with their pattern-isotopy classes relative to a fixed reference
system: for \emph{every} face the atlas records both an actual gluing
homeomorphism, as a map and not merely as an isotopy class, and the
class, in the pattern-preserving mapping class group of that face, of
its difference from a fixed reference gluing (for annular and toral
faces these classes are the familiar twist coordinates);
\item the peripheral meridian label $\mu$.
\end{enumerate}
An \emph{isomorphism of marked hierarchy atlases} is a bijection
between the sets of hierarchy records together with, for each matched
pair, a system of PL homeomorphisms of the associated terminal
patterned balls which matches the boundary patterns and incidence
data, \emph{commutes strictly with the gluing maps of matched faces},
and carries the reference systems, the gluing classes of (M4), and the
meridian label of one record to those of the other.

Two deliberate over-markings are built in.  First, (M4) records a
gluing class on \emph{every} face, not only on annular and toral ones;
for faces with trivial pattern-preserving mapping class group the
class carries no information, but no triviality assumption is made.
Second, strict commutation is required rather than commutation up to
pattern isotopy.  This is not an innocent convenience: distinct faces
of a terminal ball meet along edges and vertices, so they do not admit
pairwise disjoint collars, and improving a system of ball
homeomorphisms which intertwines the gluing maps only up to pattern
isotopy would require an induction over the strata of the face
structure --- matching first on vertices, then on edges, then on
faces, with isotopy extension at each step.  The atlas records the
strictly compatible system instead of deriving it.  We make no attempt
to render the reference data canonical; that dependence is part of the
marking.
\end{definition}

\begin{lemma}[Patterned ball realization]
\label{lem:ball}
Let $B,B'$ be PL $3$-balls with finite boundary patterns.  A
combinatorial isomorphism of the boundary cell structures is realized
by a PL homeomorphism $B\to B'$.
\end{lemma}

\begin{proof}
A finite boundary pattern gives a finite cell structure on
$S^2=\partial B$.  A combinatorial isomorphism of cell structures is
realized by a PL homeomorphism of $S^2$, built cell by cell over the
skeleta, and the PL Alexander trick extends it over the ball; see
Rourke--Sanderson \cite{RourkeSanderson}.
\end{proof}

\begin{remark}[Why the atlas reconstructs]
Choose a hierarchy record $\mathfrak h$ in $\HA(E,\mu)$.  By
definition of atlas isomorphism, $\mathfrak h$ is sent to a specified
record $\mathfrak h'$ in $\HA(E',\mu')$, together with a system of PL
homeomorphisms of the associated terminal patterned balls commuting
strictly with the gluing maps of matched faces.  The disjoint union of
these ball homeomorphisms therefore intertwines every gluing
identification on the nose.  Reassembling the hierarchy in reverse
order consequently yields a homeomorphism $E\to E'$, carrying the
meridian label to the meridian label by (M5).

Here the notation $\HA(E,\mu)$ intentionally retains $\mu$: unlike the
three main objects of the article, the over-marked atlas explicitly
contains the meridian label as part of its reconstruction data.  No
independence of the initial hierarchy choice is asserted or needed;
the atlas isomorphism records the image of every record, and one
compatible pair suffices.

The atlas is intentionally over-marked and reconstruction is
definitional by design.  Lemma~\ref{lem:ball} shows only that an
individual combinatorial boundary-pattern isomorphism is realizable.
What strict compatibility adds is agreement of the ball-by-ball
realizations across shared faces, edges, and vertices; extracting such
agreement from weaker isotopy-level data would require the skeletal
induction indicated above, which we avoid by recording the compatible
system itself.  This is a calibration result, not a substitute for an
Ivanov theorem: the problem is to recover this information from a
natural economical complex, not to assume it as part of the object.
\end{remark}

\section{\texorpdfstring{A recipe: computing $\ES(E)$ for a two-bridge knot}{A recipe: computing ES(E) for a two-bridge knot}}
\label{app:recipe}

This appendix is a worked procedure.  It uses only
Hatcher--Thurston's classification \cite{HatcherThurston} together
with the slope-layer lemma (Lemma~\ref{lem:slope-layer}), and it is
the concrete tool behind the computational projects of
Section~\ref{sec:problems}.  We carry it out completely for the
figure-eight knot; the same steps apply verbatim to any two-bridge
knot, with Step~6 the only genuinely difficult one.

Let $K=K_{p/q}$ be a two-bridge knot in Schubert normal form, with
$q$ odd and $0<p<q$.

\begin{recipe}[Computing $\ES(E)$]
\label{rec:main}
\
\begin{enumerate}[label=\textbf{Step \arabic*.},leftmargin=*,itemsep=4pt]
\item \emph{List the minimal continued-fraction expansions.}  Write
all continued-fraction expansions
\[
  p/q=r+\cfrac{1}{b_1-\cfrac{1}{b_2-\cdots-\cfrac{1}{b_k}}},
  \qquad r\in\mathbb Z,
\]
with $|b_i|\ge2$ for all $i$, and write $[b_1,\dots,b_k]$ for the
expansion.  We use this \emph{subtractive} form because the sign count
of Step~4 then takes its simplest form.  If one prefers the additive
form $r+1/(b'_1+1/(b'_2+\cdots))$, the two are related by
$b'_i=(-1)^{i+1}b_i$, and every sign count below must be applied to
$(-1)^{i+1}b'_i$ rather than to $b'_i$; compare the formulation in
\cite[Section~2]{CurtisEtAl}.  Hatcher--Thurston show that the
essential surfaces in $E(K)$ are exactly those carried by the branched
surfaces $\Sigma(b_1,\dots,b_k)$ associated with these expansions
\cite[Theorem~1]{HatcherThurston}.
\item \emph{Count the carried surfaces.}  For each expansion,
Hatcher--Thurston's Propositions~1--2 count the connected
incompressible, boundary-incompressible surfaces with $n$ sheets.  A
carried surface can be connected only if $n=1$, or $n=2$ and at
least one $b_i$ is odd, or at least two $b_i$ are odd
\cite{HatcherThurston}; see also \cite[Section~2]{CurtisEtAl} and
Theorem~\ref{thm:HT}.  For $n\le2$ these conditions are also
sufficient (Lemma~\ref{lem:carried-basic}); for $n\ge3$ connectedness
depends on the $n_i$.  If at least two of the $b_i$ are odd, connected
$n$-sheeted surfaces exist for every $n$, and the vertex set of
$\ES(E)$ is infinite (Proposition~\ref{prop:two-bridge-finite}); if at
most one $b_i$ is odd, only $n\le2$ occurs.
\item \emph{Discard non-vertices.}  Keep only \emph{orientable}
surfaces: one-sided surfaces are not vertices by
Definition~\ref{def:essential}.  Discard also boundary-parallel
annuli and boundary-parallel tori.  A one-sided surface $N$ still
contributes indirectly: the frontier of a regular neighbourhood
$\partial N(N)$ is an orientable two-sheeted carried surface, and is
a vertex.  In fact every two-sheeted carried surface is orientable
(Lemma~\ref{lem:carried-basic}), also those which are not of this
form; the number of isotopy classes is computed from (HT3) as in
Lemma~\ref{lem:orbits}.
\item \emph{Read off the boundary slopes.}  Set
$b_{\pm}=\#\{i:\pm b_i>0\}$ and $\sigma[b_1,\dots,b_k]=b_{+}-b_{-}$.
With Hatcher--Thurston's conventions all boundary slopes of a
two-bridge knot are even integers, and differences of boundary slopes
are given by
\[
  \Sl[b_1,\dots,b_k]-\Sl[b'_1,\dots,b'_{k'}]
  =2\bigl(\sigma[b_1,\dots,b_k]-\sigma[b'_1,\dots,b'_{k'}]\bigr).
\]
Normalize by declaring the all-even expansion of Step~1, which carries
the orientable Seifert surface, to have slope $0$; every other slope
is then determined.  The absolute formula $\Sl=2\sigma$ is valid only
when the all-even expansion itself has $\sigma=0$, as happens for
$4_1$; it already fails for the trefoil $K_{2/3}$, whose all-even expansion
$[2,2]$ (indeed $1/(2-1/2)=2/3$) has $\sigma=2$ while its Seifert
surface has slope $0$.  Check
this convention, and the mirror-image convention, against the worked
table below before using either version.
\item \emph{Form the slope layers.}  Group the vertices by boundary
slope.  By Lemma~\ref{lem:slope-layer} there are no edges between
different layers, so $\ES(E)$ is already determined up to the internal
structure of each layer.
\item \emph{Decide disjointness inside each layer.}  This is item (B2)
of Problem~\ref{prob:two-bridge-rigidity}.  If every layer is a
singleton, the knot is slope-separated and Step~7 finishes the
computation.  In general, surfaces carried by different expansions are
never disjoint, and two surfaces carried by the same expansion are
disjoint exactly when they are the two components of a carried
surface, which is decided by the sheet rule
(Theorem~\ref{thm:two-bridge-decomposition} and
Lemma~\ref{lem:sheet-rule}).
\item \emph{Apply the visibility theorem.}  If $K$ is
slope-separated, Theorem~\ref{thm:visibility} gives the image and the
kernel of the natural action on all bare and peripheral objects at once:
the image is $\mathbb Z/2$ if $K$ is amphichiral and some slope is
nonzero, and trivial otherwise; the kernel is $\Mod^{+}(E)$ or
$\Mod^{\pm}(E)$ correspondingly.  No mapping class group
computation is required.  Proposition~\ref{prop:aut-slope-separated}
then gives all the combinatorial automorphism groups from the
peripheral slope--type list alone.
\end{enumerate}
\end{recipe}

\begin{figure}[htb]
\centering
\begin{tikzpicture}[
  every node/.style={font=\small},
  box/.style={draw=greyish,rounded corners=2pt,inner sep=4pt,align=center,
              text width=32mm,minimum height=8mm},
  hi/.style={box,draw=surfA!70!black,fill=lightfill},
  hard/.style={box,draw=surfB!80!black,fill=surfB!8},
  ar/.style={-{Stealth[length=4pt]},thick,greyish}]
\node[box] (s1) at (0,0)    {\textbf{1--2.} expansions\\ and carried surfaces};
\node[box] (s3) at (0,-1.6) {\textbf{3.} keep orientable\\ and essential};
\node[box] (s4) at (0,-3.2) {\textbf{4.} slopes\\ $2(b_+-b_-)$};
\node[hi]  (s5) at (5.0,-3.2) {\textbf{5.} slope layers\\ (no cross edges)};
\node[hard](s6) at (5.0,-1.6) {\textbf{6.} disjointness\\ inside a layer\\ \emph{(sheet rule)}};
\node[hi]  (s7) at (5.0,0)    {\textbf{7.} image and kernel\\ by Thm~\ref{thm:visibility}};
\draw[ar] (s1) -- (s3);
\draw[ar] (s3) -- (s4);
\draw[ar] (s4) -- (s5);
\draw[ar] (s5) -- (s6);
\draw[ar] (s6) -- (s7);
\node[greyish,align=flush left,text width=36mm] at (9.5,-1.6)
  {\hyphenpenalty=10000\exhyphenpenalty=10000 If every layer is a singleton, Step 6 is vacuous and the whole
   computation is mechanical.};
\end{tikzpicture}
\caption{Recipe~\ref{rec:main}.  Steps 1--5 are algorithmic; Step 6 is
decided by the sheet rule of Section~\ref{subsec:sheet-rule}; Step 7 is
automatic for slope-separated knots.}
\label{fig:recipe}
\end{figure}

\medskip
\noindent\textbf{The figure-eight knot, done in full.}

Here $4_1=K_{2/5}$, equivalently $K_{3/5}$ since
$3\equiv2^{-1}\pmod5$.  Hatcher--Thurston's example table gives three
minimal expansions.  In the subtractive convention of Step~1 they
represent $3/5$:
\[
  \tfrac35=0+\frac1{2-\frac13}=1+\frac1{-2-\frac12}
  =1+\frac1{-3+\frac12},
\]
for $[2,3]$, $[-2,2]$ and $[-3,-2]$ respectively.  (In the additive
convention the same three symbols would represent $3/7$, $1/3$ and
$5/7$, which is why the convention must be fixed.)  The recipe produces
the following.  The all-even
expansion is $[-2,2]$, with $\sigma=0$, so here the absolute formula
$\Sl=2\sigma$ is valid.

\begin{center}
\small
\renewcommand{\arraystretch}{1.25}
\begin{tabular}{@{}lcccccl@{}}
\toprule
expansion & $b_+$ & $b_-$ & $\Sl$ (here $=2\sigma$) & sheets & orientable?
 & vertex \\
\midrule
$[-2,2]$  & $1$ & $1$ & $0$  & $1$ & yes & $F$: once-punctured torus\\
$[2,3]$   & $2$ & $0$ & $+4$ & $1$ & no  & (one-sided $N_+$)\\
$[2,3]$   & $2$ & $0$ & $+4$ & $2$ & yes & $S_+=\partial N(N_+)$\\
$[-3,-2]$ & $0$ & $2$ & $-4$ & $1$ & no  & (one-sided $N_-$)\\
$[-3,-2]$ & $0$ & $2$ & $-4$ & $2$ & yes & $S_-=\partial N(N_-)$\\
\bottomrule
\end{tabular}
\end{center}

\noindent
Step~3 removes the two one-sided once-punctured Klein bottles $N_\pm$
and keeps their doubles.  Step~5 produces three singleton layers, at
slopes $0,+4,-4$; none of these is the meridian, and no closed
essential surface occurs, so Step~6 is vacuous and $4_1$ is
slope-separated.
Step~7 gives, without computing $\Mod^{\pm}(E_8)$:
\[
  \Image\Phi_X\cong\mathbb Z/2,
  \qquad
  \mathcal K_X(E_8)=\Mod^{+}(E_8),
\]
because $4_1$ is amphichiral and $S=\{0,\pm4\}\ne\{0\}$.  Feeding in
the classical fact that $\Mod^{\pm}(E_8)$ is dihedral of order
eight (Lemma~\ref{lem:figure-eight-action}) identifies the kernel as
$(\mathbb Z/2)^2$, which is Corollary~\ref{cor:fig8-final}.
Proposition~\ref{prop:aut-slope-separated} now reads directly from the peripheral slope--type list
\[
  \bigl\{(0,(1,1)),\ (+4,(1,2)),\ (-4,(1,2))\bigr\}.
\]
It gives $S_3$ for the bare complex and $\mathbb Z/2$ for the peripheral refinement; the type-only and slope-only diagnostic views also have automorphism group $\mathbb Z/2$.

\medskip
\noindent\textbf{What to try next.}

Section~\ref{sec:two-bridge} carries out the recipe for every
two-bridge knot with finitely many essential surfaces: the torus knots
$T(2,q)$ and the twist knots, for which Step~6 is vacuous
(Theorem~\ref{thm:two-bridge-slope-separated}), and the knots of
Class~II (Theorem~\ref{thm:class-two}), beginning with $6_2$, $7_3$
and $7_4$.  For the remaining two-bridge knots, beginning with $6_3$
and $7_5$, the vertex set is infinite; the complex is still computable
by the decomposition theorem and the sheet rule, its structure is
described in Theorem~\ref{thm:general-structure}, and the remaining
questions are collected in Remark~\ref{rem:two-bridge-remaining}.

\end{document}